\documentclass[11pt]{amsart}
\usepackage{amsmath,amsthm,amscd,amssymb}
\usepackage{pb-diagram}
\usepackage{graphicx}

\usepackage{xcolor}
\usepackage{amssymb}
\newcommand{\Z}{\mathbb{Z}}
\newcommand{\sP}{\scriptscriptstyle{\mathcal{P}}}
\newcommand{\K}{\mathcal{K}}

\newcommand{\ov}{\overline}
\newcommand{\ra}{\rightarrow}
\newcommand{\lra}{\longrightarrow}

\newcommand{\Hom}{\operatorname{Hom}}

\newtheorem{thm}{Theorem}[section]
\newtheorem{lem}[thm]{Lemma}
\newtheorem{prop}[thm]{Proposition}
\newtheorem{cor}[thm]{Corollary}

\theoremstyle{definition}
\newtheorem{ex}[thm]{Example}
\newtheorem{defn}[thm]{Definition}

\newtheorem{rem}[thm]{Remark}

\numberwithin{equation}{section}
\numberwithin{figure}{section}

\newcommand{\G}{\Gamma}
\newcommand{\La}{\Lambda}
\newcommand{\Q}{\mathbb{Q}}
\newcommand{\R}{\mathcal{R}}

\usepackage{pstricks}
\usepackage{pst-node}

\newcommand{\sss}{\scriptscriptstyle}

\newcommand{\gn}{\ensuremath{G^{\sss (n)}}}
\newcommand{\gnp}{\ensuremath{G^{\sss (n+1)}}}

\newcommand{\an}{\ensuremath{A^{\sss (n)}}}
\newcommand{\bn}{\ensuremath{B^{\sss (n)}}}

\newcommand{\bnp}{\ensuremath{B^{\sss (n+1)}}}

\title{Primary Decomposition and the Fractal Nature of Knot Concordance}

\author{Tim D. Cochran$^{\dag}$}
\address{Department of Mathematics MS-136, P.O. Box 1892, Rice University, Houston, TX 77251-1892}
\email{cochran@rice.edu}

\author{Shelly Harvey$^{\dag\dag}$}
\address{Department of Mathematics MS-136, P.O. Box 1892, Rice University, Houston, TX 77251-1892}
\email{shelly@rice.edu}

\author{Constance Leidy$^{\dag\dag\dag}$}
\address{Department of Mathematics, Wesleyan University, Wesleyan Station, Middletown, CT 06459}
\email{cleidy@wesleyan.edu}

\thanks{$^{\dag}$Partially supported by the National Science Foundation  DMS-0706929}
\thanks{ $^{\dag\dag}$Partially supported by NSF CAREER DMS-0748458 and The Alfred P. Sloan Foundation}
\thanks{$^{\dag\dag\dag}$Partially supported by NSF DMS-0805867 }

\subjclass[2000]{Primary $57$M$25$; Secondary $20$J$05$}

\begin{document}
%\date{\today}
\begin{abstract} For each sequence $\mathcal{P}=(p_1(t),p_2(t),\dots)$ of polynomials we define a characteristic series of groups, called the \emph{derived series localized at $\mathcal{P}$}. These group series yield filtrations of the knot concordance group that refine the $(n)$-solvable filtration. We show that the quotients of successive terms of these refined filtrations have infinite rank. The new filtrations allow us to distinguish between knots whose classical Alexander polynomials are coprime and even to distinguish between knots with coprime higher-order Alexander polynomials. This provides evidence of higher-order analogues of the classical $p(t)$-primary decomposition of the algebraic concordance group. We use these techniques to give evidence that the set of smooth concordance classes of knots is a fractal set.
\end{abstract}

\maketitle

\section{Introduction}\label{intro}

A (classical) \textbf{knot} $K$ is the image of a tame embedding of an oriented circle in $S^3$. Two knots, $K_0\hookrightarrow S^3\times \{0\}$ and $K_1\hookrightarrow S^3\times \{1\}$, are \textbf{concordant} if there exists a proper smooth embedding of an annulus into $S^3\times [0,1]$ that restricts to the knots on $S^3\times \{0,1\}$.  The equivalence relation of concordance first arose in the early $1960's$ in work of Fox, Kervaire and Milnor  in their study of isolated singularities of $2$-spheres in $4$-manifolds and indeed, certain concordance problems are known to be \emph{equivalent} to whether higher-dimensional surgery techniques ``work'' for topological $4$-manifolds ~\cite{FM,KM1,CF}. Let $\mathcal{K}$ be the set of ambient isotopy classes of knots and let $\mathcal{C}$ denote the set of (smooth) concordance classes of knots. Since isotopic knots are concordant, there is a natural surjection $\mathcal{K}\to \mathcal{C}$. Furthermore, it is known that the connected sum operation endows $\mathcal{C}$ with the structure of an abelian group, called the \textbf{smooth knot concordance group}. The identity element is the class of the trivial knot. Any knot in this class is concordant to a trivial knot and is called a \textbf{slice knot}. Equivalently, a slice knot is one that is the boundary of a smooth embedding of a $2$-disk in $B^4$. In the late $60's$ Milnor and Tristram showed that this group has infinite rank. It is also known to contain an infinite linearly independent set of elements of order two. Much work has been done on the subject of knot concordance (for excellent surveys see ~\cite{Go1,Li1}). In particular,  ~\cite{COT} introduced a natural filtration of $\mathcal{C}$ by subgroups
$$
\cdots \subset \mathcal{F}_{n+1} \subset \mathcal{F}_{n.5}\subset\mathcal{F}_{n}\subset\cdots \subset
\mathcal{F}_1\subset \mathcal{F}_{0.5} \subset \mathcal{F}_{0} \subset \mathcal{C}.
$$
called the \textbf{($n$)-solvable filtration} of $\mathcal{C}$ and denoted $\{\mathcal{F}_{n}\}$ (defined in Section~\ref{sec:series}). The filtration is significant due to its intimate connection to the work of A. Casson and M. Freedman on the topological classification problem for $4$-manifolds; but also because the classical knot concordance invariants are neatly encapsulated in the low-order terms. The non-triviality of $\mathcal{C}$ can be measured in terms of the associated graded abelian groups $\{\mathbb{G}_n=\mathcal{F}_{n}/\mathcal{F}_{n.5}~|~n\in\mathbb{N}\}$. We ignore the other ``half'' of the filtration, $\mathcal{F}_{n.5}/\mathcal{F}_{n+1}$, where almost nothing is known!

The first term, $\mathbb{G}_0$, is essentially Levine's \emph{algebraic knot concordance group}. That it has infinite rank is the aforementioned result of Milnor and Tristram. In fact the work of Milnor, Levine, and Stolzfus in the $1960$'s resulted in a complete classification:
$$
\mathbb{G}_{0}/\text{torsion}\cong\Z^\infty\cong\bigoplus_{\substack{p(t)}}\Z^{r_p}
$$
where the sum is over all irreducible $p(t)\in \Z[t]$ with $p(1)=\pm 1$, $p(t^{-1})\doteq p(t)$, and where $r_p$ is the number of distinct pairs $(z,\overline{z})$ of unit norm complex roots of $p(t)$ ~\cite[Section 11-20]{Le10}\cite[p.120]{Hi}\cite[Proposition 3.2]{Cha2}~\cite{Sto}.  Said differently, the algebraic knot concordance group decomposes into the direct sum of its $p(t)$-primary parts, each of which, modulo torsion, is isomorphic to $\Z^{r_p}$, as detected by a \textbf{Milnor signature} associated to the pair $(z,\overline{z})$ ~\cite{M3}. Indeed, the first step in the classification of $\mathbb{G}_{0}$ is to decompose the Alexander module of $K$ (together with its Blanchfield form) into its primary parts. In this decomposition, a given knot has a non-zero $p(t)$-component only if $p(t)$ is a factor of its Alexander polynomial. This relies heavily on the fact that $\Z[t,t^{-1}]$ is a unique factorization domain. The reader should be aware that the Alexander polynomial of a knot is not itself invariant under concordance, but there are signature invariants that are associated to its roots (as above).

The primary goal of the present paper is to suggest, and give evidence for, an analogous  but much more intricate decomposition for each $\mathbb{G}_n$. The full decomposition will necessarily be much more complicated since certain so-called higher-order Alexander modules are relevant, and these are modules over noncommutative rings that are not unique factorization domains.
We herein define, to each sequence of polynomials $\mathcal{P}=(p_1(t),...,p_n(t))$, a new filtration of $\mathcal{C}$, denoted $\{\mathcal{F}_n^\mathcal{P}\}$, such that $\mathcal{F}_n\subset \mathcal{F}_{n}^\mathcal{P}$ (Section~\ref{sec:series}). Then we consider the product of all quotient maps
$$
\mathbb{G}_n/\text{torsion}\to \prod_{\substack{\mathbb{P}_n}} \frac{\mathcal{F}_n}{\mathcal{F}_{n.5}^\mathcal{P}\cap \mathcal{F}_n}
$$
where the product is taken over the set $\mathbb{P}_n$ of all ``distinct'' sequences $\mathcal{P}$. In this scheme, $p_1(t)$ should be thought of as a prime factor of the classical Alexander polynomial and the other $p_i$ are related to higher-order Alexander polynomials. The quotient corresponding to $\mathcal{P}$ should be thought of as localizing at $\mathcal{P}$ in the sense that, loosely speaking, knots whose higher-order Alexander polynomials are coprime to $\mathcal{P}$ will vanish in this quotient. We conjecture that the image of this map is the direct \emph{sum} over $\mathbb{P}_n$. This then gives the broad outlines of our proposed ``primary decomposition'' of $\mathbb{G}_n$. As evidence for this we produce, for each $\mathcal{P}$, an infinite rank summand $\Z^\infty\subset \mathbb{G}_n$ such that the composition
$$
\bigoplus_{\substack{\mathbb{P}_n}}\Z^\infty\subset \mathbb{G}_n\overset{\psi_\mathcal{P}}\twoheadrightarrow \frac{\mathcal{F}_n}{\mathcal{F}_{n.5}^\mathcal{P}\cap \mathcal{F}_n}
$$
is injective on the $\Z^\infty\subset \mathbb{G}_n$ summand corresponding to $\mathcal{P}$ and is zero on all other summands. We note that, in a subsequent paper, we exhibit this same structure for $2$-torsion elements of $\mathbb{G}_n$ ~\cite{CHL6}.

Having outlined our primary goal, we will now review some further history related to primary decomposition, and reformulate our results in a more concrete form.

In the $1970$'s the introduction of Casson-Gordon invariants in ~\cite{CG1,CG2} led to the result \cite{Ji1}:
$$
\Z^\infty\subset \mathbb{G}_{1}.
$$
A specific family of knots realizing such a $\Z^\infty$ is shown on the right-hand side of Figure ~\ref{fig:examplesofdifferentG1}. These knots result from starting with the $9_{46}$ ribbon knot, denoted $R^1$, shown on the left-hand side of Figure ~\ref{fig:examplesofdifferentG1} (here $-1$ means one full negative twist); then modifying it by tying its central bands into the shape of an auxiliary (Arf invariant zero) knot $J$. The resulting knot is denoted $R^1(J)$. If $J$ varies over a set of knots whose Milnor signatures are ``independent'' then $\{R^1(J)\}$ will be linearly independent in $\mathbb{G}_{1}$.
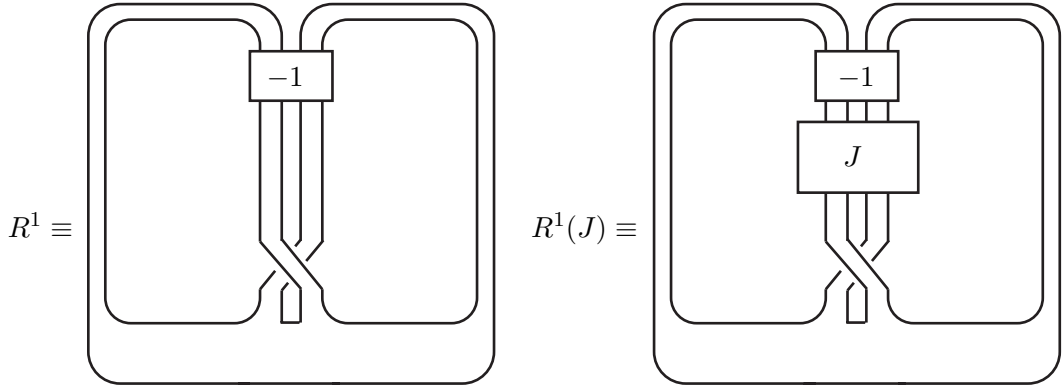
\begin{figure}[htbp]
\setlength{\unitlength}{1pt}
\begin{picture}(327,151)
\put(-20,0){\includegraphics{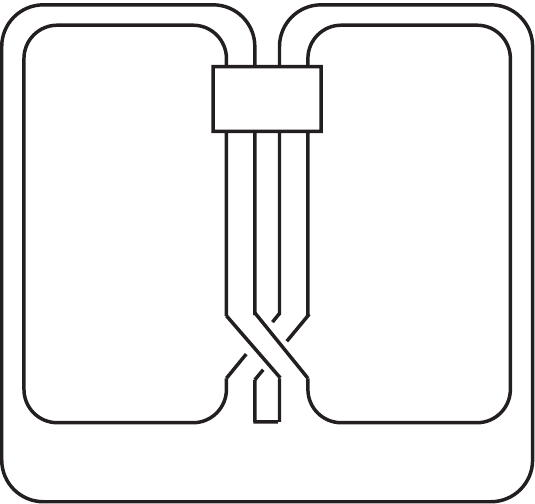}}
\put(194,0){\includegraphics{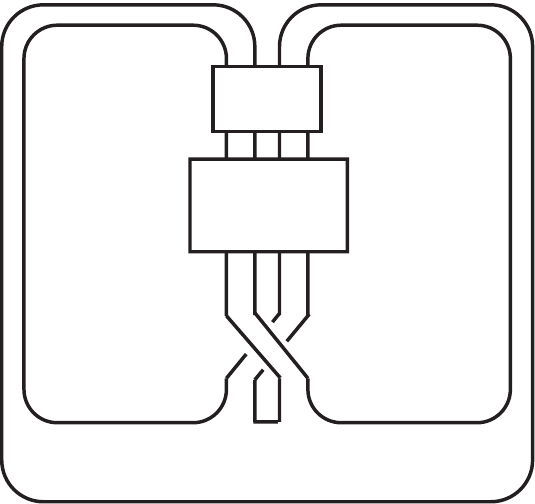}}
\put(48,113){$-1$}
\put(84,78){}
\put(266,83){$J$}
\put(264,113){$-1$}
\put(148,56){$R^{1}(J)\equiv$}
\put(-50,56){$R^{1}\equiv$}
\end{picture}
\caption{Family of knots, $R^1(J)$, in $\mathcal{G}_1$ distinguished by classical signatures of $J$}\label{fig:examplesofdifferentG1}
\end{figure}
But more recent work of Se-Goo Kim (on the ``$p$-primary splitting'' of Casson-Gordon invariants)~\cite{KiS}; and by Se-Goo Kim and Taehee Kim (using metabelian $L^{(2)}$-signatures)  ~\cite{KiKi} leads to a refinement analogous to that of Levine:
$$
\bigoplus_{\substack{p(t)}}\Z^\infty\subset\mathbb{G}_{1}.
$$
A family of knots that realizes such a subgroup is shown in Figure~\ref{fig:examplesofdifferentG1Kim}. Here the base ribbon knot, denoted $R^k$, $k>0$, is allowed to vary ($-k$ denotes $k$ full negative twists). The Alexander polynomial of $R^k$ is $p_1(t)p_1(t^{-1})$, where $p_1(t)=kt-(k+1)$, and these are coprime for different values of $k$. The result is a two parameter family, $R^{k}(J)$, distinguished up to concordance not only by the signatures of $J$ but also by the Alexander polynomials of the $R^k$.
\begin{figure}[htbp]
\setlength{\unitlength}{1pt}
\begin{picture}(327,151)
\put(84,0){\includegraphics{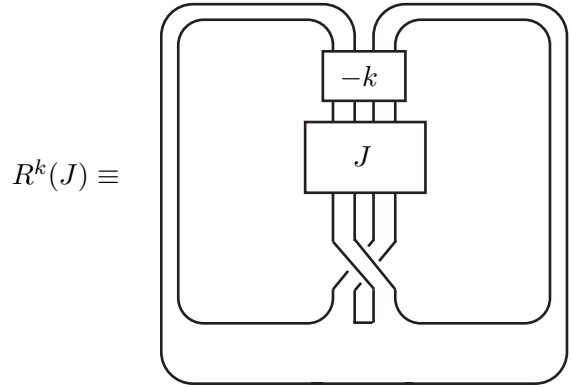}}
\put(157,83){$J$}
\put(152,113){$-k$}
\put(028,76){$R^{k}(J)\equiv$}
\end{picture}
\caption{A $2$-parameter family in $\mathcal{G}_1$, distinguished by the signatures of $J$ and the Alexander polynomial of $R^k$}\label{fig:examplesofdifferentG1Kim}
\end{figure}
This (and other results of ~\cite{KiKi}) gives strong evidence for the existence of a primary decomposition of $\mathbb{G}_{1}$.  These same authors have announced some partial results about the polynomial splitting of certain higher-order $L^{(2)}$-signatures.  However, in all these results the knots are once again distinguished only by their \emph{classical} Alexander polynomials.

Now we consider the possibility of a decomposition of $\mathbb{G}_n$ for $n>1$. A number of papers have addressed the  non-triviality of $\{\mathbb{G}_{n}\}$, \cite{Gi3,GL2,Ki1,Fr2,COT2,CT}, culminating in ~\cite{CHL3} where it was shown that $\mathbb{G}_n$ has infinite rank for any integer $n$, that is, it was shown that there exists
\begin{equation}\label{eq:ourthm}
\Z^\infty\subset \mathbb{G}_{n}.
\end{equation}
Our present work enables us to prove a substantial generalization of this fact, along the lines of the Levine-Milnor-Stoltzfus primary decomposition of $\mathbb{G}_0$ and in the spirit of the Kim-Kim work on $\mathbb{G}_1$. We prove that, for each prime $p_1(t)$ (that can occur as a divisor of the Alexander polynomial of a knot) there is a distinct infinite rank subgroup of $\mathbb{G}_n$ all of whose classical Alexander modules are cyclic of order $p_1(t)p_1(t^{-1})$. But we go much farther. We show that each of \emph{these} subgroups (consisting of knots with the same Alexander polynomial) decomposes into (infinite rank) subgroups whose members are distinguished by the orders of their \emph{second higher-order Alexander modules}; et cetera.

To see how this works in a specific case, fix $n=2$. We first describe an infinite family realizing ~(\ref{eq:ourthm}) in the case $n=2$. The left-hand side of Figure~\ref{fig:examplesofinfgen} shows a ribbon knot, $R^1$ with the same Alexander polynomial as $9_{46}$ (here $T_1$ is a certain fixed knot which is not relevant to this overview). Now, for any knot $J$ we can form $R^1_\alpha(J)$ and then define the knot $R^{1}_\alpha(R^1_\alpha(J))$ as shown on the right-hand side of Figure~\ref{fig:examplesofinfgen}. Then varying $J$ over any collection of Arf invariant zero knots with independent signatures yields an infinite family generating $\Z^{\infty}\subset \mathbb{G}_2$ \cite{CHL3}. All of these knots have the same classical Alexander polynomial, that of $R^1$. They are distinguished by the classical signatures of $J$.
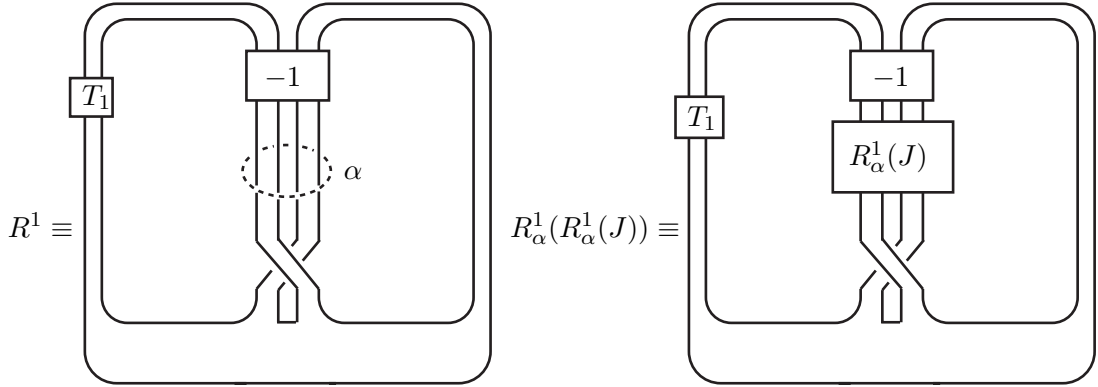
\begin{figure}[htbp]
\setlength{\unitlength}{1pt}
\begin{picture}(327,151)
\put(-25,0){\includegraphics{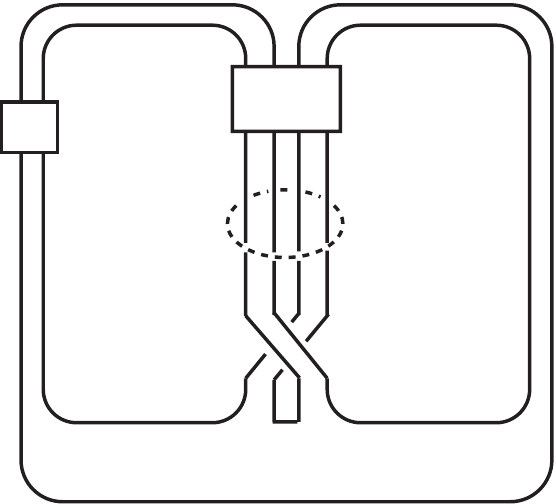}}
\put(204,0){\includegraphics{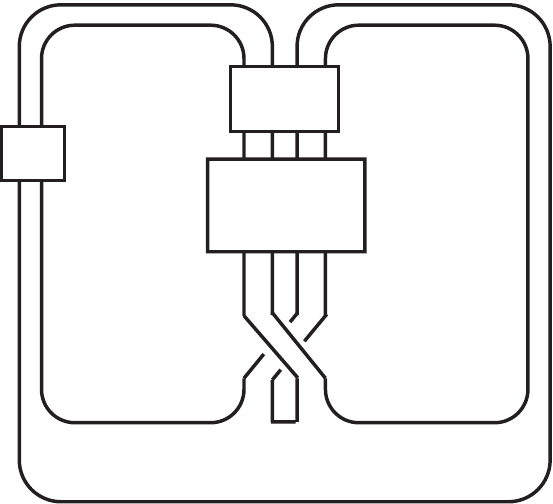}}
\put(49,113){$-1$}
\put(79,78){$\alpha$}
\put(270,83){$R^1_\alpha(J)$}
\put(279,113){$-1$}
\put(209,98){$T_1$}
\put(142,56){$R^{1}_\alpha(R^1_\alpha(J))\equiv$}
\put(-20,106){$T_1$}
\put(-48,56){$R^{1}\equiv$}
\end{picture}
\caption{A family, $R^{1}_\alpha(R^1_\alpha(J))$, in $\mathbb{G}_2$ distinguished by the classical signatures of $J$}\label{fig:examplesofinfgen}
\end{figure}

But now we can consider the family of ribbon knots $R^k$ as shown in Figure~\ref{fig:examplesofdifferent}, which, for varying $k$ have coprime Alexander polynomials.
\begin{figure}[htbp]
\setlength{\unitlength}{1pt}
\begin{picture}(327,151)
\put(80,0){\includegraphics{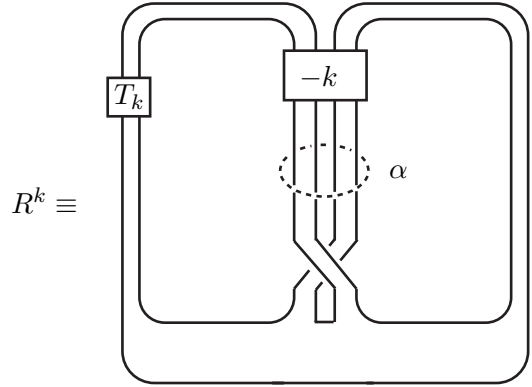}}
\put(153,113){$-k$}
\put(187,78){$\alpha$}
\put(83,106){$T_k$}
\put(44,65){$R^{k}\equiv$}
\end{picture}
\caption{Family of ribbon knots with different Alexander polynomials}\label{fig:examplesofdifferent}
\end{figure}
Then by utilizing $R^k$ and $R^m$ we can form a two parameter family of knots, $R^{k}_\alpha(R^m_\alpha(J))$, as shown in Figure~\ref{fig:examplesofdifferent2}.
\begin{figure}[htbp]
\setlength{\unitlength}{1pt}
\begin{picture}(327,151)
\put(94,0){\includegraphics{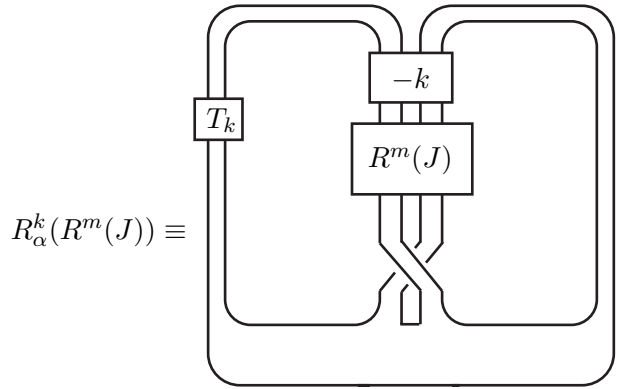}}
\put(160,83){$R^m(J)$}
\put(169,113){$-k$}
\put(99,98){$T_k$}
\put(25,56){$R^{k}_\alpha(R^m(J))\equiv$}
\end{picture}
\caption{Three parameter family of knots $R^{k}_\alpha(R^m(J))\in \mathbb{G}_2$, distinguished by torsion in first and second Alexander modules, and the signatures of $J$}\label{fig:examplesofdifferent2}
\end{figure}
Then, upon choosing a set of $J$ with independent signatures, we will show that this yields a $2$-parameter collection of $\Z^\infty$ subgroups of $\mathbb{G}_2$. For fixed $J$ and different values of $k$ these are distinguished by their classical Alexander polynomials. For $J$ and $k$ fixed, these all have isomorphic classical Alexander modules but, we claim, are distinguished by torsion in their  \emph{second} higher-order  Alexander modules.

To be more specific, recall that the higher-order Alexander modules of a knot have a purely group theoretic description as the quotient, $G^{(i)}/G^{(i+1)}$, of successive terms in the derived series of the fundamental group of the zero framed surgery on the knot ~\cite{C,COT}.  Each of these is known to be a torsion module over the ring $\Z[G/G^{(i)}]$ ~\cite[Section 2]{COT}. Taking $i=1$ gives the classical Alexander module. A Mayer-Vietoris argument shows that the second order Alexander module (taking $i=2$) of $R^{k}_\alpha(J')$, for any $J'$, contains a summand related to the classical Alexander module of $J'$ ~\cite[Thm 8.2]{C}~\cite[Lemma 5.10]{CHL6}. Specifically in the case of the knots of Figure~\ref{fig:examplesofdifferent2}, the Alexander module of $J'=R^{m}_\alpha(J)$ is cyclic of order
$$
p_2(t)p_2(t^{-1})=\left(mt-(m+1)\right)\left(mt^{-1}-(m+1)\right),
$$
so the second order Alexander module of $R^{k}_\alpha(R^{m}_\alpha(J))$ contains a submodule of the form
$$
\frac{\Z\G}{p_{2}(x)p_{2}(x^{-1})\Z\G}
$$
for $\G=G/G^{(2)}$  and for some $x\in G^{(1)}/G^{(2)}$.
In summary, the examples of Figure~\ref{fig:examplesofdifferent2} give a $3$-parameter family (varying $k$, $m$ and signatures of $J$) of $2$-solvable knots that are linearly independent modulo $\mathcal{F}_{2.5}$. But what the reader should focus on is that they are distinguished by the orders of elements in their first and second order Alexander modules. That these knots are distinct up to isotopy is obvious. What is difficult is to show that this difference persists in $\mathcal{C}$.

Moreover this pattern continues. Consider a set $\mathbb{P}_n=\{\mathcal{P}\}$ of all ``distinct'' $n$-tuples $\mathcal{P}=(p_1(t),...,p_n(t))$ of polynomials with $p_i(1)=\pm 1$ (for the definition of distinct see Definitions~\ref{def:stronglycoprime} and ~\ref{def:distinctP}). We then prove in Theorem~\ref{thm:fractal} that for each $\mathcal{P}\in\mathbb{P}_n$ there is a distinct $\Z^\infty\subset \mathcal{F}_n$, yielding a subgroup
\begin{equation}\label{eq:bigsubgroup}
\bigoplus_{\substack{\mathbb{P}_n}}\Z^\infty\subset \mathbb{G}_n.
\end{equation}
The polynomial $p_1$ relates to the order of torsion in the classical, or first, Alexander module, while the higher $p_i$ relate to the type of torsion in the $i^{th}$ higher-order Alexander module, as in the examples above. Thus we show that one can distinguish concordance classes of knots not only by their classical Alexander polynomials, but also, loosely speaking, by their higher-order Alexander polynomials.

We briefly explain our strategy to distinguish knots with different torsion in their higher-order Alexander modules, since it motivates several chapters of new mathematics we must create. The presence of a certain type of $\Z[G/G^{(n)}]$-torsion in the module $G^{(n)}/G^{(n+1)}$ can potentially be detected by localization. The process of localization of modules, when such a process exists, serves to kill torsion, depending on what subset of $\Z[G/G^{(n)}]$ is inverted. Although, there is no good notion of localizing a ring or module at a prime ideal in a general noncommutative ring, there is a classical theory of localization over Ore domains. Therefore the project of distinguishing among knots with different torsion in higher order Alexander modules hinges critically on our ability to accurately specify new families of right divisor sets in the Ore domain $\Z[G/G^{(n)}]$. Moreover, taking advantage of a philosophy introduced by the second author, we show that choosing different localizations for each $n$ leads to different group series, each an enlargement of the derived series (Section~\ref{sec:seriesfromloc}). Specifically, we distinguish among knots with different torsion in their higher-order Alexander modules by defining, for each sequence $\mathcal{P}$, a characteristic series of groups, $\{G^{(n)}_\mathcal{P}\}$, that we call the derived series localized at $\mathcal{P}$ (Section~\ref{sec:locprimes}). Then to any knot we can associate a $3$-manifold, $M_K$, the zero framed surgery on $K$ and a
coefficient system
$$
\phi:\pi_1(M_K)\equiv G\to G/G^{(n+1)}_\mathcal{P}.
$$
To this we associate the real-valued von Neumann $\rho$-invariant, $\rho(M_K,\phi)$ (Section~\ref{sec:rhoinvs}). This invariant is shown to vanish trivially on all knots \emph{except} those whose higher-order Alexander modules have the torsion characteristics to match $\mathcal{P}$. Finally, the higher-order signature invariants are shown to obstruct membership in $\{\mathcal{F}_n^\mathcal{P}\}$ (Section~\ref{sec:rhoinvs}).

This paper has a secondary goal (but related to the primary goal). Recall that a \emph{fractal set} $\mathcal{F}$ is one that admits \emph{self-similarity structures}, by which is meant merely a system of injective maps $\phi_i:\mathcal{F}\hookrightarrow \mathcal{F}$
~\cite[Def. 3.1]{BGN}. Of course any infinite set has many proper self-embeddings, so one normally expects self-similarities to be in some sense natural with respect to other structure that may exist on the set. $\mathcal{K}$ itself has many well-known self-similarity structures given by classical \emph{satellite constructions}. If $R$ is a knot and $\alpha$ is a simple closed curve in $S^3-R$ that bounds an embedded disk in $S^3$ as shown on the left-hand side of Figure~\ref{fig:satellite}, then $(R,\alpha)$ parametrizes an operator $R_\alpha :\mathcal{K}\to \mathcal{K}$ wherein $R_{\alpha}(K)$ is obtained from $R$ by tying all the strands of $R$ that pass through $\alpha$ into parallel copies of the knot $K$ as indicated schematically on the right-hand of Figure~\ref{fig:satellite}.
\begin{figure}[htbp]
\setlength{\unitlength}{1pt}
\begin{picture}(162,111)
\put(21,3){$R$}
\put(-14,52){$\alpha$}
\put(96,3){$R_{\alpha}(K)$}
\put(0,20){\includegraphics{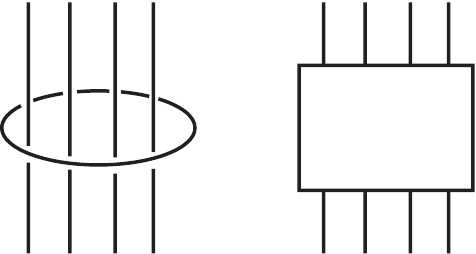}}
\put(105,53){$K$}
\end{picture}
\caption{$R_{\alpha}(K)$}\label{fig:satellite}
\end{figure}
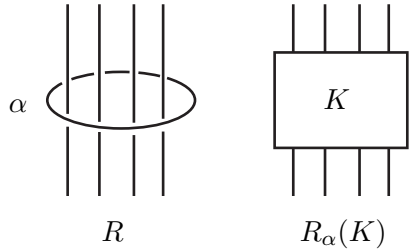
As long as this operator is non-trivial (in the sense that $\alpha$ does not bound an embedded disk in $S^3-R$) it is known to be injective (by the uniqueness of the torus decomposition of $3$-manifolds). Moreover, the number of distinct such embeddings (varying $R$ and $\alpha$) is infinite, with the parameter space containing many natural independent parameters. For example, the Alexander polynomial of $R$ is one such parameter, as is the integer $n$ that represents the maximum depth of $\alpha$ in the derived series of $\pi_1(S^3-R)$. Any two operators with differing values of these parameters \emph{necessarily} are different operators (again by the uniqueness of the torus decomposition). Since isotopic knots are concordant, there is a natural surjective map $\mathcal{K}\to \mathcal{C}$. It is well known that each of these operators descends to give an operator (not a homomorphism)
$$
R_\alpha :\mathcal{C}\to \mathcal{C}.
$$
We conjecture (and present evidence) that many of these self-similarity structures on $\mathcal{K}$ descend to self-similarity structures on $\mathcal{C}$, that is, we conjecture that many of the operators
$R_\alpha$ on $\mathcal{C}$ are injective. For example, if $\alpha$ is a meridional loop for $R$ then $R_\alpha([K])=[R\# K]=[R]+[K]$, so $R_\alpha$ is injective. The classical operation of ``Whitehead doubling'' is another particular example of such an operation that is conjectured to be injective on $\mathcal{C}$.

We do not know if fractal structures on a set (that happens to be an abelian group) can be useful in understanding that set. But we profit from considering mathematical structures on $\mathcal{C}$ \emph{other} than its \emph{group} structure. Certainly topologists \emph{are} interested in more than the group structure of $\mathcal{C}$. In particular, they are interested in how knot concordance behaves with respect to natural geometric operations such as satellite and cabling operations that are \emph{known} to not induce homomorphisms. Furthermore there are many interesting questions such as: Is $R_\alpha$ continuous with respect to the topology on $\mathcal{C}$ induced by the $n$-solvable filtration? Is there a good metric topology on $\mathcal{C}$?

We remark that there is another interesting question (related to injectivity): Does $R_\alpha$ send large linearly independent sets to linearly independent sets? Since these operators are not homomorphisms, this question is logically independent of the question of injectivity! In the context of this paper we address both questions.

As evidence for the existence of self-similarities we show that certain of these operators, that we call \emph{robust doubling operators} (see Definitions~\ref{def:doublingop} and ~\ref{def:robust}), are injective on the subgroup consisting of essentially all known non-trivial examples (modulo torsion). Moreover we conjecture that the number of distinct robust operators (varying $R$ and $\alpha$) is infinite, with the parameter space containing natural independent parameters. This would mean that $\mathcal{C}$ embeds in itself in many distinct ways, all with disjoint images. In particular the Alexander polynomial, $p(t)$, of the knot $R$ is a natural parameter. We prove not only that each such robust $R_\alpha$ is injective on  the subgroup consisting of essentially all known examples, but that if $R_\alpha$ and $R'_\beta$ are such that $R$ and $R'$ have Alexander polynomials that are coprime, then $R_\alpha$ and $R'_\beta$ have \emph{disjoint images} on this subgroup!

\newtheorem*{thm:injectivity}{Theorem~\ref{thm:injectivity}}
\begin{thm:injectivity} If $R_\alpha$ is a robust operator then $R_\alpha:\mathcal{C}\to \mathcal{C}$ is injective on the subgroup
$$
\bigoplus_{\substack{n}}\bigoplus_{\substack{\mathbb{P}_n}}\Z^\infty\subset \mathcal{C}
$$
from ~\eqref{eq:bigsubgroup}. Moreover, if $R_\alpha$ and $R'_\beta$ are robust doubling operators for which the Alexander polynomials of $R$ and $R'$ are coprime, then $R_\alpha$ and $R'_\beta$ have disjoint images (i.e. intersecting only in $\{0\}$), when restricted to this subgroup. Furthermore, the composition $\mathcal{C}\overset{R_{\alpha}}{\longrightarrow}\mathcal{C}\to\mathcal{C}/\mathcal{F}_{n.5}$
is injective on the subgroup
$$
\bigoplus_{\substack{\mathbb{P}_n}}\Z^\infty\subset \mathcal{F}_n\subset\mathcal{C}
$$
from ~\eqref{eq:bigsubgroup}.
\end{thm:injectivity}

To view this evidence for the existence of many distinct self-similarities diagrammatically, for each knot polynomial $p_k(t)=\delta(t)\delta(t^{-1})$ with $\delta$ prime, we define a robust operator (abbreviated here as) $p_k:\mathcal{C}\to \mathcal{C}$ with the property that $p_k(\mathbb{G}_j)\subset \mathbb{G}_{j+1}$ (Example~\ref{ex:robustoperators2}). In fact the operators $R^k_\alpha$ in Figure~\ref{fig:examplesofdifferent} are the examples where $\delta_k(t)$ is linear, so the reader can focus on those. Since there are countably infinitely many such polynomials $p_k$, these maps and their compositions are parametrized by an infinite tree with countably infinite valence at each vertex, as indicated in Diagram~\ref{diagram:fractal}.
$$
\begin{diagram}\label{diagram:fractal}\dgARROWLENGTH=2.0em
\node{\mathcal{C}} \arrow{se,t}{p_1}\\ \node{\vdots}\node{\mathcal{C}}\arrow{sse,t}{p_1}\\ \node{\mathcal{C}} \arrow{ne,b}{p_k} \\ \node[2]{\vdots}\node{\mathcal{C}}\\ \node{\mathcal{C}} \arrow{se,t}{p_1}\\ \node{\vdots}\node{\mathcal{C}}\arrow{nne,t}{p_k}\\ \node{\mathcal{C}} \arrow{ne,b}{p_k}
\end{diagram}
$$

\noindent The set $\mathbb{P}_n$ parametrizes compositions of $n$ of these operators (terminating in the right-most copy of $\mathcal{C}$) as suggested by the diagram. Our conjectures would imply that each such composition of length $n$ (that is, $p_{i_n}\circ \dots\circ p_{i_1}$) that terminates at the right-most copy of $\mathcal{C}$ is an embedding and moreover that these compositions have images that intersect only in the class of the trivial knot. As evidence (re-wording Theorem~\ref{thm:injectivity}) we exhibit infinite linearly independent subsets of $\mathbb{G}_0$ the union of whose images under these compositions are linearly independent in $\mathbb{G}_n$, giving the very large subgroup of ~\eqref{eq:bigsubgroup}, as suggested by the diagram below.

$$
\begin{diagram}\label{diagram:fractal3}\dgARROWLENGTH=1.5em
\node{\mathbb{G}_{0}\to} \node{\dots\mathbb{G}_{n-2}} \arrow{se,t}{p_1}\\ \node{\mathbb{G}_{0}\to}\node{\vdots}\node{\mathbb{G}_{n-1}}\arrow{sse,t}{p_1}\\ \node{\mathbb{G}_{0}\to  }\node{\dots\mathbb{G}_{n-2}} \arrow{ne,b}{p_k} \\ \node[3]{\vdots}\node{\mathbb{G}_n}\node{\!\!\!\!\!\!\!\!\!\supset{\oplus}_{\mathbb{P}_n}~\Z^\infty}\\ \node{\mathbb{G}_{0}\to} \node{\dots\mathbb{G}_{n-2}} \arrow{se,t}{p_1}\\ \node{\mathbb{G}_{0}\to}\node{\vdots}\node{\mathbb{G}_{n-1}}\arrow{nne,t}{p_k}\\ \node{\mathbb{G}_{0}\to}\node{\dots\mathbb{G}_{n-2}} \arrow{ne,b}{p_k}
\end{diagram}
$$

To see an artistic suggestion of the self-similarity arising from iterated satellite operations, consider the case that $R$ is the $9_{46}$ knot shown on the left-hand side of Figure~\ref{fig:R3} with two designated circles along which we will perform infection. Then $R\left (R(R,R),R(R,R)\right )$ is shown on the right-hand side of Figure~\ref{fig:R3}.
\begin{figure}[htbp]
\setlength{\unitlength}{1pt}
\begin{picture}(362,271)
\put(-25,40){\includegraphics{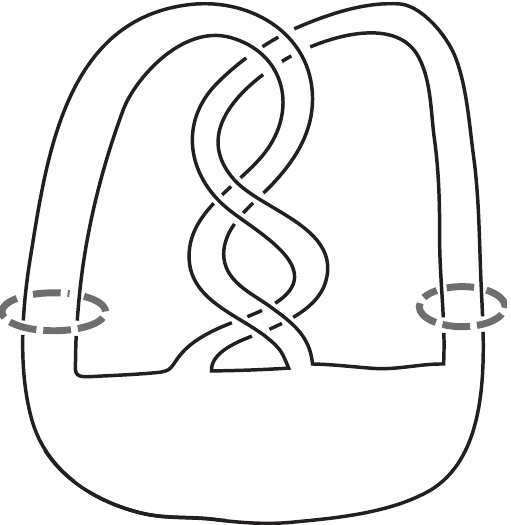}}
\put(180,0){\includegraphics{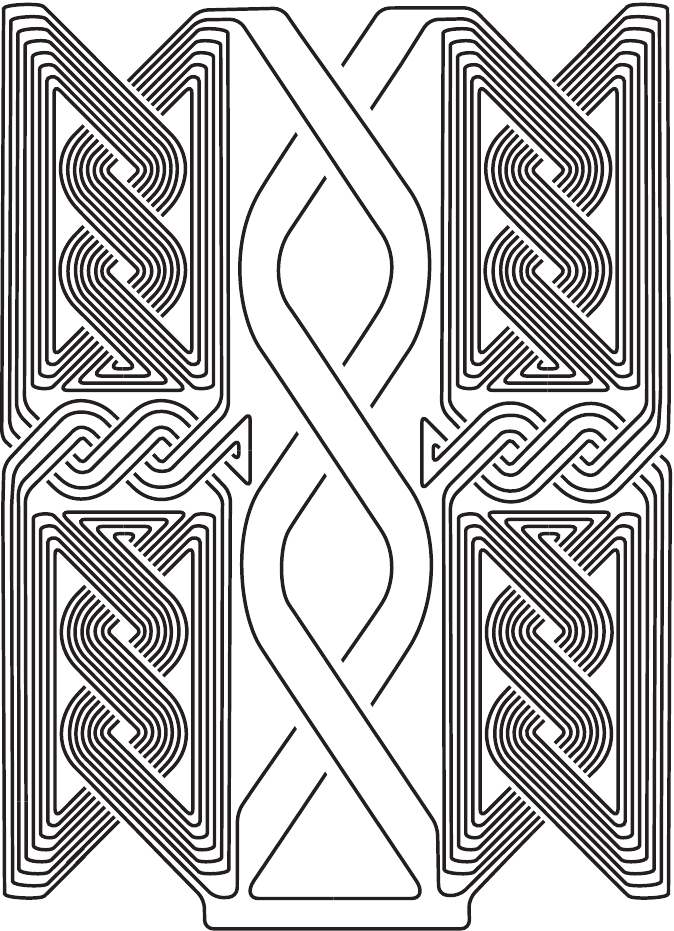}}
\put(-42,99){$\alpha_1$}
\put(127,99){$\alpha_2$}
\end{picture}
\caption{$R\left (R(R,R),R(R,R)\right )$
}\label{fig:R3}
\end{figure}

Finally, recall that Cochran-Orr-Teichner exhibited the first knots with vanishing Casson-Gordon invariants that are not slice in the topological category and showed that  $\mathbb{G}_{n}$ has infinite rank for $n=0,1,2$ ~\cite{COT}\cite{COT2}. In \cite{CT} it was shown that each $\mathbb{G}_{n}$ has positive rank and this was extended in \cite{CK}\cite{Ki2}. As a further example of the utility of different filtrations and commutator series, we show that if $n\geq 2$ then none of the knots that appeared in these papers (henceforth called \textbf{COT knots}- see Section~\ref{sec:COTknots} for precise definitions), is concordant to any of the knots that appeared in the more recent papers of the authors ~\cite{CHL1}\cite{CHL1A}\cite{CHL3}, henceforth called \textbf{CHL knots}. Recall that it was shown in the latter papers (using CHL knots) that each $\mathbb{G}_n$ has \emph{infinite} rank. In particular this implies that the subgroup of focus in the present paper
$$
\bigoplus_{\substack{\mathbb{P}_n}}\Z^\infty\subset \mathcal{F}_n\subset\mathcal{C}
$$
is not \emph{all} of $\mathcal{C}$ if $n\geq 2$.
\vspace{1in}

\textbf{Acknowledgements} We are grateful to Brendan Hassett for very enlightening conversations about some of the algebra in Section~\ref{sec:locprimes}.

\section{Commutator Series and Filtrations of the knot concordance groups}\label{sec:series}

\begin{defn}\label{def:series} A \emph{commutator series} is a function, $*$, that assigns to each group $G$ a nested sequence of normal subgroups
$$
\dots\vartriangleleft\gnp _*\vartriangleleft \gn_*\vartriangleleft\dots\vartriangleleft \ensuremath{G^{\sss (0)}}_*\equiv G,
$$
such that $\gn_*/\gnp_*$ is a torsion-free abelian group. (We restrict to torsion-free in order to avoid zero divisors in $\Z[G/\gn_*]$). A \emph{functorial commutator series} is one that is a functor from the category of groups to the category of series of groups, that is, a commutator series such that, for any group homomorphism $f:G\to P$,
$f(\gn_*)\subset P^{(n)}_*$ for each $n$. If $G_*^{(i)}$ is defined only for $i\leq n$, then this will be called a \emph{partially defined commutator series}.
\end{defn}

The model example is the \emph{rational derived series}, $\{G^{(n)}_r\}$, given by $G^{(0)}_r\equiv G$ and
$$
G^{(n+1)}_r\equiv \{x\in\gn_r~|~\exists k>0, ~x^k\in[\gn_r,\gn_r]\},
$$
first used systematically in ~\cite{Ha1}. There is also the \emph{rational lower central series}
~\cite{St}, and mixtures of these two. These (mixtures of) derived and lower central series subgroups are verbal subgroups ~\cite[Section 2.2]{MKS} and hence fully invariant. It follows readily that these model series are functorial. Other examples will be given in Section~\ref{sec:seriesfromloc}.

\begin{prop}\label{prop:commseriesprops} For any commutator series $\{\gn_*\}$,
\begin{itemize}
\item [1.] $\{x\in\gn_*~|~\exists k>0, ~x^k\in[\gn_*,\gn_*]\}\subset \gnp_*$ (and in particular $[\gn_*,\gn_*]\subset \gnp_*$, whence the name commutator series);
\item [2.] $\gn_r\subset \gn_*$, that is, every commutator series  contains the rational derived series;
\item [3.] $G/\gn_*$ is a poly-(torsion-free abelian) group (abbreviated PTFA);
\item [4.] $\Z [G/\gn_*]$ and $\Q [G/\gn_*]$ are right (and left) Ore domains.
\end{itemize}
\end{prop}

\begin{proof} Recall that a group is \textbf{poly-(torsion-free abelian) (abbreviated PTFA)} if it admits a finite subnormal series for which the successive quotients are torsion-free abelian groups. Properties $1$ and $3$ follow from the definitions. Property $2$ then follows inductively from $1$. Property $4$ is verified in ~\cite[Proposition 2.5]{COT}.
\end{proof}

We will now show that any commutator series that satisfies a weak functoriality induces a filtration, $\{\mathcal{F}^*_n\}$, of $\mathcal{C}$. These filtrations generalize and refine the ($n$)-solvable filtration $\{\mathcal{F}_n\}$ of ~\cite{COT}. Let $M_K$ denote the closed $3$-manifold obtained by zero framed surgery on $S^3$ along $K$. Recall that the motivation for the following filtrations is the following well-known fact: If a knot $K$ admits a slice disk $\Delta\hookrightarrow B^4$ then $M_K$ is the boundary of the $4$-manifold $W=B^4\setminus \Delta$ for which $H_2(W)=0$ and $H_1(M_K)\cong H_1(W)$.

\begin{defn}\label{def:Gnsolvable} A knot $K$ is an element of $\mathcal{F}_{n}^*$ if
the zero-framed surgery $M_K$ bounds a compact smooth $4$-manifold $W$ such that
\begin{itemize}
\item [1.] $H_1(M_K;\Z)\to H_1(W;\Z)$ is an isomorphism;
\item [2.] $H_2(W;\Z)$ has a basis consisting of connected compact oriented surfaces, $\{L_i,D_i~|~1\leq i\leq r\}$,  embedded in $W$ with trivial normal bundles, wherein the surfaces are pairwise disjoint except that, for each $i$, $L_i$ intersects $D_i$ transversely once with positive sign.
\item [3.] for each $i$, $\pi_1(L_i)\subset \pi_1(W)^{(n)}_*$
and $\pi_1(D_i)\subset \pi_1(W)^{(n)}_*$.
\end{itemize}
A knot $K\in \mathcal{F}_{n.5}^*$ if in addition,
\begin{itemize}
\item [4.] for each $i$, $\pi_1(L_i)\subset \pi_1(W)^{(n+1)}_*$
\end{itemize}
Such a $4$-manifold is called an $\boldsymbol{(n,*)}$\textbf{-solution} (respectively \textbf{an }$\boldsymbol{(n.5,*)}$-\textbf{solution}) \textbf{for} $\boldsymbol {K}$ and it is said that $\boldsymbol{K}$ and $\boldsymbol{M_K}$ are $\boldsymbol{(n,*)}$-\textbf{solvable} (respectively $\boldsymbol{(n.5,*)}$-\textbf{solvable}) \textbf{via} $\boldsymbol{W}$. The case where the commutator series is the derived series (without the torsion-free abelian restriction) is denoted $\mathcal{F}_{n}$ and we speak of being of $W$ being an ($n$)-solution, and $K$ or $M_K$ being ($n$)-solvable via $W$ ~\cite[Section 8]{COT}. If a $4$ manifold $W$ is a submanifold of another $4 $-manifold $V$ then we say that $W$ is \textbf{effectively an} $\boldsymbol{(n,*)}$\textbf{-solution} (respectively \textbf{an }$\boldsymbol{(n.5,*)}$-\textbf{solution}) with respect to $V$ if $H_1(\partial W)\cong H_1(W)$ and $W$
satisfies a modified version of conditions [$2$] and [$3$] above:
\begin{itemize}
\item [$3^\prime$.] $H_2(W)$ has a basis as in condition $2.$ such that for each $i$, $\pi_1(L_i)\subset \pi_1(V)^{(n)}_*$
and $\pi_1(D_i)\subset \pi_1(V)^{(n)}_*$.
\end{itemize}
Similarly we say that $W$ is \textbf{effectively an} $\boldsymbol{(n.5,*)}$\textbf{-solution} with respect to $V$ if in addition it satisfies a similarly modified version of condition $4.$ above.
\end{defn}

It follows from the definition that any such $W$ is spin and that the intersection form on $H_2(W)$ is a direct sum of hyperbolic pairs, hence has zero signature. Dropping the hypothesis that $W$ be smooth, one can also define $\mathcal{F}_{n}^{*,TOP}$. But work of Freedman-Quinn implies that this hypothesis is redundant, so that in fact $K\in \mathcal{F}_{n}^{*,TOP}$ if and only if $K\in \mathcal{F}_n^*$ ~\cite[Proposition 8.8A]{FQ}.

Some of the most important group series are not fully functorial, including the one of most interest in this paper (see Theorem~\ref{thm:derivedlocpfunctorial}), but usually a much weaker notion is required for applications. For example, in discussing knots, usually all the spaces one deals with have $H_1$ either infinite cyclic or trivial.

\begin{defn}\label{def:weakfunctorial} A commutator series $\{G^{(n)}_*\}$ is \textbf{weakly functorial} if, for any homomorphism $f:G\to \pi$ that induces an isomorphism $G/G^{(1)}_r\cong \pi/\pi^{(1)}_r$ (i.e. induces an isomorphism on $H_1(-;\Q)$),  $f(G^{(n)}_*)\subset \pi^{(n)}_*$ for each $n$.
\end{defn}

\begin{prop}\label{prop:functseriesfilters} Suppose $*$ is a weakly functorial commutator series defined on the class of groups with $\beta_1=1$. Then  $\{\mathcal{F}_n^*\}_{n\geq 0}$ is a filtration by subgroups of the classical (smooth) knot concordance group $\mathcal{C}$:
$$
\cdots \subset \mathcal{F}_{n+1}^* \subset \mathcal{F}_{n.5}^*\subset\mathcal{F}_n^*\subset\cdots \subset
\mathcal{F}_1^*\subset \mathcal{F}_{0.5}^* \subset \mathcal{F}_{0}^* \subset \mathcal{C}.
$$
\end{prop}

The case where the commutator series is the derived series (without the torsion-free abelian restriction) is the \textbf{(n)-solvable filtration} ~\cite{COT}, denoted $\{\mathcal{F}_{n}\}$, and, for any $n$, $\mathcal{F}_n\subset\mathcal{F}_n^*$.

\begin{proof} We sketch the proof. If $K_0$ and $K_1$ are concordant then it is well known that their exteriors, and hence their zero-framed surgeries, $M_{K_0}$ and $M_{K_1}$, are homology cobordant via a $4$-manifold $C$ (obtained by performing zero framed surgery on the annulus). Suppose $K_0\in \mathcal{F}_n^*$ via $W_0$. Let $W_1=W_0\cup C$ so that $\partial W_1=M_{K_1}$. Note that the inclusion $W_0\hookrightarrow W_1$ induces isomorphisms on homology. By weak functoriality, $\pi_1(W_0)^{(k)}_*\subset \pi_1(W_1)^{(k)}_*$ for every $k$. Then it is easy to see that $K_1\in \mathcal{F}_n^*$ via $W_1$ using the surfaces from $W_0$. Therefore $\mathcal{F}_n^*$ descends to define a filtration of $\mathcal{C}$.

We claim that $\mathcal{F}_n^*$ is a subgroup.  If $K$ is a slice knot then $M_K=\partial (B^4-\Delta)$ where $H_2(B^4-\Delta)=0$. Thus $K\in \mathcal{F}_n^*$ via $B^4-\Delta$ for any $n$ and any $*$. It follows that $[0]\in \mathcal{F}_n^*$. Since $-M_K=M_{-K}$, $\mathcal{F}_n^*$ is closed under taking inverses in $\mathcal{C}$.  There is a standard cobordism $E$ whose boundary consists of $M_K$, $M_J$ and $-M_{K\#J}$. If $M_J$ is $(n,*)$-solvable via $W_J$ and $M_K$ is $(n,*)$-solvable via $W_K$, then let $W=W_J\cup W_K \cup E$ so that $\partial W= M_{K\# J}$. Note that the inclusion $W_K\hookrightarrow W$ induces an isomorphism on $H_1$ (in fact $\pi_1(W_0)$ is a retract of $\pi_1(W)$). Using weak functorality, one then easily shows that $K\# J$ is $(n,*)$-solvable via $W$ (using the union of the surfaces from $W_K$ and $W_J$. Thus $\mathcal{F}_n^*$ is closed under connected sum of knots.

Since $G^{(n)}\subset G^{(n)}_*$ it is obvious from the definitions that $\mathcal{F}_n\subset\mathcal{F}_n^*$.
\end{proof}

Examples of knots in these filtration levels can be provided by generalizations of satellite constructions.

\subsection{Doubling operators}\label{subsec:doubling}

Let $R$ be a knot in $S^3$ and $\vec{\alpha}=\{\alpha_1,\alpha_2,\ldots,\alpha_m\}$ be an ordered oriented trivial link in $S^3$, that misses $R$, bounding a collection of oriented disks that meet $R$ transversely as shown on the left-hand side of Figure~\ref{fig:infection}. Suppose $(K_1,K_2,\ldots,K_m)$ is an $m$-tuple of auxiliary knots. Let $R_{\vec{\alpha}}(K_1,\ldots,K_m)$ denote the result of the operation pictured in Figure~\ref{fig:infection}, that is, for each $\alpha_j$, take the embedded disk in $S^3$ bounded by $\alpha_j$; cut off $R$ along the disk; grab the cut strands, tie them into the knot $K_j$ (with no twisting) and reglue as shown in Figure~\ref{fig:infection}.

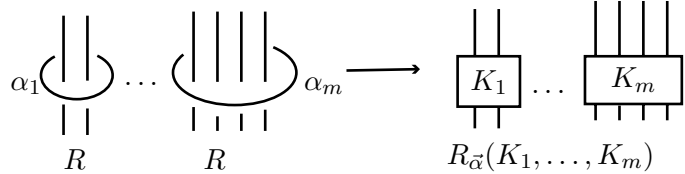
\begin{figure}[htbp]
\setlength{\unitlength}{1pt}
\begin{picture}(262,71)
\put(9,37){$\alpha_1$} \put(120,37){$\alpha_m$} \put(52,39){$\dots$}
\put(206,36){$\dots$} \put(183,37){$K_1$} \put(236,38){$K_m$}
\put(174,9){$R_{\vec{\alpha}}(K_1,\ldots,K_m)$}
\put(29,7){$R$} \put(82,7){$R$}
\put(20,20){\includegraphics{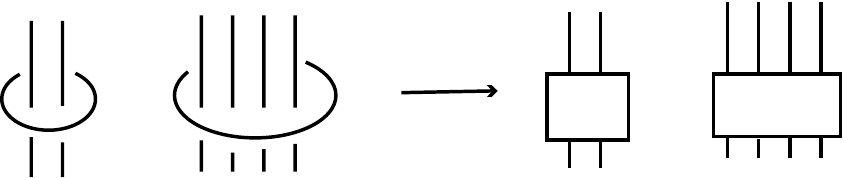}}
\end{picture}
\caption{$R_{\vec{\alpha}}(K_1,\ldots,K_m)$:
Infection of $R$ by $K_j$ along $\alpha_j$}\label{fig:infection}
\end{figure}
\noindent We will call this the \textbf{result of infection performed on the knot} $\boldsymbol{R}$ \textbf{using the infection knots} $\boldsymbol{K_j}$ \textbf{along the curves} $\mathbf{\alpha_j}$. This construction can also be
described in the following way. For each $\alpha_j$, remove a tubular neighborhood of $\alpha_j$ in $S^3$ and glue in the exterior of a tubular neighborhood of $K_j$ along their common boundary, which is a
torus, in such a way that the longitude of $\alpha_j$ is identified with the meridian of $K_j$ and the meridian of $\alpha_j$ is identified with the reverse of the longitude of $K_j$. The resulting space can be seen to be homeomorphic to $S^3$ and the image of $R$ is the new knot. In the case that $m=1$ this is the same as the classical satellite construction. In general it can be considered to be a generalized satellite construction ~\cite{COT2}.

For simplicity, in this paper we focus on the following special case.

\begin{defn}\label{def:doublingop} A \textbf{doubling operator}, $R_\alpha:\mathcal{C}\to \mathcal{C}$ is one that arises from infection on a ribbon knot $R$ along a single curve $\alpha$ where $lk(R,\alpha)=0$.
\end{defn}

\begin{prop}\label{prop:operatorsact} For any weakly functorial commutator series $*$, if $R_\alpha$ is a doubling operator with $\alpha\in \pi_1(M_R)_*^{(k)}$ then
$$
R_\alpha(\mathcal{F}_{n-k})\subset \mathcal{F}_n^*.
$$
Specifically, if $\alpha\in \pi_1(M_R)_*^{(n)}$ and $K$ is an Arf invariant zero knot, then
$$
R_\alpha(K)\in \mathcal{F}_n^*.
$$
The same holds for the more general operators $R_{\vec{\alpha}}$ if, for each $i$, $\alpha_i\in \pi_1(M_R)_*^{(k)}$ and $K_i\in \mathcal{F}_{n-k}$.
\end{prop}

\begin{proof} Consider the last (strongest) claim of the Proposition. Let us use $L$ to abbreviate $R_{\vec{\alpha}}(K_1,...,K_r)$. Recall from ~\cite[Lemma 2.3, Figure 2.1]{CHL3} that there is a cobordism $E$ (as shown on the left-hand side of Figure~\ref{fig:mickey} with $r=2$) whose boundary is the disjoint union of the zero-framed surgeries on $L$, $K_1$,...,$K_r$ and $R$. Recall that $E$ is obtained from the union of $M_R\times [0,1]$ and the various $M_{K_i}\times [0,1]$ by identifying the solid tori $\alpha_i\times D^2\subset M_R\times \{1\}$ with the solid tori $M_{K_i}-(S^3-K_i)$. Suppose $K_i\in \mathcal{F}_{n-k}$ via $V_i$ and suppose $S=B^4-\Delta$ where $\Delta$ is a slice disk for the slice knot $R$. Gluing these to $E$ we obtain the $4$-manifold $W$, as shown on the right-hand side of Figure~\ref{fig:mickey}, whose boundary is the zero framed surgery on
$L$.

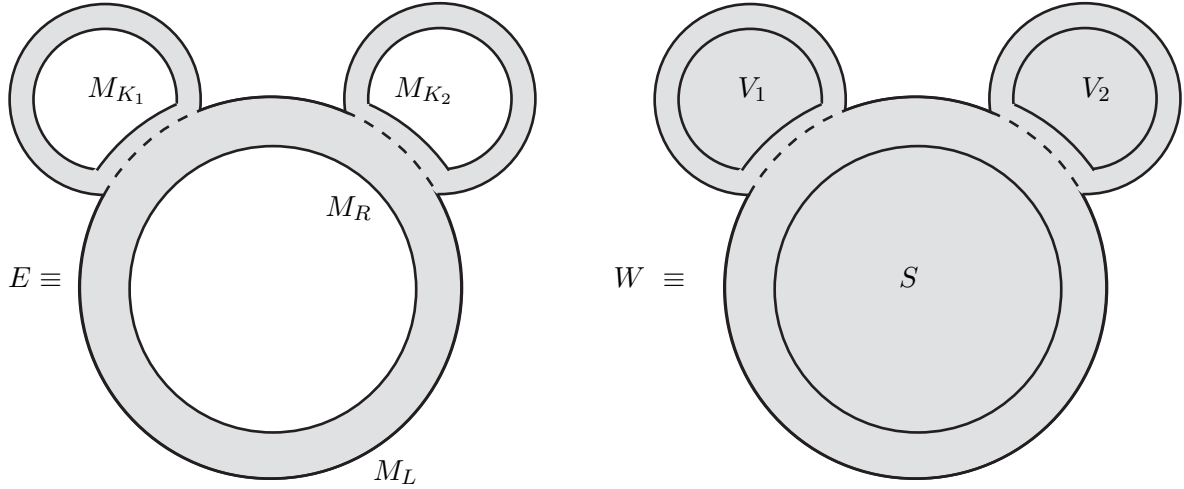
\begin{figure}[htbp]
\setlength{\unitlength}{1pt}
\begin{picture}(350,200)
\put(-50,0){\includegraphics{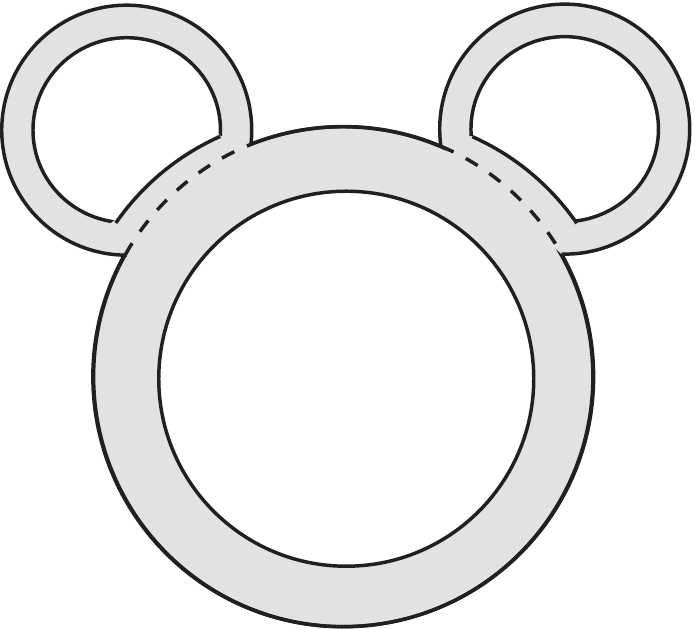}}
\put(194,0){\includegraphics{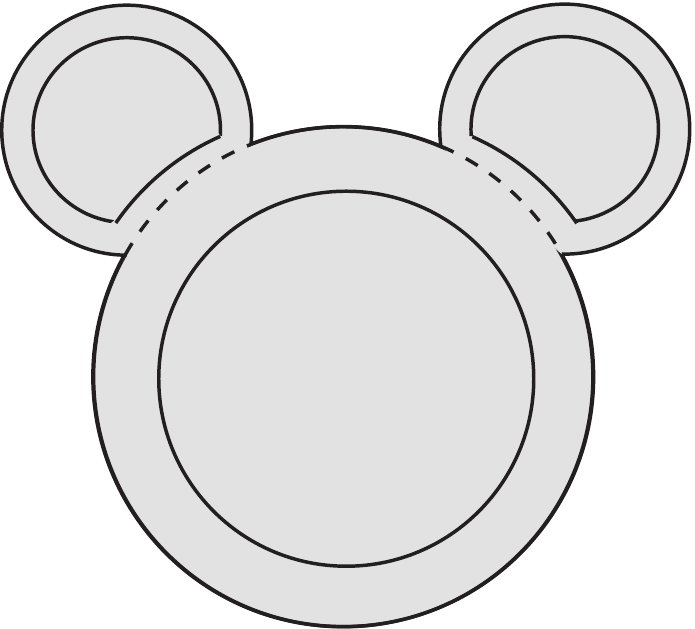}}
\put(70,100){$M_R$}
\put(88,0){$M_L$}
\put(-20,145){$M_{K_1}$}
\put(96,145){$M_{K_2}$}
\put(226,145){$V_1$}
\put(356,145){$V_2$}
\put(287,73){$S$}
\put(179,73){$W~~\equiv$}
\put(-50,73){$E\equiv$}
\end{picture}
\caption{The cobordism}\label{fig:mickey}
\end{figure}
We claim that $L\in\mathcal{F}_n^*$ via $W$. For this we must analyze the homology of $W$ using the following result.
\begin{lem}\label{lem:mickeyfacts}~\cite[Lemma 2.5]{CHL3} With regard to $E$ above, the inclusion maps induce
\begin{itemize}
\item [(1)] an epimorphism $\pi_1(M_L)\to \pi_1(E)$ whose kernel is the normal closure of the longitudes of the infecting knots $K_i$ viewed as curves $\ell_i\subset S^3-K_i\subset M_L$;
\item [(2)] isomorphisms $H_1(M_L)\to H_1(E)$ and $H_1(M_R)\to H_1(E)$;
\item [(3)] and isomorphisms $H_2(E)\cong H_2(M_L)\oplus_i H_2(M_{K_i})\cong H_2(M_R)\oplus_i H_2(M_{K_i})$.
\item [(4)] The meridian of $K_i$, $\mu_i\subset M_{K_i}$ is isotopic in $E$ to both $\alpha_i\subset M_R$ and to the longitudinal push-off of $\alpha_i$, $\ell_{\alpha_i}\subset M_L$.
\item [(5)] The longitude of $K_i$, $\ell_i\subset M_{K_i}$ is isotopic in $E$ to the reverse of the meridian of $\alpha_i$, $(\mu_{\alpha_i})^{-1}\subset M_L$ and to the longitude of $K_i$ in $S^3-K_i\subset M_L$ and to the reverse of the meridian of $\alpha_i$, $(\mu_{\alpha_i})^{-1}\subset M_R$ (the latter bounds a disk in $M_R$).
\end{itemize}
\end{lem}
The inclusion maps $M_R\to S$, $M_{K_i}\to V_i$ induce isomorphisms on $H_1$ and zero maps on $H_2$. From this and ($2$) of Lemma~\ref{lem:mickeyfacts}, it follows that $H_1(M_L)\to H_1(W)$ is an isomorphism. A Mayer-Vietoris sequence then implies that
$$
H_2(W)\cong \oplus_{i=1}^rH_2(V_i)
$$
since $H_2(S)=0$. For each fixed $i$, $H_2(V_i;\Z)$ has a basis consisting of connected compact oriented surfaces, $\{L_j,D_j|1\leq j\leq r_i\}$, satisfying the conditions of Definition~\ref{def:Gnsolvable}. In particular $\pi_1(L_j)\subset \pi_1(V_i)^{(n-k)}$
and $\pi_1(D_j)\subset \pi_1(V_i)^{(n-k)}$.
We claim that
\begin{equation}\label{eq:natural}
\pi_1(V_i)\subset \pi_1(W)^{(k)}_*.
\end{equation}
Assuming this for the moment it then would follow that
$$
\pi_1(L_j)\subset \pi_1(V_i)^{(n-k)}\subset (\pi_1(W)^{(k)}_*)^{(n-k)}\subset\pi_1(W)^{(n)}_*,
$$
where for the last inclusion we use iterated applications of part $1$ of Proposition~\ref{prop:commseriesprops}; and similarly for $\pi_1(D_j)$. This would then complete the verification that $L\in\mathcal{F}_n^*$ via $W$.

To establish our claim ~(\ref{eq:natural}), consider the inclusion $\phi:\pi_1(M_{K_i})\to \pi_1(V_i)$. Since the meridian $\mu_i$ normally generates $\pi_1(M_{K_i})$, it normally generates $\pi_1(V_i)$ modulo $\pi_1(V_i)^{(k)}$ by the following elementary result.

\begin{lem}\label{lem:normgender}~\cite[Lemma 6.5]{CHL4} Suppose $\phi:A\to B$ is a group homomorphism that is surjective on abelianizations. Then, for any positive integer $k$, $\phi(A)$ normally generates $B/B^{(k)}$.
\end{lem}

\noindent Continuing with our proof of Proposition~\ref{prop:operatorsact}, since
$$
\pi_1(V_i)^{(k)}\subset \pi_1(W)^{(k)}\subset\pi_1(W)^{(k)}_*,
$$
to establish ~(\ref{eq:natural}) we need only show that $\mu_i\in \pi_1(W)^{(k)}_*$. By property ($4$) of Lemma~\ref{lem:mickeyfacts}, $\mu_i$ is isotopic in $W$ to $\alpha_i\subset M_R$. By hypothesis $\alpha_i\in \pi_1(M_R)_*^{(k)}$ and combined with the weak functoriality of the commutator series we conclude
$$
\alpha_i\in \pi_1(M_R)_*^{(k)}\subset \pi_1(W)^{(k)}_*,
$$
as required.
\end{proof}

\
\section{Commutator Series from Localization}\label{sec:seriesfromloc}

It is a consequence of Proposition~\ref{prop:commseriesprops} that, in any commutator series, the canonical epimorphism $\gn_*\to \gn_*/\gnp_*$ factors as follows
$$
\gn_*\to \frac{\gn_*/[\gn_*,\gn_*]}{\Z -\text{torsion}}\overset{\pi_n}{\to}\gn_*/\gnp_*,
$$
Hence each commutator series is uniquely determined  by recursively specifying the normal subgroups that will be the kernels of the maps $\pi_n$.

In this section we generalize a philosophy introduced by the second author ~\cite{Ha2} to define various commutator series by killing (via the $\pi_n$) selected torsion elements of $\gn_*/[\gn_*,\gn_*]$, where the latter is viewed as a module (see below). This is accomplished by noncommutative localization. We first review (classical) noncommutative localization of rings and modules.

\begin{subsection}{Review of Classical Localization for Domains}\label{sub:classlocal}

Let $R$ be a non-trivial domain with unity and $S$ be a multiplicatively closed set of non-zero elements of $R$ with $1\in S$. Then the \emph{right quotient ring}, $RS^{-1}$, is a ring containing $R$ with the property that
\begin{itemize}
\item [i.] every element of $S$ has an inverse in $RS^{-1}$, and
\item [ii.] every element of $RS^{-1}$ is of the form $rs^{-1}$ with $r\in R, s\in S$.
\end{itemize}
Such quotient rings exist if $R$ is commutative, but in general one needs extra conditions on $S$ to guarantee their existence.

\begin{def}\label{def:divisorset} A subset $S\subset R$ is called a \emph{right divisor set} of $R$ if
\begin{itemize}
\item [1.] $1\in S$, $0\notin S$;
\item [2.] $S$ is multiplicatively closed and
\item [3.] given $r\in R, s\in S$, there exist $r'\in R$, $s'\in S$ such that $rs'=sr'$.
\end{itemize}
\end{def}
If $S$ is a right divisor set of $R$ then the right quotient ring $RS^{-1}$ exists ~\cite[Theorem 2.12, page 427]{P}. Note that the third condition is always satisfied if $R$ is commutative. A domain $R$ for which $R\setminus\{0\}$ is a right divisor set is called a \emph{right Ore domain}. In this case the right quotient field $R(R\setminus\{0\})^{-1}$ is a division ring called the (right classical) \emph{quotient field} of $R$ and will be denoted $\mathcal{K}R$. The ring $RS^{-1}$ is naturally an $R-RS^{-1}$ bimodule and is flat as a left $R$-module ~\cite[Prop. II.3.5]{Ste}.

If $R$ is a domain and $\mathcal{M}$ is a right $R$-module then an element $x\in \mathcal{M}$ is \emph{torsion} if there is some non-zero $s\in R$ such that $xs=0$. Then we say that \emph{$x$ is $s$-torsion}. If $S$ is a right divisor set of $R$ then the set of elements of $\mathcal{M}$ that are $s$-torsion for some $s\in S$ is a submodule called the \emph{$S$-torsion submodule of $\mathcal{M}$}. The connection between localization and torsion is given by the fact that
$$
\ker\left( \mathcal{M}\to \mathcal{M}\otimes_R RS^{-1}\right)
$$
is precisely the $S$-torsion submodule of $\mathcal{M}$ ~\cite[Cor.II.3.3]{Ste}. We will often abbreviate $\mathcal{M}\otimes RS^{-1}$ by $\mathcal{M}S^{-1}$, a right $RS^{-1}$-module that is called the \emph{localization of the module} $\mathcal{M}$ corresponding to $S$.
\end{subsection}

\begin{subsection}{Commutator series arising from localizations}\label{sub:commserieslocal}

We follow and generalize ~\cite[Section 2]{Ha2}.  Given any group $G$, set $G^{(0)}_\mathcal{S}\equiv G$ and suppose inductively that $G^{(k)}_\mathcal{S}$ has been defined for $k\leq n$ in such a way that $G^{(k)}_\mathcal{S}/G^{(k+1)}_\mathcal{S}$ is a torsion-free abelian group for each $k<n$. This is a partially defined commutator series.  Then $G/\gn_\mathcal{S}$ is a poly-(torsion-free-abelian) group (henceforth called PTFA) and consequently $\Q[G/\gn_\mathcal{S}]$ is a domain (Bovdi, see ~\cite[p.592]{P}), and in fact is a right Ore domain ~\cite[Lemma 3.6 iii, p.6.11]{P}). Recall that that
 $$
\frac{\gn_\mathcal{S}}{[\gn_\mathcal{S} ,\gn_\mathcal{S} ]}
 $$
 is not only an abelian group but also a right $\Z [G/\gn_\mathcal{S}]$-module where the action is induced by conjugation ($x*g=g^{-1}xg$ for any $g\in G$ and $x\in \gn_\mathcal{S})$. Henceforth the action of $g$ on $x$ will be denoted by $xg$. Suppose $S_n\subset \Q[G/\gn_\mathcal{S}]$ is a right divisor set. Then define
\begin{equation}\label{eq:defserieslocSn}
\gnp_\mathcal{S}\equiv \ker \left(\gn_\mathcal{S}\to\frac{\gn_\mathcal{S}}{[\gn_\mathcal{S} ,\gn_\mathcal{S} ]}\to \frac{\gn_\mathcal{S}}{[\gn_\mathcal{S} ,\gn_\mathcal{S} ]}\otimes_{\Z [G/\gn_\mathcal{S}]}\Q[G/\gn_\mathcal{S}]S_n^{-1}\right)
\end{equation}
where we use the fact that $\Q[G/\gn_\mathcal{S}]S_n^{-1}$ is a left $\Z[G/\gn_\mathcal{S}]$-module. The second map in this composition should really be viewed as first tensoring with $\Q$, which kills the $\Z$-torsion, and then inverting $S_n$, which kills the $S_n$-torsion. It follows easily that $\gn_\mathcal{S}/\gnp_\mathcal{S}$ is a torsion-free abelian group (since $\Q[G/\gn_\mathcal{S}]S_n^{-1}$ is a rational vector space). Hence this procedure recursively defines a commutator series  (or a partial commutator series up to $ \gnp_\mathcal{S}$), that depends only on $\mathcal{S}$, a sequence, $S_1,...,S_n$ of right divisor sets.  Note that since $S_0\subset\Q$, the choice of $S_0$ is irrelevant and we ignore it. For the same reason, for any commutator series so defined,
\begin{equation}\label{eq:firstderived}
 G^{(1)}_\mathcal{S}=G^{(1)}_r,
\end{equation}
the first term of the rational derived series (the radical of the commutator subgroup).

Elements of $\gnp_\mathcal{S}$ have a simple characterization. Associated to $S_n\subset \Q[G/\gn_\mathcal{S}]$ is a ``lifting'' $\tilde{S_n}\subset \Z[G/\gn_\mathcal{S}]$ consisting of those elements $\tilde{s}$ of $\Z[G/\gn_\mathcal{S}]$ for which $\tilde{s}r\in S_n$ for some non-zero $r\in \Q$. Note that if $\Q-\{0\}\subset S_n$ then $\tilde{S_n}\subset S_n$ since $S_n$ is multiplicatively closed.

\begin{prop}\label{prop:char*} If $x\in \gn_\mathcal{S}$, then $x\in \gnp_\mathcal{S}$ if and only if $x$ represents $\tilde{S}_n$-torsion in the module $\frac{\gn_\mathcal{S}}{[\gn_\mathcal{S} ,\gn_\mathcal{S} ]}$.
\end{prop}
\begin{proof} Suppose $x\in \gnp_\mathcal{S}$. Let $\overline{x}$ be the class represented by $x$ in the module $\gn_\mathcal{S}/[\gn_\mathcal{S} ,\gn_\mathcal{S} ]$. Then there exists some $s \in S_n$ such that $\overline{x}\otimes 1$ is $s$-torsion in the module
$$
\frac{\gn_\mathcal{S}}{[\gn_\mathcal{S} ,\gn_\mathcal{S} ]}\otimes_{\Z [G/\gn_\mathcal{S}]}\Q[G/\gn_\mathcal{S}]\cong \frac{\gn_\mathcal{S}}{[\gn_\mathcal{S} ,\gn_\mathcal{S} ]}\otimes_{\Z}\Q.
$$
Let $k$ be a non-zero integer such that $sk\in \Z[G/\gn_\mathcal{S}]$. Then $\overline{x}\otimes 1$ is also $sk$-torsion. Thus
$$
0=\overline{x}\otimes sk= \overline{x}(sk)\otimes 1.
$$
Since tensoring an abelian group with $\Q$ kills only $\Z$-torsion, this implies that $\overline{x}(sk)$ is $\Z$-torsion in $\gn_\mathcal{S}/[\gn_\mathcal{S} ,\gn_\mathcal{S} ]$. Hence $\overline{x}$ is annihilated by $\tilde{s}=skt$ for some non-zero integer $t$. Since  $(1/kt)\tilde{s}=s \in S_n$, $\tilde{s}\in \tilde{S}_n$. Thus $\overline{x}$ is $\tilde{S}_n$-torsion.

The converse follows similarly.
\end{proof}

Whether or not such commutator series are functorial depends on whether or not the right divisor sets $S_1,...,S_n$ are defined in a ``functorial'' manner.

\begin{prop}\label{prop:preservSisfunctorial} Suppose sequences of right divisor sets $\mathcal{S}^A$ and $\mathcal{S}^B$ are chosen for the groups $A$ and $B$ respectively yielding partial commutator series. If $\psi:A\to B$ is a group homomorphism such that $\psi(S^A_i)\subset S^B_i$ for $0\leq i\leq n$ then $\psi(A^{(n+1)}_{\mathcal{S}})\subset B^{(n+1)}_{\mathcal{S}}$.
\end{prop}
\begin{proof} The proposition is true for $n=0$ since $A^{(1)}_\mathcal{S}=A^{(1)}_r$ and the rational derived series is functorial. Now assume it is true for $n-1$. Assume $\psi(S^A_i)\subset S^B_i$ for $0\leq i\leq n$. Then, by the induction hypothesis, $\psi(A^{(i)}_{\mathcal{S}})\subset B^{(i)}_{\mathcal{S}}$ for $0\leq i\leq n$. In particular $\psi$ induces a ring homomorphism
$$
\psi:\Z [A/\an_\mathcal{S}]\to \Z [B/\bn_\mathcal{S}],
$$
(and similarly for the rational group rings) such that, by hypothesis, $\psi(S^A_n)\subset S^B_n$. Thus we have the following commutative diagram,
$$
\begin{diagram}\label{diagram-propSfunct}\dgARROWLENGTH=1.2em
\node{\an_\mathcal{S}}\arrow{s,r}{\psi}\arrow{e,t}{}\node{\frac{\an_\mathcal{S}}{[\an_\mathcal{S} ,\an_\mathcal{S} ]}} \arrow{s,r}{\psi_*}\arrow{e,t}{}\node{\frac{\an_\mathcal{S}}{[\an_\mathcal{S} ,\an_\mathcal{S} ]}\otimes_{\Z [A/\an_\mathcal{S}]}\Q[A/\an_\mathcal{S}](S_n^A)^{-1}}
\arrow{s,r}{\psi_*}\\
\node{\bn_\mathcal{S}}\arrow{e,t}{}\node{\frac{\bn_\mathcal{S}}{[\bn_\mathcal{S} ,\bn_\mathcal{S} ]}} \arrow{e,t}{}\node{\frac{\bn_\mathcal{S}}{[\bn_\mathcal{S} ,\bn_\mathcal{S} ]}\otimes_{\Z [B/\bn_\mathcal{S}]}\Q[B/\bn_\mathcal{S}](S_n^{B})^{-1}}
\end{diagram}
$$
from which it follows immediately that $\psi(A^{(n+1)}_{\mathcal{S}})\subset B^{(n+1)}_{\mathcal{S}}$.
\end{proof}

\end{subsection}

\begin{subsection}{Examples from the previous literature}\label{subsec:examplesseries}

The series that have appeared in the previous literature may be seen in this context.

\begin{ex}\label{ex:commseriesRD} [The Rational Derived Series] If we choose each $S_n$ minimally, that is $S_n=\{1\}$, then the resulting commutator series is the \emph{rational derived series} $\gn_r$ ~\cite{Ha1}. This series is functorial.
\end{ex}
\begin{ex}\label{ex:Harveyseries} [The Torsion-Free Derived Series] ~If we choose each $S_n$ maximally, that is $S_n=\Q[G/\gn_\mathcal{S}]-\{0\}$ then the resulting commutator series is Harvey's \emph{torsion-free derived series}, $G^{(n)}_H$, ~\cite{Ha2}\cite{CH1}. This series has a remarkable monotonicity property under concordance and hence leads to invariants of links and $3$-manifolds ~\cite{Ha4}\cite{CH2} (see also ~\cite{Cha}). This series is not useful for knots since if $G$ is a group with $\beta_1(G)=1$ then $G/G^{(n)}_H\cong \Z$ for $n>0$ \cite[Example 2.9]{CH1}. This series is not functorial, although it is functorial for maps that induce $2$-connected maps on rational homology.
\end{ex}
We are interested in series that interpolate between these extremes.
\begin{ex}\label{ex:commseriesCOT} [The COT Series] The following partial commutator series was suggested by the work of ~\cite{COT} and \cite{CT}. We call it the \textbf{COT series} (at level $n+1$). Fix $n>0$. For each $0\leq k\leq n$ let $G^{(k)}_{cot}=G^{(k)}_r$. Let $S_{n,cot}=\Q[G^{(1)}/\gn_r]-\{0\}$ and set
\begin{equation}\label{eq:defCOTseries}
\gnp_{cot}\equiv \ker \left(\gn_r\to\frac{\gn_r}{[\gn_r ,\gn_r ]}\to \frac{\gn_r}{[\gn_r ,\gn_r ]}S_{n,cot}^{-1}\right)
\end{equation}
as in Equation~(\ref{eq:defserieslocSn}). This partial commutator series agrees with the rational derived series until the $(n+1)$-st term. This series is not functorial. There is an alternative series suggested by their work wherein one performs the analogous localization at \emph{each} stage.
\end{ex}
\begin{ex}\label{ex:commseriesHarvey}
Another example defined and subsequently used extensively by Harvey and others, depends only on an initial choice of an element $\psi\in H^1(G;\Z)=\Hom(G,\Z)$ ~\cite{Ha1}\cite{Lei2}\cite{Lei3}\cite{Lei4}. Given $\psi$, let $S_n=\Q[\ker\psi/\gn_\mathcal{S}]-\{0\}$. Then $\Q[G/\gn_\mathcal{S}](S_n)^{-1}$ can be identified with a twisted polynomial ring $\mathbb{K}_n[t,t^{-1}]$ with coefficients in $\mathbb{K}_n$, the (skew) field of quotients of $\Q[\ker\psi/\gn_\mathcal{S}]$, which is itself a principal ideal domain ~\cite[Section 4]{Ha1}. Thus $\gn_\mathcal{S}/\gnp_\mathcal{S}$ is a module over this principal ideal domain. The torsion submodule of this module is called the $(n-1)^{st}$-higher-order Alexander module of $G$. This series is not functorial.
\end{ex}

\end{subsection}

\section{Localization at polynomials}\label{sec:locprimes}

Our goal in this section is, loosely speaking, to define, using Section~\ref{sec:seriesfromloc}, a commutator series in which, at each stage, $S_n$ consists of certain torsion ``coprime'' to $p_n$ for some chosen $p_n\in \Q[G/\gn_\mathcal{S}]$. This statement does not really make sense over a noncommutative ring. The purpose of this section is to make sense of it. This will entail some new results. By the end of this section we will have defined a new functorial commutator series, called the \textbf{derived series localized at $\mathcal{P}$}, where $\mathcal{P}=(p_1(t),...,p_n(t),...)$ is any sequence of non-zero elements of $\Q[t,t^{-1}]$ (not necessarily prime).

Suppose $p(t)$ is a non-zero element of the Laurent polynomial ring $\Q[t^{\pm 1}]$ and $a\in A$, , where $A$ is a torsion-free abelian group. Then the group homomorphism $\langle t \rangle=\Z\to A$ given by $t\to a$ induces a ring homomorphism $\Q[t^{\pm 1}]\to \Q A$. Thus $p(a)\in\Q A$ has an obvious meaning, as the image of $p$ under this map. Note that if $p(1)\neq 0$ then $p(a)$ is non-zero for any $a$.

Over a commutative UFD, in order to ``localize at $p(a)$'', we would seek to invert all $q$ ``coprime to $p(a)$''. In our more general situation, this begs the question: What is the smallest localization of $\Q\G$ in which $q(a)$ has an inverse?'' (One could also ask about the maximal divisor sets that do \emph{not} contain the fixed $p(a)$, but we do not address this here.) To address this question we must investigate the right divisors sets that contain $q(a)$. First we have a very general result, presumably known to the experts. This result is central to all of our efforts.

\begin{prop}\label{prop:divisorsets} Suppose $A\lhd \G$ where $\Q A$ is a domain. Suppose $S$ is a right divisor set of $\Q A$ that is $\G$-invariant ($g^{-1}Sg=S$ for all $g\in\G$). Then $S$ is a right divisor set of $\Q\G$.
\end{prop}
\begin{proof} Since $S$ is a right divisor set of $\Q A$, $S$ is multiplicatively closed, contains $1$ and does not contain $0$. Suppose $s\in S$ and  $\beta\in\Q\G$ where $\beta=\sum_{i=1}^k r_ig_i$ for some $r_i\in \Q$ and $g_i\in \G$. Since $S$ is a right divisor set of $\Q A$ and $s^{g_i}=g_i^{-1}sg_i$ is an element of $S$, the set $\{(s^{g_i})^{-1}~|~i=1,...,n\}$ has a common denominator $\tilde{s}\in S$ ~\cite[Lemma 2.13, p.428]{P}, that is to say there exist $\beta_i'\in\Q A$ and $\tilde{s}\in S$ such that, for each $i$,
$$
(s^{g_i})^{-1}=\beta_i'(\tilde{s})^{-1},
$$
implying that
$$
\tilde{s}=(s^{g_i})\beta_i'.
$$
(But we remark that, for applications in this paper, we will always take $A$ to be abelian, in which case $\Q A$ is commutative, and the common denominator is merely the product!)
Set $\beta'=\sum_{i=1}^k r_ig_i\beta_i'$. Then
$$
\beta \tilde{s}=\sum r_ig_i\tilde{s}=\sum r_ig_i(s^{g_i})\beta_i'=s\sum r_ig_i\beta_i'=s\beta '.
$$
Therefore $S$ is a right divisor set of $\Q\G$.
\end{proof}

\begin{defn}\label{def:S(Q)} Suppose $\mathcal{Q}$ is a set of polynomials $q_i(t)\in \Q[t^{\pm 1}]$ with $q_i(1)\neq 0$. For any $\tilde{A}\subset A\lhd \G$ where $A$ is a torsion-free abelian group, $\tilde{A}$ is a $\G$-invariant subset, and $\Q\G$ is a domain we define
$$
S=S(\mathcal{Q})=\{q_1(a_{1})...q_r(a_{r})~|~q_i\in \mathcal{Q}, a_{i}\in \tilde{A}, r\geq 0 \}\subset \Q A\subset \Q\G,
$$
the set of all finite products of evaluations of the elements of $\mathcal{Q}$. If $\mathcal{Q}$ is empty, we understand that $S(\mathcal{Q})=\{1\}$. Note that $S$ does not depend on $\G$, only on $\tilde{A}$.
\end{defn}

\begin{cor}\label{cor:divisorsets} $S=S(\mathcal{Q})$, as above, is a right divisor set of $\Q\G$. Moreover this correspondence is functorial in the sense that for any homomorphism of such triples, $\psi:(\G,A,\tilde{A})\to (\G',A',\tilde{A}')$, we have $\psi(S^{\G,\tilde{A}}(\mathcal{Q}))\subset S^{\G ',\tilde{A}'}(\mathcal{Q})$. If $\psi:\tilde{A}\to \tilde{A}'$ is surjective then $\psi(S^{\G,\tilde{A}}(\mathcal{Q}))= S^{\G ',\tilde{A}'}(\mathcal{Q})$.
\end{cor}
\begin{proof} By construction $S$ is a multiplicatively closed subset of $\Q A$ that contains $1$ (the empty product). Certainly $q_i(a_i)\neq 0$ since its augmentation, $q_i(1)$, is non-zero. Since $A$ is torsion-free abelian, $\Q A$ is a commutative domain. Thus $0\neq S$. Hence $S$ is a right divisor set of $\Q A$. For any $g\in \G$ and $a\in \tilde{A}$,
$$
g^{-1}q(a)g=q(g^{-1}ag)=q(a')
$$
where $a'\in \tilde{A}$ since $\tilde{A}$ is $\G$-invariant . It follows that $S$ is $\G$-invariant. Proposition~\ref{prop:divisorsets} then implies that $S$ is a right divisor set of $\Q\G$.

If $\psi:(\G,A)\to (\G ',A')$ then $\psi$ induces a ring homomorphism $\psi:\Q A\to \Q A '$ with respect to which $\psi(q(a))=q(\psi(a))$. Hence $\psi(S^{\G,\tilde{A}}(\mathcal{Q}))\subset S^{\G,\tilde{A}' }(\mathcal{Q})$. If, in addition, $\psi(\tilde{A})=\tilde{A}'$, then, given $q(a')$, where $a'\in \tilde{A}'$ there is an $a\in \tilde{A}$ such that $\psi(a)=a'$ so $\psi(q(a))=q(a')$. Hence $S^{\G ',\tilde{A}'}(\mathcal{Q})\subset \psi(S^{\G,\tilde{A}}(\mathcal{Q}))$.
\end{proof}

In our applications we will focus on inverting only one-variable polynomials. Corollary~\ref{cor:divisorsets} provides us with right divisor sets that contain a fixed element $q(a)$, by, for example, taking $\mathcal{Q}=\{q(t)\}$. But how do we characterize the set $S(\mathcal{Q})$? Specifically, if $p(t)$ is relatively prime to $q(t)$ then is $p(a)\in S(\{q(t)\})$?

\begin{defn}\label{def:stronglycoprime} Two non-zero polynomials $p(t),q(t)\in\Q [t,t^{-1}]$ are said to be \textbf{strongly coprime}, denoted $\widetilde{(p,q)}=1$ if, for every pair of non-zero integers, $n,k$, $p(t^n)$ is relatively prime to $q(t^k)$. Otherwise they are said to be \textbf{isogenous}, denoted $\widetilde{(p,q)}\neq 1$. Alternatively, $\widetilde{(p,q)}\neq 1$ if and only if there exist \emph{non-zero} roots, $r_p, r_q\in \mathbb{C}$*, of $p(t)$ and $q(t)$ respectively, and non-zero integers $k,n$, such that $r_p^k=r_q^n$. Clearly, $\widetilde{(p,q)}= 1$ if and only if for each prime factor $p_i(t)$ of $p(t)$ and $q_j(t)$ of $q(t)$, $\widetilde{(p_i,q_j)}= 1$.
\end{defn}

The following is another useful characterization of being \textbf{strongly coprime}.

\begin{prop}\label{prop:characterizationequivalent} Suppose $p(t), q(t)\in \Q [t,t^{-1}]$ are non-zero. Then $p$ and $q$ are strongly coprime if and only if, for any finitely-generated free abelian group $F$ and any nontrivial $a,b\in F$, $p(a)$ is relatively prime to $q(b)$ in $\Q F$ (a unique factorization domain).
\end{prop}
\begin{proof} $(\Rightarrow$) We show the contrapositive. Suppose that for some $F,a$ and $b$,  $p(a)$ and $q(b)$ do have a common factor over $\Q F$ that is not a unit. It follows immediately that neither $p(a)$ nor $q(b)$ is a unit. Thus they have a common factor over $\mathbb{C}F$ that is not a unit since $\mathbb{C}F$ has only the trivial units ($zf$, $z\in\mathbb{C}$ $f\in F$). Since $a$ and $b$ are non-trivial, we can choose a basis $\{x,y,x_3...x_\ell\}$ for $F$ such that $a=x^n$ and $b=x^ky^m$ for some positive integer $n$ and integers $k$ and $m$. Then we can identify $\mathbb{C} F$ with $\mathbb{C} [x^{\pm 1},y^{\pm 1},...,x_\ell^{\pm 1}]$. Observe $p(x^n)$ factors in $\mathbb{C}[x^{\pm 1},y^{\pm 1},...,x_{\ell}^{\pm 1}]$ into a unit times a product of monomials $x-\alpha_i$ where $\alpha_i$ runs over the set of all the  $n^{th}$ roots of the non-zero roots of $p(t)$ (with multiplicity). This set is non-empty since $p(a)$ is not a unit. Thus $x-\alpha_i$ divides $q(x^ky^m)$ for some $\alpha_i\neq 0$ where $\alpha^n=r$, where $r$ is a non-zero root of $p(t)$. On the other hand $q(x^ky^m)$  factors as a unit times
$$
\prod (x^ky^m-s_j)
$$
where $\{s_j\}$ are the non-zero roots of $q(t)$. This set is non-empty since $q(t)$ is not a unit. Therefore $x=\alpha_i$ is a zero of some $x^ky^m-s_j$ \emph{for every y}. This implies that $m=0$, $k\neq 0$ and that $\alpha_i^k=s_j$ so $(\alpha_i^n)^k=s_j^n$. Hence $r^k=s_j^n$. Thus  $\widetilde{(p,q)}\neq 1$,  as claimed.

$(\Leftarrow$): Take $F=\Z=\langle t \rangle$. Then for any non-zero $k$ and $n$, $a=t^n$ and $b=t^k$ are non-trivial in $F$. By hypothesis $(p(t^n),q(t^k))=1$ in $\Q F=\Q[t,t^{-1}]$. Thus $p$ and $q$ are strongly coprime.
\end{proof}

\begin{ex}\label{ex:isogeny} Consider $p(t)=t-4$ and $q(t)=t^2-4$. Then $(p,q)=1$ since they have no common roots. But $p(t)$ and $q$ are not strongly coprime since $p(t^2)=q(t)$.
\end{ex}
\begin{ex}\label{ex:isogenoustodual}  $p(t)$ is isogenous to $p(t^{-1})$ for any non-zero, non-unit polynomial.
\end{ex}
\begin{ex}\label{ex:isogenycyclo} If $p$ is cyclotomic and $q$ is prime then $p$ is isogenous to $q$ if and only if $q$ is cyclotomic.
\end{ex}
\begin{ex}\label{ex:remarklogs} Given $p(t)$ and $q(t)$, if the sets $\{\text{log}(r)~|~r\in\mathbb{C}^*, p(r)=0\}$ and $\{\text{log}(r)~|~r\in\mathbb{C}^*, q(r)=0\}$ are linearly independent in the rational vector space $\mathbb{C}/\langle 2\pi i\rangle$, then $\widetilde{(p,q)}=1$.
\end{ex}
\begin{ex}\label{ex:stronglycoprime} Consider the family of quadratic polynomials
$$
\{p_k(t)=(kt-(k+1))((k+1)t-k)~|~k\in \mathbb{Z}^+\},
$$
whose roots are positive rational integers
$$
\mathcal{R}_k=\left\{\frac{k}{k+1},\frac{k+1}{k}\right\}.
$$
These are the Alexander polynomials of the family of ribbon knots in Figure~\ref{fig:ribbonfamily} where the $-k$ inside a box symbolizes $k$ full negative twists between the bands.
\begin{figure}[htbp]
\setlength{\unitlength}{1pt}
\begin{picture}(327,151)
\put(85,0){\includegraphics{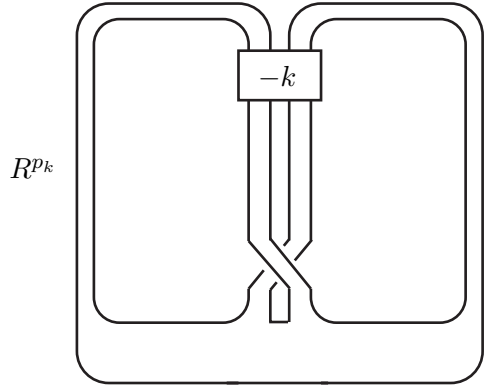}}
\put(60,78){$R^{p_k}$}
\put(154,113){$-k$}
\end{picture}
\caption{The ribbon knots $R^{p_k}$}\label{fig:ribbonfamily}
\end{figure}
We claim that $\widetilde{(p_k,p_\ell)}=1$ if $k\neq \ell$. For suppose that $p_k$ and $p_\ell$ were isogenous where $k>\ell$. It then follows that for some positive integers $n,m$ we have
$$
\left(\frac{k}{k+1}\right)^n=\left(\frac{\ell}{\ell+1}\right)^m.
$$
It follows that $n>m$ and that
\begin{equation}\label{eq:exstronglycoprime}
k^n\left(\ell +1\right)^m=\ell^m(k+1)^n.
\end{equation}
Since $k\neq \ell$ there is some prime integer $q$ and positive integer $w$ such that $q^w$ divides $k$ but does not divide $\ell$. Then $(k+1,q)=1$. Now $q^{wn}$ divides the left-hand side of ~\ref{eq:exstronglycoprime} but not the right-hand side, a contradiction.  Hence $\widetilde{(p_k,p_\ell)}=1$.
\end{ex}

We are finally prepared to define our notion of \textbf{localization at $p(t)$}.

\begin{defn}\label{def:Sp} Suppose $A\lhd \G$ where $A$ is a torsion-free abelian group and $\Q\G$ is a right Ore domain. If $p(t)\in\Q [t,t^{-1}]$ is non-zero then set
$$
S_p=S(\{q(t) \in \Q [t,t^{-1}]~|~q(1)\neq 0; ~\widetilde{(p,q)}= 1\})
$$
$$
=\{ q_1(a_1)...q_r(a_r)~|~\widetilde{(p,q_j)}=1; ~q_j(1)\neq 0; ~a_j\in A\}.
$$
In the special case that $A=\G\cong \Z =\langle \mu \rangle$ we can alternatively set
$$
S_p^*=\{ q_1(\mu^{\pm 1})...q_r(\mu^{\pm 1})~|~(p,q_j)=1, ~ q_j(1)\neq 0 ~\}.
$$
By Corollary~\ref{cor:divisorsets}, $S_p$ is a right divisor set of $\Q\G$, and $S_p^*$ is a right divisor set of $\Q[\mu,\mu^{-1}]$. We say that $\Q\G S_p^{-1}$ is $\Q\G$ \textbf{localized at p(t)}. If $\mathcal{M}$ is a right $\Q\G$-module then we say that $\mathcal{M}S_p^{-1}$ is $\mathcal{M}$ \textbf{localized at p(t)}.
\end{defn}

We have inverted $q(\mu^{-1})$ as well as $q(\mu)$  because we want to force our localized rings to inherit the natural involution. Thus if $p(t)\neq p(t^{-1})$, our notion will differ from the classical localization at $p(t)$. Secondly, we are only inverting polynomials whose augmentations are non-zero (for simplicity since it helps establish functoriality).

Over a commutative domain, localizing a module $\mathcal{M}$ at a prime ideal $\langle p \rangle$ kills all torsion in $\mathcal{M}$ except $\langle p \rangle$-torsion. A version of this remains true in our broader context. This result is the lynchpin of this paper. We focus on right modules over the Ore domain $\Q\G$ of the form
$$
\frac{\Q\G}{q(a)\Q\G},
$$

\begin{thm}\label{thm:noStorsion} Suppose $A\lhd \G$ where $A$ is a torsion-free abelian group and $\Q\G$ is a right Ore domain. Suppose $p(t)\in\Q [t,t^{-1}]$ is non-zero.  Then for any  $a_i\in A$.
$$
\frac{\Q\G}{p(a_1)\dots p(a_k)\Q\G} ~~\text{is} ~~S_p-\text{torsion-free},
$$
that is,
$$
\frac{\Q\G}{p(a_1)\dots p(a_k)\Q\G}\to \frac{\Q\G}{p(a_1)\dots p(a_k)\Q\G}S_p^{-1}
$$
is a monomorphism; whereas for any  $q(t)\in\Q [t,t^{-1}]$ with $q(1)\neq 0$ and $\widetilde{(p(t),q(t))}=1$
$$
\frac{\Q\G}{q(a)\Q\G} ~~\text{is} ~~S_p-\text{torsion},
$$
that is,
$$
\frac{\Q\G}{q(a)\Q\G}S_p^{-1}=0.
$$
\end{thm}
\begin{proof} For the first claim, fix $p$ and the $a_i$, suppose that some $x\in \Q\G$, represents an element
$$
[x]\in\Q\G/p(a_1)\dots p(a_k)\Q\G
$$
that is $S_p$-torsion. We will show that $[x]=0$, implying that $\Q\G/p(a_1)\dots p(a_k)\Q\G$ is $S_p$-torsion-free. We have $xs=p(a_1)\dots p(a_k)y$ for some $s\in S_p$ and for some $y\in \Q\G$. We examine this equation in $\Q\G$.

Recall that, since $A\subset \G$, $\mathbb{Q}\G$, viewed as a left $\mathbb{Q}A$-module, is free on the right cosets of $A$ in $\G$ ~\cite[Chapter 1, Lemma 1.3]{P}. Thus, upon fixing a set of coset representatives, any $x\in \Q\G$ has a \emph{unique} decomposition
$$
x=\Sigma_\gamma x_\gamma\gamma ,
$$
where sum is over a set of coset representatives $\{\gamma\in \G\}$  and $x_\gamma\in\mathbb{Q}A$.  Therefore we have
$$
 (\Sigma_\gamma x_\gamma\gamma )s=p(a_1)\dots p(a_k)\Sigma_\gamma y_\gamma\gamma ,
$$
and thus
$$
\Sigma_\gamma ( x_\gamma s^{\gamma^{-1}})\gamma=\Sigma_\gamma (p(a_1)\dots p(a_k)y_\gamma)\gamma,
$$
where $s^{\gamma^{-1}}=\gamma s\gamma^{-1}$ lies in $S_p$ (since $S_p$ is closed under the action of $\G$). It follows that for each coset representative $\gamma$ we have
$$
x_\gamma s^{\gamma^{-1}} =p(a_1)\dots p(a_k)y_\gamma
$$
which is an equation in $\Q A$. Recall that, for each $\gamma$, $s^{\gamma^{-1}}=q_1(b_1)...q_k(b_r)$ for some  $b_j\in A$ and $q_j(t)$ in $\Q[t,t^{-1}]$ (all depending on $\gamma$), where $\widetilde{(p,q_j)}= 1$ and $q_j(1)\neq 0$ . Thus we have
$$
x_\gamma q_1(b_1)...q_k(b_k)~=~p(a_1)\dots p(a_k)y_\gamma.
$$
This may be viewed as an equation in $\Q F_\gamma$ for some free abelian group $F_\gamma\subset A$ of finite rank.  Since $\widetilde{(p,q_j)}= 1$, the greatest common divisor, in $\Q F_\gamma$, of $p(a_i)$ and $q_j(b_j)$ is a unit. Thus, for each $\gamma$ and each $i$, $p(a_i)$ divides $x_\gamma$ in $\Q F_\gamma$. Thus $p(a_i)$ divides each $x_\gamma$ in $\Q A$ so
$$
x=\Sigma_\gamma x_\gamma\gamma=p(a_1)\dots p(a_k)\Sigma_\gamma x'_\gamma\gamma\in p(a_1)\dots p(a_k)\Q\G
$$
implying $[x]=0$. This finishes the proof of the first claim of Theorem~\ref{thm:noStorsion}.

For the second claim of Theorem~\ref{thm:noStorsion}, note that the hypotheses ensure that $q(a)\in S_p$. Now recall that the kernel, $K$, of the canonical map
$$
\mathcal{M}\equiv\frac{\Q\G}{q(a)\Q\G}\to \frac{\Q\G}{q(a)\Q\G}S_p^{-1}\equiv \mathcal{M}S_p^{-1}
$$
is a precisely the $S_p$-torsion submodule of $\mathcal{M}$. Note that $\mathcal{M}$ is a cyclic right $\Q\G$-module generated by $[1]\in\mathcal{M}$ where $1\in \Q\G$. Clearly $[1]q(a)=[q(a)]=0$ in $\mathcal{M}$. Since the generator of $\mathcal{M}$ is $S_p$-torsion, $[1]\in K$, implying that $K=\mathcal{M}$. This establishes the second claim of Theorem~\ref{thm:noStorsion}.
\end{proof}

There is a well-known companion result in the classical case when $A=\G\cong \Z=\langle \mu\rangle$. The proof is the same as above.

\begin{prop}\label{prop:noS*torsion} Suppose $p(t)\in\Q [t,t^{-1}]$ is non-zero and $p(t^{-1})\doteq p(t)$.  Then
$$
\frac{\Q[\mu^{\pm 1}]}{\langle p(\mu)\rangle} ~~\text{is} ~~S_p^*-\text{torsion-free};
$$
whereas for any  $q(t)\in\Q [t,t^{-1}]$ with $q(1)\neq 0$ and $(p(t),q(t))=1$
$$
\frac{\Q[\mu^{\pm 1}]}{\langle q(\mu^{\pm 1})\rangle} ~~\text{is} ~~S_p^*-\text{torsion}.
$$
\end{prop}

\begin{subsection}{The Derived Series Localized at $\mathcal{P}$}\label{sub:derivedserieslocalatp}

\

Finally we define the specific families of commutator series that we will use in our primary applications, using the method of Subsection~\ref{sec:seriesfromloc} and Definition~\ref{def:Sp}.

Fix an $n$-tuple $\mathcal{P}=(p_1(t),...,p_n(t))$ of non-zero elements of $\Q[t,t^{-1}]$. For each such $\mathcal{P}$ we now recursively define a functorial partial commutator series that we call the \textbf{(unrestricted) derived series localized at $\mathcal{P}$}. We also define, on the category of groups $\{G ~|~~G/G^{(1)}_r\cong \Z\}$, the \textbf{(polarized) derived series localized at $\mathcal{P}$}.

Given any group $G$, set $G^{(0)}_\mathcal{P}\equiv G$ and suppose inductively that $\gn_\mathcal{P}$ has been defined in such a way that $G^{(k)}_\mathcal{P}/G^{(k+1)}_\mathcal{P}$ is a torsion-free abelian group for each $k<n$. Then $G/\gn_\mathcal{P}$ is a PTFA group, so $\Q[G/\gn_\mathcal{P}]$ is a right Ore domain. Now consider $\G=G/\gn_\mathcal{P}$ and $A=G^{(n-1)}_\mathcal{P}/G^{(n)}_\mathcal{P}\lhd \G$. Note that $A$ is torsion-free abelian. Thus we may apply Definition~\ref{def:Sp}, to conclude that
$S_{p_n}$ is a right divisor set of $\Q[G/\gn_\mathcal{P}]$, where:
\begin{defn}\label{def:Sdefderivedlocalp} If $n\geq 1$,
$$
S_{p_n}=S_{p_n}(G)=\{ q_1(a_1)...q_r(a_r)~|~\widetilde{(p_n,q_j)}=1; q_j(1)\neq 0; a_j\in G^{(n-1)}_\mathcal{P}/G^{(n)}_\mathcal{P}\}.
$$
It follows that $\Q-\{0\}\subset S_{p_n}$ (take $q_j$ a non-zero constant). If $n=0$ we understand that $S_{p_n}=\{1\}$.
\end{defn}
We claim that $S_{p_n}$ is closed under the natural involution on $\Q[G/\gn_\mathcal{P}]$. For if $~\widetilde{(p_n,q_j)}=1$, then we claim that $~\widetilde{(p_n,q_j(t^{-1}))}=1$. For if not then $r_p^k=(r_q^{-1})^n$ for some non-zero roots of $p_n$ and $q_j$ and non-zero $k,n$. Thus  $r_p^k=(r_q)^{-n}$ which is a contradiction.

By Subsection~\ref{sub:commserieslocal} the sequence of right divisor sets $S_{p_n}$ defines a commutator series.
\begin{defn}\label{def:defderivedlocalp} The \textbf{(unrestricted) derived series localized at $\mathcal{P}$} is given by $G^{(0)}_{\mathcal{P}}\equiv G$ and for $n\geq 0$,
$$
\gnp_{\mathcal{P}}\equiv \ker \left(\gn_{\mathcal{P}}\to \frac{\gn_{\mathcal{P}}}{[\gn_{\mathcal{P}} ,\gn_{\mathcal{P}} ]}\otimes_{\Z [G/\gn_{\mathcal{P}}]}\Q[G/\gn_{\mathcal{P}}]S_{p_n}^{-1}\right).
$$
For a group $G$ with $\beta_1(G)=1$, we also define the \textbf{(polarized) derived series localized at $\mathcal{P}$} exactly as above, except that for $n=1$ we use $S_{p_1}^*$ (see Definition~\ref{def:Sp}) instead of $S_{p_1}$.
\end{defn}

We remark that, for either series, $G^{(0)}_\mathcal{P}=G$ and $G^{(1)}_\mathcal{P}=G^{(1)}_r$. This concludes the definition of our partially defined (up to $\gnp_\mathcal{P}$) commutator series dependent on $\mathcal{P}$.  By choosing an (infinite) sequence of polynomials, one may define an entire series.

\begin{thm}\label{thm:derivedlocpfunctorial} The (unrestricted) derived series localized at $\mathcal{P}$ is a functorial commutator series. The polarized derived series localized at $\mathcal{P}$ is functorial with respect to homomorphisms $f:G\to \pi$ for which $f_*:\Z\cong G/G^{(1)}_r\to \pi/\pi^{(1)}_r\cong \Z$ is either an isomorphism or the zero map. In particular it is weakly functorial.
\end{thm}
\begin{proof} We are given $\mathcal{P}=(p_1(t),...,p_n(t),...)$, a sequence of non-zero elements of $\Q[t,t^{-1}]$. First we consider the (unrestricted) derived series localized at $\mathcal{P}$.  Suppose $\psi:G\to B$ is a homomorphism. We show, by induction on $n$, that $\psi(\gn_\mathcal{P})\subset \bn_\mathcal{P}$. This holds for $n=0$ so suppose it holds for $n$. We will show that $\psi(\gnp_\mathcal{P})\subset \bnp_\mathcal{P}$. By Proposition~\ref{prop:preservSisfunctorial}, it suffices to verify that, for each $0\leq i\leq n$,
$\psi(S_{p_i}^G)\subset S_{p_i}^B$.  Since $i\leq n$, the induction hypothesis guarantees that $\psi$ induces a homomorphism of pairs
$$
\psi:(G/G^{(i)}_\mathcal{P}, G^{(i-1)}_\mathcal{P}/G^{(i)}_\mathcal{P})\to (B/B^{(i)}_\mathcal{P},B^{(i-1)}_\mathcal{P}/B^{(i)}_\mathcal{P})
$$
It then follows from the second part of Corollary~\ref{cor:divisorsets} that $\psi(S_{p_i}^G)\subset S_{p_i}^B$.

Now consider the (polarized) derived series localized at $\mathcal{P}$.  Suppose $\psi:G\to B$ is a homomorphism that induces either an isomorphism $\overline{\psi}:G/G^{(1)}_r\to B/B^{(1)}_r\cong \Z=\langle \mu\rangle$ or induces the zero map. Then $\overline{\psi}(\mu)=\pm\mu$ or $\overline{\psi}(\mu)=1$. The proof is the same as above, except that we must verify (for the case $n=1$) that $\overline{\psi}(S_{p_1}^{*,G})\subset S_{p_1}^{*,B}$. Recall that
$$
S_{p_1}^{*,G}=\{ q_1(\mu^{\pm 1})...q_r(\mu^{\pm 1})~|~(p_1,q_j)=1, ~q_j(1)\neq 0~\}.
$$
So, for any such $q$, either
$$
\overline{\psi}(q(\mu^{\pm 1}))=q(\overline{\psi}(\mu^{\pm 1}))=q(\mu^{\pm 1})\in S_{p_1}^{*,B},
$$
or
\begin{equation}\label{eq:fun}
\overline{\psi}(q(\mu^{\pm 1}))=q(\overline{\psi}(\mu^{\pm 1}))=q(1)=k \in S_{p_1}^{*,B},
\end{equation}
since the constant polynomial $k$ lies in $S_{p_1}^*$ if $k\neq 0$.
Thus $\overline{\psi}(S_{p_1}^{*,G})\subset S_{p_1}^{*,B}$.
\end{proof}

Note that the step ~\ref{eq:fun} fails if, say, $\overline{\psi}(\mu)=\mu^2$.

The following basic result is useful.

\begin{prop}\label{prop:idempotency} If $\phi :A \to B$ is surjective and $\ker\phi \subset A^{(m)}_{\sP}$ then $\phi$ induces
isomorphisms $A/A^{(n)}_{\sP} \cong B/B^{(n)}_{\sP}$ for all $n\leq m$. In particular, $(A/A^{(n)}_{\sP})^{(n)}_{\sP}=0$. For the polarized derived series localized at $\mathcal{P}$ we must also assume that $\phi$ induces an isomorphism $A/A^{(1)}_r\cong B/B^{(1)}_r\cong\Z$.
\end{prop}
\begin{proof}[Proof of Proposition~\ref{prop:idempotency}] By induction we assume $\phi$ induces isomorphisms $A/A^{(i)}_{\sP} \cong B/B^{(i)}_{\sP}$ for each $i\leq n$ for some $n$ such that $0\leq n<m$. By functoriality (Theorem~\ref{thm:derivedlocpfunctorial}), $\phi\left(A^{(i)}_{\sP}\right)\subset B^{(i)}_{\sP}$ for \emph{any} $i$ so $\phi$ induces an epimorphism $A/A^{(n+1)}_{\sP} \to B/B^{(n+1)}_{\sP}$. We need to show this is injective to complete the proof.  We claim that $\phi\left(A^{(i)}_{\sP}\right)= B^{(i)}_{\sP}$ for any $i\leq n$, because for any $b\in B^{(i)}_{\sP}$, since $\phi$ is surjective, there is some $a\in A$ such that $\phi(a)=b$ and, by the inductive hypothesis, it follows that $a\in A^{(i)}_{\sP}$. Therefore  $\phi\left([A^{(n)}_{\sP},A^{(n)}_{\sP}]\right)= [B^{(n)}_{\sP},B^{(n)}_{\sP}]$, a fact we use below.  Suppose $a\in A$ such that $\phi (a)=b\in B^{(n+1)}_{\sP}$. By the inductive hypothesis, $a\in A^{(n)}_{\sP}$.
Thus we have the following commutative diagram,
$$
\begin{diagram}\label{diagram-propSfunct2}\dgARROWLENGTH=1.2em
\node{\an_{\sP}}\arrow{s,r}{\phi}\arrow{e,t}{\pi_A}\node{\frac{\an_{\sP}}{[\an_{\sP} ,\an_{\sP} ]}} \arrow{s,r}{\phi_*}\arrow{e,t}{}\node{\frac{\an_{\sP}}{[\an_{\sP} ,\an_{\sP} ]}\otimes_{\Z [A/\an_{\sP}]}\Q[A/\an_{\sP}](S_{\sP}(A))^{-1}}
\arrow{s,r}{\phi_*}\\
\node{\bn_{\sP}}\arrow{e,t}{\pi_B}\node{\frac{\bn_{\sP}}{[\bn_{\sP} ,\bn_{\sP} ]}} \arrow{e,t}{}\node{\frac{\bn_{\sP}}{[\bn_{\sP} ,\bn_{\sP} ]}\otimes_{\Z [B/\bn_{\sP}]}\Q[B/\bn_{\sP}](S_{\sP}(B))^{-1}}
\end{diagram}
$$
where $b$ is in the kernel of the bottom composition and we need to show that $a$ is in the kernel of the top composition. By Proposition~\ref{prop:char*}, $\pi_B(b)$ is $\tilde{t}$-torsion where $\tilde{t}\in\tilde{S}_{p_n}(B)$. Thus there is a non-zero rational $r$ such that $\tilde{t}r\in S_{p_n}(B)$. Since $\phi$ is surjective,
$\phi(S_{p_n}(A))=S_{p_n}(B)$, by the last part of Corollary~\ref{cor:divisorsets}. Hence $\tilde{t}r=\phi(s)$ for some $s\in S_{p_n}(A)$. Then
$$
\phi_*(\pi_A(a)s)=\phi_*(\pi_A(a)(\phi(s))=\pi_B(b)\tilde{t}r=0.
$$
Since $\pi_A$ is surjective there is some $x\in \an_{\sP}$ such that $\pi_A(x)=\pi_A(a)s$. Therefore
$$
\pi_B(\phi(x))=\phi_*(\pi_A(x))=\phi_*(\pi_A(a)s)=0.
$$
This implies that
$$
\phi(x)\in [B^{(n)}_{\sP},B^{(n)}_{\sP}]=\phi\left([A^{(n)}_{\sP},A^{(n)}_{\sP}]\right).
$$
Therefore $x=zc$ where $z\in \ker(\phi)$ and $c\in [A^{(n)}_{\sP},A^{(n)}_{\sP}]$.  Thus
$$
\pi_A(a)s=\pi_A(x)=\pi_A(zc)=\pi_A(z).
$$
Since $\ker\phi \subset A^{(m)}_{\sP}$ and $n+1\leq m$, $\ker\phi \subset A^{(n+1)}_{\sP}$. Thus $z\in A^{(n+1)}_{\sP}$. Hence, Proposition~\ref{prop:char*}, $\pi_A(z)$ is $\tilde{s}$-torsion where $\tilde{s}\in\tilde{S}_{p_n}(A)$. Let $r'$ be a non-zero integer  such that $sr'\in \tilde{S}_{p_n}(A)$. Then
$$
=\pi_A(a)(sr'\tilde{s})=(\pi_A(a)s)\tilde{s}r'=(\pi_A(z)\tilde{s})r'=0.
$$
Since $sr'\tilde{s}\in \tilde{S}_{p_n}(A)$, we have shown that $\pi_A(a)$ is $\tilde{S}_{p_n}(A)$-torsion. Hence by Proposition~\ref{prop:char*}, $a\in A^{(n+1)}_{\sP}$. This completes the proof of the first claim of Proposition~\ref{prop:idempotency}.

To prove that $(A/A^{(n)}_{\sP})^{(n)}_{\sP}=0$, apply the above to the $\phi:A\to A/A^{(n)}_{\sP}$.
\end{proof}

\end{subsection}

\begin{section}{von Neumann signature defects as obstructions}\label{sec:rhoinvs}

To each commutator series there exist obstructions, that arise as \emph{signature-defects}, that can assist in determining whether or not a given knot lies in a particular term of $\mathcal{F}^{*}$. Given a closed, oriented 3-manifold $M$, a discrete group $\G$, and a representation $\phi : \pi_1(M)
\to \G$, the \textbf{von Neumann
$\boldsymbol{\rho}$-invariant}, $\rho(M,\phi)$, was defined by Cheeger and Gromov by choosing a Riemannian metric and defining $\rho$ as the difference between the $\eta$-invariants of $M$ and its covering space induced by $\phi$. It can be thought of as an oriented homeomorphism invariant associated to an arbitrary regular covering space of $M$ ~\cite{ChGr1}. If $(M,\phi) = \partial
(W,\psi)$ for some compact, oriented 4-manifold $W$ and $\psi : \pi_1(W) \to \G$, then it is known that $\rho(M,\phi) =
\sigma^{(2)}_\G(W,\psi) - \sigma(W)$ where $\sigma^{(2)}_\G(W,\psi)$ is the
$L^{(2)}$-signature (von Neumann signature) of the equivariant intersection form defined on
$H_2(W;\mathbb{Z}\G)$ twisted by $\psi$ and $\sigma(W)$ is the ordinary
signature of $W$ ~\cite{LS}. Thus the $\rho$-invariants should be thought of as \textbf{signature defects}. They were first used to detect non-slice knots in ~\cite{COT}. For a more thorough discussion see ~\cite[Section 2]{CT}~\cite[Section 2]{COT2}. All of the coefficient systems $\G$ in this paper will be of the form $\pi/\pi^{(n)}_*$ where $\pi$ is the fundamental group of a space. Hence all such $\G$ will be PTFA. Aside from the definition, a few crucial properties that we use in this paper are:
\begin{prop}\label{prop:rhoprops}\

\begin{itemize}
\item [1.] If $\phi$ factors through $\phi': \pi_1(M)\to \G'$ where
$\G'$ is a subgroup of $\G$, then $\rho(M,\phi') = \rho(M,\phi)$.
\item [2.] If $\phi$ is trivial (the zero map), then $\rho(M,\phi) = 0$.
\item [3.] If $M=M_K$ is zero surgery on a knot $K$ and $\phi:\pi_1(M)\to \mathbb{Z}$ is the abelianization, then $\rho(M,\phi)$ is denoted $\boldsymbol{\rho_0(K)}$ and is equal to \textbf{the integral over the circle \\ of the Levine-Tristram signature function} of $K$ ~\cite[Prop. 5.1]{COT2}. Thus $\rho_0(K)$ is the average of the classical signatures of $K$.
\item [4.] If $K$ is a slice knot or link and $\phi:\pi_1(M_K)\to \G$ ($\G$ PTFA) extends over $\pi_1$ of a slice disk exterior then $\rho(M_K,\phi)=0$ by ~\cite[Theorem 4.2]{COT}.
    \item [5.] The von Neumann signature satisfies Novikov  additivity, i.e., if $W_1$ and $W_2$ intersect along a common boundary component then $\sigma^{(2)}_\G(W_1\cup W_2)=\sigma^{(2)}_\G(W_1)+\sigma^{(2)}_\G(W_2)$ ~\cite[Lemma 5.9]{COT}.
    \end{itemize}
\end{prop}

Generalizing the last property, we have:

\begin{thm}\label{thm:generalsignaturesobstruct} Suppose $*$ is a commutator series (no functoriality is required). Suppose $K\in \mathcal{F}_{n.5}^*$, so
the zero-framed surgery $M_K$ is $(n.5,*)$-solvable via $W$ as in Definition~\ref{def:Gnsolvable}. Let $G=\pi_1(W)$ and consider
$$
\phi:\pi_1(M_K)\to G\to G/\gnp_*\to \G,
$$
where $\G$ is an arbitrary PTFA group.
Then
$$
\sigma^{(2)}(W,\phi)-\sigma(W)=0=\rho(M_K,\phi).
$$
More generally, if $W\subset V$ is an effective $(n.5,*)$-solution with respect to $V$ and
$$
\phi:\pi_1(V)\to \pi_1(V)/\pi_1(V)^{(n+1)}_*\to \G,
$$
then
$$
\sigma^{(2)}(W,\phi)-\sigma(W)=0.
$$
\end{thm}
\begin{proof} The proof is identical, verbatim, to the proof of ~\cite[Theorem 4.2]{COT} which was done only for the derived series. We sketch the key points. It follows immediately from Definition~\ref{def:Gnsolvable} that $\sigma(W)=0$. Since $\G$ is a poly-(torsion-free-abelian) group, $\Q\G$ is a right Ore domain and hence admits a classical ring of quotients which is a skew field, $\mathcal{K}\G$ (See Section~\ref{sub:classlocal} and ~\cite[Proposition 2.5]{COT}). Thus the rank of a $\Q\G$-module can be defined. Moreover, in this case, $\sigma^{(2)}(W,\phi)$ is the von-Neumann signature of the (non-singular) equivariant intersection form on the free module $H_2(W;\mathcal{K}\G)$ ~\cite[Section 5]{COT}. Therefore it suffices to exhibit a half-rank submodule of $H_2(W;\Q\G)$ on which this intersection form vanishes. The next key point is that half of the basis $\{L_i,D_i|1\leq i\leq r$ for $H_2(W;\Z)$, namely the surfaces $\{L_i\}$, lifts to the covering space determined by the kernel of $\phi$. Hence these surfaces generate a submodule of $H_2(W;\Q\G)$. Since these surfaces and their translates are disjointly embedded surfaces with product neighborhoods the equivariant intersection form on $H_2(W;\Q\G)$ vanishes on the submodule generated by them. The remainder of the proof consists of showing that this submodule is indeed of half-rank (see ~\cite[Theorem 4.2]{COT}).
\end{proof}

\end{section}

\section{Distinguishing concordance classes using the polarized derived series localized at $\mathcal{P}$: Triviality}\label{sec:examplestrivial}

Let $\mathcal{P}=(p_1(t),...,p_n(t),...)$ be a sequence of non-zero elements of $\Q[t,t^{-1}]$. By subsection~\ref{sub:derivedserieslocalatp} and  Theorem~\ref{thm:derivedlocpfunctorial} the polarized derived series localized at $\mathcal{P}$, denoted $\{\gn_\mathcal{P}\}$ is a commutator series that is weakly functorial on the class of groups with $\beta_1=1$. Thus by Proposition~\ref{prop:functseriesfilters} there is a corresponding filtration, $\{\mathcal{F}^\mathcal{P}_{n}\}$, of the knot concordance group. In the remainder of this paper we will use the \textbf{polarized} derived series localized at $\mathcal{P}$, since it allows for slightly sharper results.

\begin{defn}\label{def:distinctP} Given $\mathcal{P}=(p_1(t),...,p_n(t))$ and $\mathcal{Q}=(q_1(t),...,q_n(t))$, we say that $\mathcal{P}$ is \textbf{strongly coprime} to $\mathcal{Q}$ if either $(q_1,p_1)=1$, or, for some $k>1$, $\widetilde{(q_k,p_{k})}=1$; and otherwise we say that $\mathcal{P}$ is \textbf{isogenous} to $\mathcal{Q}$.
\end{defn}

We now consider knots of the form $K=R^n_{\alpha_n}\circ...\circ R^1_{\alpha_1}(K_0)$ with $K_0\in  \mathcal{F}_{0}$ where each $R^i_{\alpha_i}$ is a doubling operator as in Definition~\ref{def:doublingop}.  Every such knot lies in $\mathcal{F}_{n}$ and hence in $\mathcal{F}_{n}^\mathcal{P}$ by repeated applications of Proposition~\ref{prop:operatorsact}. Let $\mathcal{Q}=(q_n(t),...,q_1(t))$ where $q_i(t)$ is the order of $\alpha_i$ in $\mathcal{A}(R^i)$, the classical rational Alexander module of $R^i$ (note the descending index). Our next result shows that if $\mathcal{Q}$ is strongly coprime to $\mathcal{P}$ then $K\in \mathcal{F}_{n+1}^\mathcal{P}$. Thus the only knots of this form that (possibly) survive in
$$
\frac{\mathcal{F}_{n}}{\mathcal{F}_{n.5}^\mathcal{P}}
$$
are those wherein $\mathcal{Q}$ is isogenous to $\mathcal{P}$. This justifies thinking of $\mathcal{F}_{n}/\mathcal{F}_{n.5}^\mathcal{P}$ as localizing $\mathcal{F}_{n}/\mathcal{F}_{n.5}$ at $\mathcal{P}$.

\begin{thm}\label{thm:qliesinpnplusone} Suppose $K=R^n_{\alpha_n}\circ...\circ R^1_{\alpha_1}(K_0)$ with Arf$(K_0)=0$. Let $\mathcal{Q}=(q_n(t),...,q_1(t))$ be the sequence of orders of the classes $(\alpha_n,...,\alpha_1 )$ in $(\mathcal{A}(R^n),...,\mathcal{A}(R^1))$ . If  $\mathcal{Q}$ is strongly coprime to $\mathcal{P}$ then $K\in \mathcal{F}^\mathcal{P}_{n+1}$. Thus $K=0$ in $\mathcal{F}_{n}/\mathcal{F}^\mathcal{P}_{n.5}$.
\end{thm}

\begin{proof} For simplicity of notation we recursively define $K_1=R^1_{\alpha_1}(K_0)$, ..., $K_{i}=R^i_{\alpha_i}(K_{i-1})$ for $i=1,\dots,n$ and abbreviate $K=K_n=R^n_{\alpha_n}(K_{n-1})$. Recall from Figure~\ref{fig:mickey} that, since $K_i=R^i_{\alpha_i}(K_{i-1})$, there is a cobordism $E_{i}$ whose boundary is the disjoint union of $M_{K_{i}}$, $-M_{K_{i-1}}$ and $-M_{R^i}$. Consider $X=E_n\cup E_{n-1}\cup...\cup E_1$, gluing  $E_i$ to $E_{i-1}$ along their the common boundary component $M_{K_{i-1}}$ (refer to Figure~\ref{fig:Thm6.2}). The boundary of $X$ is a disjoint union of $M_K$, $-M_{K_0}$ and $-M_{R^n},...,-M_{R^1}$. For $1\leq i\leq n$, let $S_i$ denote the exterior of any ribbon disk in $B^4$ for the ribbon knot $R^i$. Since Arf($K_0$)$=0$, $K_0\in  \mathcal{F}_{0}$ via some $V$. In fact, it is known that we can choose $V$ such that $\pi_1(V)\cong\Z$ so that the meridian $\mu_0$ of $K_0$ generates $\pi_1(V)$ ~\cite[Section 5]{COT2}. Gluing $V$ and all the the $S_i$ to $X$, we obtain a $4$-manifold, $Z$ as shown in Figure~\ref{fig:Thm6.2}. Note $\partial Z=M_{K}$.
\begin{figure}[htbp]
\setlength{\unitlength}{1pt}
\begin{picture}(350,250)
\put(-20,10){\includegraphics{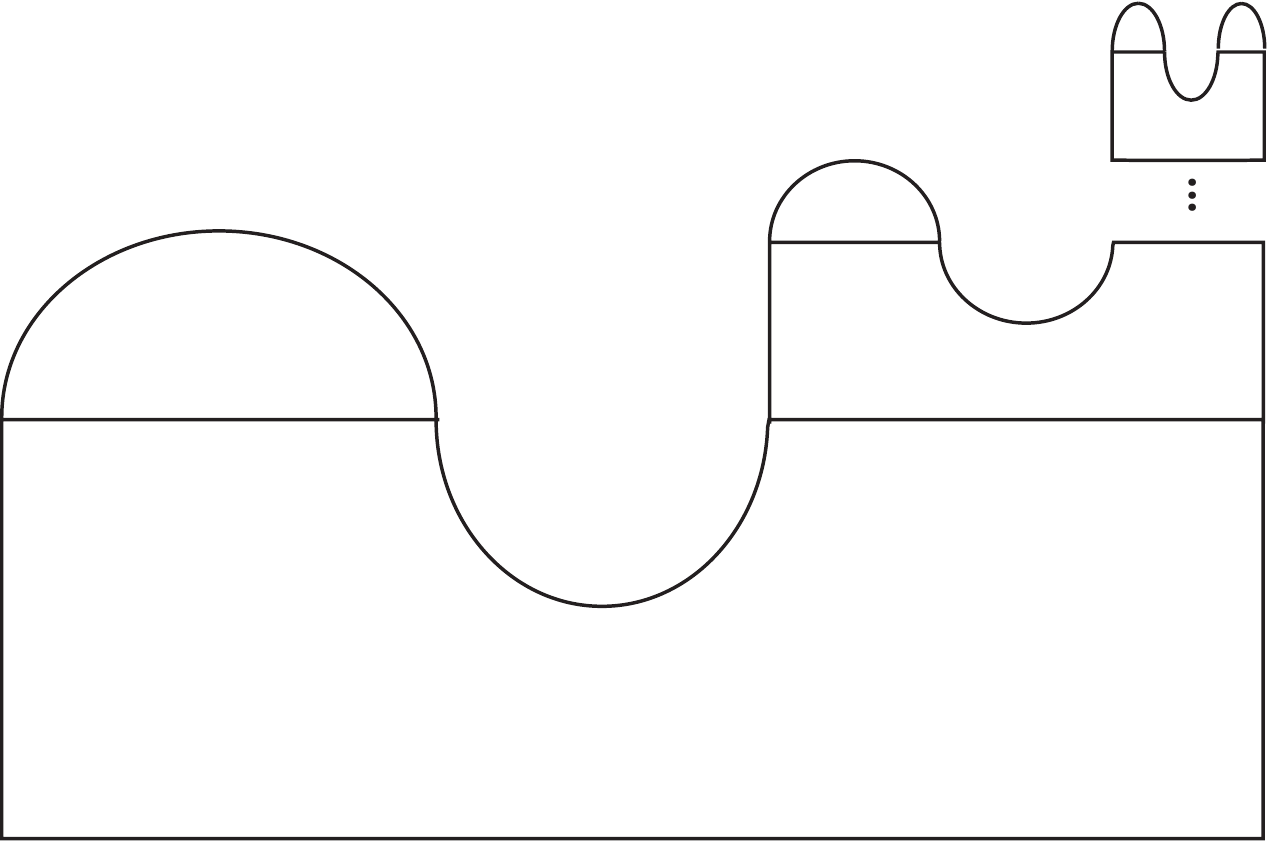}}
\put(35,120){$M_{R_n}$}
\put(36,152){$S_n$}
\put(145,45){$E_n$}
\put(145,-2){$M_K$}
\put(216,192){$S_{n-1}$}
\put(216,173){$M_{R_{n-1}}$}
\put(310,173){$M_{K_{{n-2}}}$}
\put(266,144){$E_{n-1}$}
\put(304,241){$S_1$}
\put(335,240){$V$}
\put(318,212){$E_1$}
\put(263,120){$M_{K_{n-1}}$}
%\put(-50,73){$Z~~\equiv$}
\end{picture}
\caption{$Z$}\label{fig:Thm6.2}
\end{figure}

We claim that, for any $\mathcal{Q}$,
\begin{equation}\label{eq:Knsolv}
K\in \mathcal{F}_{n} ~~\text{via} ~Z,
\end{equation}
while, if $\mathcal{Q}$ is strongly coprime to $\mathcal{P}$, then
\begin{equation}\label{eq:Kn+1solv}
K\in \mathcal{F}_{n+1}^\mathcal{P} ~~\text{via} ~Z.
\end{equation}

First, as in the proof of Proposition~\ref{prop:operatorsact}, a Mayer-Vietoris sequence implies that $H_2(Z)\cong H_2(V)$ since $H_2(S_i)=0$ (this analysis is carried out in detail in the proof of ~\cite[Proof of Theorem 7.1]{CHL3}). Since $K_{0}\in \mathcal{F}_{0}$ via $V$, $H_2(V)$ has a basis of connected compact oriented surfaces, $\{L_j,D_j|1\leq j\leq r_i\}$, satisfying the conditions of Definition~\ref{def:Gnsolvable}.
We claim that, with no hypothesis on $\mathcal{Q}$,
\begin{equation}\label{eq:mu0}
\mu_0\in \pi_1(Z)^{(n)}
\end{equation}
while if $\mathcal{Q}$ is strongly coprime to $\mathcal{P}$ then
\begin{equation}\label{eq:mu1}
\mu_0\in \pi_1(Z)^{(n+1)}_{\mathcal{P}}.
\end{equation}
Then, since $\mu_0$ generates $\pi_1(V)$, for any $\mathcal{Q}$,
$$
\pi_1(L_j)\subset\pi_1(V)\subset \pi_1(Z)^{(n)},
$$
while if $\mathcal{Q}$ is strongly coprime to $\mathcal{P}$,
$$
\pi_1(L_j)\subset\pi_1(V)\subset \pi_1(Z)^{(n+1)}_{\mathcal{P}},
$$
and similarly for $\pi_1(D_j)$. This would complete the verification of claims ~(\ref{eq:Knsolv}) and ~(\ref{eq:Kn+1solv}) since $\{L_j,D_j\}$ would then satisfy the criteria of Definition~\ref{def:Gnsolvable}. In the rest of the proof we will establish claims ~(\ref{eq:mu0}) and ~(\ref{eq:mu1}).

Let $G=\pi_1(Z)$. Let $\mu_i$ denote both the meridian of $K_i$ in $M_{K_i}\subset Z$ and its homotopy class in $G$. Let $\alpha_i$ denote the circle in $M_{R_i}\subset Z$ and, by abuse of notation, a push-off of this circle in $M_{K_i}\subset Z$ (the abuse is slight since these are isotopic in $E_i$ by Lemma~\ref{lem:mickeyfacts}). To complete the proof of claims ~(\ref{eq:mu0}) and ~(\ref{eq:mu1}) and hence the proof of Theorem~\ref{thm:generalqliesinpnplusone}, we need the following lemma.

\begin{lem}\label{lem:mui} For $0\leq i\leq n$,
$$
\mu_i\in G^{(n-i)}~ \text{and} ~\alpha_{i}\in G^{(n-i+1)};
$$
and if, for some $k$, $\alpha_k\in G^{(n-k+2)}_{\mathcal{P}}$ then for each $i$, $0\leq i< k$,
$$
\mu_i\in G^{(n-i+1)}_\mathcal{P} ~ \text{and} ~\alpha_{i}\in G^{(n-i+2)}_\mathcal{P}.
$$

\end{lem}
\begin{proof}
We will establish the first clause of the lemma by reverse induction on $i$. For $i=n$, clearly $\mu_n\in G^{(0)}\equiv G$ and since, by definition of a doubling operator, $\alpha_n$ lies in $\pi_1(M_{R_n})^{(1)}$, $\alpha_n\in G^{(1)}$. This is the base of the induction. Now assume that $\mu_{i}\in G^{(n-i)}$ and $\alpha_{i}\in G^{(n-i+1)}$. By property ($4$) of Lemma~\ref{lem:mickeyfacts}, $\mu_{i-1}$ is isotopic in $E_i$ to a push-off of $\alpha_{i}\subset M_{K_{i}}$.  It follows that
$$
\mu_{i-1}\in G^{(n-i+1)}.
$$
Since $\mu_{i-1}$ normally generates $\pi_1(M_{K_{i-1}})$ it follows that $\pi_1(M_{K_{i-1}})\subset G^{(n-i+1)}$. Thus,
$$
 \alpha_{i-1}\in [\pi_1(M_{K_{i-1}}),\pi_1(M_{K_{i-1}})]\subset G^{(n-i+2)}.
$$
This completes the inductive step and establishes the first clause of the lemma.

Now suppose that $\alpha_k\in G^{(n-k+2)}_{\mathcal{P}}$. By property ($4$) of Lemma~\ref{lem:mickeyfacts}, $\mu_{k-1}$ is isotopic in $Z$ to a push-off of $\alpha_{k}$. Thus
$$
\mu_{k-1}\in G^{(n-k+2)}_\mathcal{P}.
$$
Then, as above, it follows that $\pi_1(M_{K_{k-1}})\subset G^{(n-k+2)}_\mathcal{P}$ and so
$$
\alpha_{k-1}\in [\pi_1(M_{K_{k-1}}),\pi_1(M_{K_{k-1}})]\subset [G^{(n-k+2)}_\mathcal{P},G^{(n-k+2)}_\mathcal{P}]\subset G^{(n-k+3)}_\mathcal{P},
$$
where for the last inclusion we use Proposition~\ref{prop:commseriesprops}. This establishes the second clause of the lemma in the case that $i=k-1$. But now $\mu_{k-2}$ is isotopic to a push-off of $\alpha_{k-1}$ and the argument iterates. This establishes the second clause of the lemma.
 \end{proof}

Claim ~\eqref{eq:mu0} and ~\eqref{eq:Knsolv} follow immediately.

Returning to the proof of Theorem~\ref{thm:qliesinpnplusone}, we are now reduced to verifying claim ~\eqref{eq:mu1}. Recall that by hypothesis $\mathcal{Q}=(q_n(t),...,q_1(t))$ is strongly coprime to $ \mathcal{P}=(p_1(t),...,p_n(t))$. Hence there is some $k$, $1\leq k\leq n$, such that $q_k$ is strongly coprime to $p_{n-k+1}$ (or, in the case $n-k+1=1$, $q_k=q_n$ is coprime to $p_{n-k+1}=p_1$). Recall that $\alpha_k$ is $q_k(t)$ torsion in $\mathcal{A}(R_k)$. Thus the push-off of $\alpha_k$ is $q_k(t)$ torsion in the Alexander module of $K_k$. This can be interpreted in terms of the fundamental group of $M_{K_k}$ as follows ~\cite[p.174]{R}. Suppose $q_k(t)=\sum m_jt^j$. Then the fact that $\alpha_k$ is $q_k(t)$ torsion translates to the fact that
$$
\prod_j \mu_k^{-j}\alpha_k^{m_j}\mu_k^{j}\in [\pi_1(M_{K_k})^{(1)},\pi_1(M_{K_k})^{(1)}]
$$
since $t$ acts by conjugation by the meridian $\mu_k$. Thus
$$
\prod_j \mu_k^{-j}\alpha_k^{m_j}\mu_k^{j}\in [G^{(n-k+1)},G^{(n-k+1)}]\subset [G^{(n-k+1)}_\mathcal{P} ,G^{(n-k+1)}_\mathcal{P}]
$$
since we have shown in the proof of Lemma~\ref{lem:mui} that $\pi_1(M_{K_k})\in G^{(n-k)}$. We also know from Lemma~\ref{lem:mui} that $\alpha_k\in G^{(n-k+1)}\subset G^{(n-k+1)}_\mathcal{P}$. Therefore $\alpha_k$ represents an element in the module
$$
\frac{G^{(n-k+1)}_\mathcal{P}}{[G^{(n-k+1)}_\mathcal{P} ,G^{(n-k+1)}_\mathcal{P}]},
$$
that is annihilated by $q_k(\mu_k)\in \Z[G/G^{(n-k+1)}_\mathcal{P}]$ . By Lemma~\ref{lem:mui},  $\mu_k\in G^{(n-k)}_\mathcal{P}$ so
$$
q_k(\mu_k)\in \Z[G^{(n-k)}_\mathcal{P}/G^{(n-k+1)}_\mathcal{P}].
$$
If $n-k+1\neq 1$, since $q_k$ is strongly coprime to $p_{n-k+1}$,  by Definition~\ref{def:Sdefderivedlocalp}, $q_k(\mu_k)\in S_{p_{n-k+1}}$. In the case $n-k+1=1$, $q_n$ is coprime to $p_{1}$ and $\mu_n$ is the generator of $G/G^{(1)}_\mathcal{P}=G/G^{(1)}_r$. Thus, by Definition~\ref{def:Sdefderivedlocalp}, $q_n(\mu_n)\in S_{p_{1}}^*$. Therefore, in any case, by Definition~\ref{def:defderivedlocalp}, $\alpha_k\in G^{(n-k+2)}_\mathcal{P}$. Therefore, by the second clause of Lemma~\ref{lem:mui} (applied with $i=0$), $\mu_0\in G^{(n+1)}_\mathcal{P}$. This finishes the verification of claim ~\eqref{eq:mu1}, and hence the proof of Theorem~\ref{thm:qliesinpnplusone}.
\end{proof}

In the process we have recovered part of \cite[Theorem 7.1]{CHL3}, a slight strengthening of claim (\ref{eq:Knsolv}).

\begin{cor}\label{cor:Z-cap} Suppose $K=R^n_{\alpha_n}\circ...\circ R^1_{\alpha_1}(K_0)$ with $K_0\in  \mathcal{F}_{0}$. Then $K\in \mathcal{F}_{n}$ via $Z$ (see Figure~\ref{fig:Thm6.2}) which has the additional property that, for any for any PTFA coefficient system $\phi:\pi_1(Z)\to \G$ such that $\phi(\pi_1(Z)^{(n+1)})=1$
$$
\rho(M_K,\phi)=\sigma^{(2)}_\G(Z)-\sigma(Z)=c_\phi\rho_0(K)
$$
where $c_\phi=0$ if $\phi(\mu_0)=1$ and $c_\phi=1$ if $\phi(\mu_0)\neq 1$. We may also assume that $\pi_1(M_K)\to\pi_1(Z)$ is surjective.
\end{cor}

\begin{proof}[Proof of Corollary~\ref{cor:Z-cap}] We have seen in claim ~(\ref{eq:Knsolv}) that $K\in \mathcal{F}_{n}$ via $Z$ where $Z$ consists of a union of $E_i$'s, ribbon disk exteriors $S_i$ and the zero solution $V$ for $K_0$. It is known that any Arf invariant knot admits a $0$-solution $V$ with $\pi_1(V)\cong \Z$ ~\cite[Lemma 5.4]{COT2}. By property $4$ of Proposition~\ref{prop:rhoprops} (or Theorem~\ref{thm:generalsignaturesobstruct})
$$
\sigma^{(2)}_\G(S_i)-\sigma(S_i)=0.
$$
By ~\cite[Lemma 2.4]{CHL3},
$$
\sigma^{(2)}_\G(E_i)-\sigma(E_i)=0.
$$
By the additivity of the signatures (property $[5]$ of Proposition~\ref{prop:rhoprops}) this means
$$
\sigma^{(2)}_\G(Z)-\sigma(Z)=\sigma^{(2)}_\G(V)-\sigma(V)=\rho(\partial V,\phi)=\rho(M_{K_0},\phi).
$$
But from claim ~(\ref{eq:mu0}) we know that $\mu_0\in \pi_1(Z)^{(n)}$ implying that $\phi([\pi_1(M_{K_0}),\pi_1(M_{K_0}])=1$. Hence $\phi$ restricted to  $\pi_1(M_{K_0})$ factors through $\Z$ so, by properties $1$, $2$ and $3$ of Proposition~\ref{prop:rhoprops} $\rho(M_{K_0},\phi)$ is equal to  $\rho_0(K_0)$ (if $\phi(\mu_0)\neq 1$) or zero (if $\phi(\mu_0)=1$).

The surjectivity of the inclusion map on $\pi_1$ follows from the fact that $\pi_1(V)\cong \Z$, from known facts about ribbon disks and from property $1$ of Lemma~\ref{lem:mickeyfacts}.
\end{proof}

\begin{subsection}{Compositions of more general operators}\label{subsec:moregeneralops}

Theorem~\ref{thm:qliesinpnplusone} and Corollary~\ref{cor:Z-cap} hold for compositions of more general operators $R_{\vec\alpha}(-,\dots,-)$ since the proof makes no use of the fact that $\vec\alpha$ is a single circle. It is awkward to state the most general theorem without introducing a lot of burdensome notation, since a general operator requires a varying number of inputs. The hypotheses of the general result may be most easily understood in terms of the cobordism $Z$ constructed in Figure~\ref{fig:Thm6.2}. With regard to such a figure, let us say that $M_{R_n}$ occurs at height $n$ in the cobordism and $M_{R_{n-1}}$ occurs at height $n-1$, et cetera. Now consider $K$ a general composition of operators. Suppose $K=R_n(-,...,-)$ with $I_n$ inputs where each input is itself an iterated operator $R_{n-1}^i(-,...,-)$. The proof of Theorem~\ref{thm:qliesinpnplusone} for this general operator would start with the construction of a manifold $Z$. In the version of $Z$ for such a $K$ there would be a single copy of $M_{R_n}$ at height $n$ and then manifolds $M_{R^i_{n-1}}$, $1\leq i\leq I_n$, would occur at level $n-1$, et cetera. Consider all possible $n$-tuples $\mathcal{Q}=(q_n(t),q_{n-1}(t),...,q_1(t))$ where $q_n(t)$ is the order in $\mathcal{A}(R_n)$ of one of the circles $\vec\alpha$ parametrizing the inputs of $R_n$, and $q_{n-1}(t)$ is the order of one of the circles $\vec\alpha$ parametrizing the inputs of one of the ribbon knots at level $n-1$, et cetera. These are in $1-1$ correspondence with paths in $Z$ going from $M_K$ to a copy of a $0$-solution $V$ that pass through a sequence of $M_{K_{i}}$. They are also in $1-1$ correspondence with the number of $0$-solutions $V$ that will arise in the $Z$ corresponding to our given knot $K$. Then the proof of Theorem~\ref{thm:qliesinpnplusone} adapts verbatim to show:

\begin{thm}\label{thm:generalqliesinpnplusone}  If  every such $n$-tuple $\mathcal{Q}$ is strongly coprime to $\mathcal{P}$ then $K\in \mathcal{F}^\mathcal{P}_{n+1}$.
\end{thm}

\end{subsection}

\section{Distinguishing concordance classes using the polarized derived series localized at $\mathcal{P}$: Nontriviality}\label{sec:examplesnontrivial}

In this section we seek a non-triviality result to complement Theorem~\ref{thm:qliesinpnplusone}. We want to show that many knots survive in $\mathcal{F}_{n}/\mathcal{F}^\mathcal{P}_{n.5}$ if their associated sequence of Alexander polynomials is isogenous to $\mathcal{P}$.

Implicitly this will establish certain injectivity results for these operators, which we will detail in Section~\ref{sec:injectivity}. But of course it is certainly not true that any doubling operator $R_\alpha:\mathcal{C}\to \mathcal{C}$ is injective . For example if $\alpha$ bounds a disk in $S^3-R$ then $R_\alpha(K)=R=[0]$ for any $K$, that is the ``infection'' has no effect.
More generally, if the link $R\cup \alpha$ is concordant to a trivial link wherein the restriction of the concordance to $\alpha$ is isotopic to the standard unknotted annulus (recall $\alpha$ is unknotted) then $R_\alpha(K)$ is smoothly concordant to $R_{trivial}(K)=R=[0]$. Even more generally, if the slice knot $R$ admits a slice disk $\Delta$ for which $\alpha$ is merely \emph{null-homotopic} in $B^4-\Delta$ then $R_\alpha(K)$ is topologically slice for any $K$ ~\cite[Corollary 1.6]{CFT}. Therefore, in order for $R_\alpha$ to a non-zero operator, it is necessary (in the topological category) that, for \emph{no} slice disk for $R$, does $\alpha$ lie in the kernel of $\pi_1(S^3-R)\to \pi_1(B^4-\Delta)$. One of the only known ways to ensure this is to demand that $\mathcal{B}\ell_\Z^K(\alpha,\alpha)\neq 0$. So we shall define a \emph{robust operator} $R_\alpha$ and require that it has this property. This condition is close to being necessary. However we will also impose some other conditions that are not necessary, partly out of simplicity but partly due to  limitations in our current technology,

\begin{subsection}{Definition and Examples of Robust operators}\label{subsec:robust}

We first recall the definition of the \emph{first-order $L^2$-signatures of a knot} (from ~\cite{CHL4}). Suppose $K$ is a knot and let $G=\pi_1(M_K)$. Then
$$
\mathcal{A}(K)\equiv G^{(1)}/G^{(2)}\otimes_{\mathbb{Z}[t,t^{-1}]}\mathbb{Q}[t,t^{-1}]
$$
Each submodule $P\subset \mathcal{A}(K)$ corresponds to a unique metabelian quotient of $G$,
$$
\phi_P:G\to G/\tilde{P},
$$
by setting
$$
\tilde{P}\equiv \text{kernel}(G^{(1)}\to G^{(1)}/G^{(2)}\to \mathcal{A}(K) \to \mathcal{A}(K)/P).
$$
\noindent  Therefore to any such submodule $P$ there corresponds a real number, the Cheeger-Gromov von Neumann $\rho$-invariant, $\rho(M_K, \phi_P:G\to G/\tilde{P})$.

\begin{defn}~\cite[Section 4]{CHL4}\cite[Section 3]{CHL3}\label{defn:firstordersignatures} The \textbf{first-order} \textbf{$L^{(2)}$}-\textbf{signatures} of a knot $K$ are the real numbers
$\rho(M_K, \phi_P)$ where $P$ is an isotropic submodule ($P\subset P^\perp)$ of $\mathcal{A}(K)$ with respect to the classical Blanchfield form $\mathcal{B}\ell_\Z^K$. The first-order signature corresponding to $P=0$ is denoted $\rho^1(K)$. For any knot $K$, the set of all first-order signatures of $K$ will be denoted $\mathcal{FOS}(K)$.
\end{defn}

For example if $R$ is an algebraically slice knot and
$$
 \mathcal{A}(R)\cong \frac{\Q[t,t^{-1}]}{\langle\delta(t)\delta(t^{-1}) \rangle}
$$
where $\delta(t)$ is prime, then $\mathcal{A}(R)$ has $3$ isotropic submodules (in fact has just $3$ proper submodules) $P_0=0$, $P=\langle\delta(t) \rangle$ and $P=\langle\delta(t^{-1}) \rangle$ (the last two coincide if $\delta(t)\doteq \delta(t^{-1})$. Thus such an $R$ has $3$ first-order $L^{(2)}$-signatures.

Suppose $R$ is a slice knot, $\Delta$ is a slice disk for $R$ and $V=B^4-\Delta$. Then $\partial V=M_R$. Set
\begin{equation}\label{eq:defP}
P_\Delta=\text{ker}\left(H_1(M_K;\Q[t,t^{-1}])\to H_1(V;\Q[t,t^{-1}])\right).
\end{equation}
Then it is well-known that $P_\Delta$ is a Lagrangian for $\mathcal{B}\ell_\Z^K$ ($P=P^\perp$). In this case we say that the Lagrangian $P_\Delta$ \textbf{corresponds to the slice disk}, $\Delta$.

\begin{defn}\label{def:robust} A doubling operator $R_\alpha:\mathcal{C}\to \mathcal{C}$ is \textbf{robust} if $R$ is a ribbon knot such that
\begin{itemize}
\item [1.] The rational Alexander module of $R$ is generated by $\alpha$ and
 $$
 \mathcal{A}(R)\cong \frac{\Q[t,t^{-1}]}{\langle\delta(t)\delta(t^{-1}) \rangle}
$$
where $\delta(t)$ is prime (in particular $\delta(t)\neq 1$).
\item [2.] For each isotropic submodule $P\subset \mathcal{A}(R)$, either the first-order signature $\rho(M_R,\phi_P)$ is non-zero, or $P$ corresponds to a ribbon disk for $R$.
 \end{itemize}
The \textbf{Alexander polynomial of an operator} $R_\alpha$ is the Alexander polynomial of the ribbon knot $R$.
 \end{defn}

\begin{ex}\label{ex:robustoperators} We will construct examples of robust operators $\{R^{p_k}_\alpha~|~ k\geq 1\}$ wherein $\delta(t)=kt-(k+1)$, that is, whose Alexander polynomials are precisely the family discussed in detail in Example~\ref{ex:stronglycoprime}
$$
\{p_k(t)=(kt-(k+1))((k+1)t-k)~|~k\geq 1\}.
$$
Consider the knot given by the solid lines on the left-hand side of Figure~\ref{fig:robustoperators}. Here $-k$ symbolizes $k$ full negative twists between the two bands (below the circle labeled $\alpha$ one sees what we mean by a one-half negative twist between the bands). The $T_k$ means that an infection by a knot $T_k$ has been performed, that is, that the left-hand band has been tied into the shape of the knot $T_k$. The knot $T_k$ will be  either the right-handed trefoil or the unknot, a choice we will specify below.
\begin{figure}[htbp]
\setlength{\unitlength}{1pt}
\begin{picture}(327,151)
\put(-20,0){\includegraphics{ribbon_family_dashed2}}
\put(194,0){\includegraphics{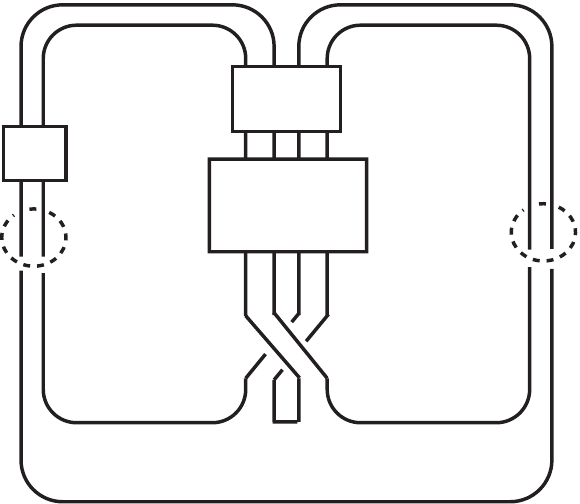}}
\put(54,113){$-k$}
\put(84,78){$\alpha$}
\put(272,83){$K$}
\put(269,113){$-k$}
\put(220,76){$\eta_+$}
\put(325,76){$\eta_-$}
\put(199,98){$T_k$}
\put(148,56){$R^{p_k}_\alpha(K)\equiv$}
\put(-17,106){$T_k$}
\put(-50,56){$R^{p_k}_\alpha\equiv$}
\end{picture}
\caption{Robust operators $R^{p_k}_\alpha$}\label{fig:robustoperators}
\end{figure}
The image of a knot $K$ under this operator is shown on the right-hand side of Figure~\ref{fig:robustoperators}. Note that $R^{p_k}_\alpha$ is a ribbon knot since, upon cutting open the left-hand band, it collapses into a $2$-component trivial  link. One easily checks that the Alexander polynomial of $R^{p_k}$ is $p_k(t)$. Since $\delta(t)$ is coprime to $\delta(t^{-1})$, the rational Alexander module is necessarily cyclic, and it can be shown that $\alpha$ represents a generator. This verifies condition $1$ of Definition~\ref{def:robust}. Now we examine condition $2$ of Definition~\ref{def:robust}. There are $3$ first-order signatures: corresponding to $P_0=0$ and $P_{\pm}=\langle\eta_{\pm} \rangle$. The circle $\eta_+$ normally generates the kernel of the inclusion map $\pi_1(S^3-R^{p_k})\to\pi_1(B^4-\Delta)$ for the ribbon disk obtained by cutting the $\eta_+$ band. Thus $P_+$ corresponds to a  ribbon disk. Hence $P_+$ satisfies condition ($2$) of Definition~\ref{def:robust}. Let $R^k$ denote the ribbon knot obtained from $R^{p_k}$ by setting $T_k=U$, the unknot. By the additivity results for the metabelian $\rho$ invariants ~\cite[Lemma 2.3, Example 3.3]{CHL3}:
\begin{eqnarray}
\rho(R^{p_k},\phi_{P_0})=\rho(R^{k},\phi_{P_0})+\rho_0(T_k),\\
\rho(R^{p_k},\phi_{P_-})=\rho(R^{k},\phi_{P_-})+\rho_0(T_k).
\end{eqnarray}
The term $\rho(R^{k},\phi_{P_0})$ is what we have called $\rho^1(R^k)$. Note that a ribbon disk, $\Delta_-$, in $B^4$ can be obtained for $R^k$ by cutting the right-hand band wherein, under the inclusion map on $\pi_1$, the circle $\eta_-$ goes to zero. Thus $\rho(R^{k},\phi_{P_-})=0$ since $\phi_{P_-}$ extends over the exterior of this ribbon disk. Hence we have
\begin{eqnarray}\label{eq:additivity1}
\rho^1(R^{p_k})\equiv\rho(R^{p_k},\phi_{P_0})=\rho^1(R^{k})+\rho_0(T_k),\\
\rho(R^{p_k},\phi_{P_-})=\rho_0(T_k)\label{eq:additivity2}.
\end{eqnarray}
Then there are two cases. In the first case, $\rho^1(R^{k})\neq 0$. In this case we take $T_k=U$ so $R^k=R^{p_k}$. Then $\rho^1(R^{p_k})=\rho^1(R^{k})\neq 0$ and $P_-$ corresponds to the ribbon disk $\Delta_-$. Thus in this case $R^{p_k}_\alpha=R^{k}_\alpha$ is robust. In the second case, $\rho^1(R^{k})=0$. In this case we take $T_k$ to be the right-handed trefoil and we have, from \eqref{eq:additivity1} and \eqref{eq:additivity2} that
$$
\rho^1(R^{p_k})=\rho_0(T_k)=\rho(R^{p_k},\phi_{P_-})\neq 0.
$$
Thus $R^{p_k}_\alpha$ is also robust in this case.
\end{ex}

\begin{ex}\label{ex:robustoperators2} More generally, we will show that, for any prime $\delta(t)$ with $\delta(1)=\pm 1$, there is a robust operator $R_\alpha$ whose Alexander polynomial is $p(t)\doteq \delta(t)\delta(t^{-1})$. By ~\cite{Ter}, given such a $\delta(t)$, there exists a ribbon knot $J$, obtained from fusing together a $2$-component trivial link with a single ribbon band, whose Alexander polynomial is $p(t)$. It follows that $\mathcal{A}(J)$ is cyclic and has either $2$ or $3$ proper submodules $P_0=0, P_+=\langle\delta(t) \rangle$ and $P_-=\langle\delta(t^{-1}) \rangle$ (in some cases $P_+=P_-$). Cutting the single band leads to a slice disk $\Delta$ to which, say, $P_+$ corresponds. Let $R$ be the knot obtained from $J$ by infection on a linking circle $\eta_+$ to the band using an auxiliary knot $K$. Then $\mathcal{A}(R)\cong \mathcal{A}(J)$ and $R$ is again a ribbon knot since cutting the band still yields a $2$-component trivial link. Hence $P_+=\langle\delta(t) \rangle=\langle\eta_+ \rangle$ corresponds to a ribbon disk. Moreover
$$
\rho(R,\phi_{P_0})=\rho(J,\phi_{P_0})+\rho_0(K),
$$
and if $P_+\neq P_-$ then
$$
\rho(R,\phi_{P_-})=\rho(J,\phi_{P_-})+\rho_0(K).
$$
Therefore, by choosing $K$ so that $\rho_0(K)$ is greater in absolute value than both $\rho(J,\phi_{P_0})$ and $\rho(J,\phi_{P_-})$ we can ensure that the first order signatures of $R$ corresponding to $P_0$ and $P_-$ are non-zero. Thus $R_\alpha$ is the desired robust operator, where $\alpha$ generates $\mathcal{A}(R)$.
\end{ex}

\end{subsection}

\begin{subsection}{The Main Theorems}\label{subsec:mainthms}

Now we can show that a composition of $n$ robust operators is a non-trivial operator, and in fact an injection when applied to certain collections of knots with independent classical signatures, even modulo $\mathcal{F}^\mathcal{P}_{n.5}$ as long as $\mathcal{P}$ corresponds to (the reverse of) the sequence of orders of the Alexander polynomials of the operators.

The following theorem also holds for the filtration induced by the unrestricted derived series localized at $\mathcal{P}$.

\begin{thm}\label{thm:robust} Suppose $R^i_{\alpha_i}$, $1\leq i\leq n$, are robust doubling operators and $\mathcal{P}=(p_1(t),...,p_n(t))$ is the sequence of orders of the classes $(\alpha_n,...,\alpha_1 )$ in $(\mathcal{A}(R^n),...,\mathcal{A}(R^1))$ (note the order of $\alpha_i$ is $p_{n-i+1}$). Suppose $\{K_0^j|~j\geq 1\}$ is an infinite set of Arf invariant zero knots such that the rational vector subspace of $\mathbb{R}$ spanned by $\{\rho_0(K_0^j)\}$ has trivial intersection with the rational span of $\mathcal{FOS}(R^1)$ (see Definition~\ref{defn:firstordersignatures}). Then
$\{K^j=K^j_n=R^n_{\alpha_n}\circ...\circ R^1_{\alpha_1}(K_0^j)|~j\geq 1\}$ is linearly independent in $\mathcal{F}_{n}/\mathcal{F}^\mathcal{P}_{n.5}$.
\end{thm}

We postpone the proof of Theorem~\ref{thm:robust} until after deriving from it some important applications.

\begin{cor}\label{cor:infgen} Collections $\{K^j_n\}$ satisfying the hypotheses of Theorem~\ref{thm:robust} exist; and any such collection is the basis of a $\Z^\infty$ subgroup of $\mathcal{F}_{n}/\mathcal{F}^\mathcal{P}_{n.5}\subset\mathcal{F}^\mathcal{P}_{n}/\mathcal{F}^\mathcal{P}_{n.5}$.
\end{cor}

\begin{proof}[Proof of Corollary~\ref{cor:infgen}] Clearly any such collection has the claimed property, so we need only establish that such collections exist. We saw in Examples~\ref{ex:robustoperators} and ~\ref{ex:robustoperators2} that collections of robust doubling operators $R^i_{\alpha_i}$  exist. It was shown in ~\cite[Proposition 2.6]{COT2} that there exists a set of Arf invariant zero knots $\{K_0^j~|~1\leq j< \infty\}$ such that $\{\rho_0(K_0^j)\}$ is $\mathbb{Q}$-linearly independent in $\mathbb{R}$. Since $\mathcal{FOS}(R^1)$ is a finite set when $R^1$ is a robust operator, the intersection of the subspaces spanned by $\mathcal{FOS}(R^1)$ and $\{\rho_0(K_0^j)\}$ is finite-dimensional. It follows easily that, after deleting a finite number of elements of $\{K_0^j\}$, we arrive at a collection that satisfies the hypotheses of Theorem~\ref{thm:robust}.
\end{proof}

In the following $\mathbb{P}_n$ is an arbitrary index set.

\begin{thm}\label{thm:fractal}Suppose $\{\mathcal{Q}_i~|~i\in \mathbb{P}_n\}$ is any collection of (pairwise) strongly coprime $n$-tuples $\mathcal{Q}_i=(q_{i_n}(t),\dots,q_{i_1}(t))$ of non-zero, non-unit polynomials; $\{\mathcal{R}_i~|~i\in \mathbb{P}_n\}$ is a collection of iterated robust doubling operators $\mathcal{R}_i\equiv R^{i_n}_{\alpha_n^i}\circ\dots\circ R^{i_1}_{\alpha_1^i}$ such that the sequence of orders of $(\alpha_n^i,\dots,\alpha_1^i)$ is $\mathcal{Q}_i$, and $\{\mathcal{K}^i|~i\in \mathbb{P}_n\}$ is a collection of sequences $\mathcal{K}^i=\{K^{i,j}|~1\leq j< \infty\}$ of Arf invariant knots such that, for each fixed $i$, the rational vector subspace of $\mathbb{R}$ spanned by $\{\rho_0(K^{i,j})\}$ has trivial intersection with the rational span of $\mathcal{FOS}(R^{i_1})$. Then the set
$$
\mathcal{J}=\{\mathcal{R}_i(K^{i,j})~|i\in \mathbb{P}_n, ~1\leq j< \infty\}
$$
is linearly independent in $\mathcal{F}_{n}/\mathcal{F}_{n.5}$ and hence a fortiori linearly independent in $\mathcal{F}_n\subset \mathcal{C}$. Therefore $\mathcal{F}_{n}/\mathcal{F}_{n.5}$ has distinct $\mathbb{Z}^\infty$ subgroups, indexed by the collection $\mathbb{P}_n$.
\end{thm}

\begin{cor}\label{cor:bigsubgroup} Infinite collections $\{\mathcal{Q}_i\}$  and $\{\mathcal{K}^i|~i\in \mathbb{P}_n\}$ exist that satisfy the hypotheses of Theorem~\ref{thm:fractal}. Therefore we have exhibited
\begin{equation}\label{eq:bigsubgroup2}
\bigoplus_{\substack{\mathbb{P}_n}}\Z^\infty\subset \frac{\mathcal{F}_{n}}{\mathcal{F}_{n.5}},
\end{equation}
 as claimed in ~\eqref{eq:bigsubgroup}.
\end{cor}
\begin{proof}[Proof of Corollary~\ref{cor:bigsubgroup}] As an example, take $\{\mathcal{Q}_i\}$ to be the set of all $n$-tuples of polynomials  $\{p_k\}$ of Example~\ref{ex:robustoperators}. As explained there, these polynomials are realized by robust doubling operators. The sets $\{\mathcal{K}^i|~i\in \mathbb{P}_n\}$ can be chosen just as in the proof of Corollary~\ref{cor:infgen}. Even more generally, one can take any maximal collection of pairwise strongly coprime $n$-tuples of polynomials. Unlike the set of all primes, such a set is not unique, since ``being strongly coprime'' (or even ``being coprime'') is not an equivalence relation.
\end{proof}

\begin{proof}[Proof that Theorem~\ref{thm:robust} implies Theorem~\ref{thm:fractal}] Assume that $\mathcal{J}$ satisfies some non-trivial equation in the abelian group $\mathcal{F}_{n}/\mathcal{F}_{n.5}$. Then there is some value, say $i_0\in \mathbb{P}_n$, that occurs non-trivially in this equation. Collecting, on one side of the equation, all the terms that correspond to $i_0$ (for varying $j$) we have an equation of the form $\tilde K=\tilde J$ in $\mathcal{F}_{n}/\mathcal{F}_{n.5}$ where
$$
\tilde K\equiv \#_{j=1}^\infty m_j\mathcal{R}_{i_0}(K^{{i_0},j})
$$
is a non-trivial finite sum and $\tilde J$ is a connected sum of knots (or their mirror images) each of which has the form $\mathcal{R}_{i}(K^{{i},j})$ where $i\neq i_0$. The distinguished value $i_0$ corresponds to a distinguished $n$-tuple $\mathcal{Q}_{i_0}$ that we denote by $\mathcal{P}$. Since $\mathcal{F}_{n.5}\subset \mathcal{F}_{n.5}^\mathcal{P}$ we may consider our equation in the quotient
$$
\mathcal{F}_{n}/\mathcal{F}_{n.5}^\mathcal{P}.
$$
The other values of $i$ correspond to $n$-tuples $\mathcal{Q}_{i}$ each of which is strongly coprime to $\mathcal{P}$. Therefore, by Theorem~\ref{thm:qliesinpnplusone}, each of the summands of $\tilde J$ lies in $\mathcal{F}^\mathcal{P}_{n+1}$ and so in $\mathcal{F}^\mathcal{P}_{n.5}$. Hence $\tilde J\in \mathcal{F}^\mathcal{P}_{n.5}$. Thus our projected equation reduces to $\tilde K=0$. This contradicts Theorem~\ref{thm:robust}. This contradiction finishes the proof of Theorem~\ref{thm:fractal}, modulo the proof of Theorem~\ref{thm:robust}.
\end{proof}

\begin{proof}[Proof of Theorem~\ref{thm:robust}] The proof is entirely analogous to the proof of ~\cite[Section 8,Step 4]{CHL3}, but here we must use our new results on the (polarized) derived series localized at $\mathcal{P}$. Here $n$ is fixed. We proceed by contradiction. Suppose that
$$
\tilde K\equiv \#_{j=1} m_jK^j\in \mathcal{F}^\mathcal{P}_{n.5}.
$$
By re-indexing and taking a mirror image if necessary, without loss of generality we may assume that $m_1>0$. As in the proof of Theorem~\ref{thm:qliesinpnplusone}, we let $K_1^j=R^1_{\alpha_1}(K_0^j)$, ..., $K_{i}^j=R^i_{\alpha_i}(K_{i-1}^j)$ and $K_n^j=K^j=R^n_{\alpha_n}(K_{n-1}^j)$. Throughout we abbreviate the zero-framed surgery $M_{K^j_n}$ by $M^j_n$.

First we will define a $4$-manifold such that $\partial W_n=M^1_n$. Suppose $\tilde{K}\in \mathcal{F}^\mathcal{P}_{n.5}$ via $V$ so $\partial V= M_{\tilde{K}}$. Let $C$ be the standard cobordism from $M_{\tilde {K}}$ to the disjoint union of  $m_j$ copies of $M_n^j$ for all $j$. Specifically
$$
\partial C= -M_{\tilde{K}}\coprod_jm_jM_n^j,
$$
where if $m_j<0$ we mean $|m_j|$ copies of $-M_n^j$. This cobordism is discussed in detail in ~\cite[p.113-116]{COT2}. Alternatively, note that any connected sum of knots $A\# B$ arises from an infection of $A$ along a meridian using the auxiliary knot $B$. From this point of view, the standard cobordism $C$ may be viewed as an instance of the cobordisms $E$ whose properties were detailed in Lemma~\ref{lem:mickeyfacts}. Now identify $C$ with $V$ along $M_{\tilde{K}}$.  Then cap off all of its boundary components except one copy of $M_n^1$ using copies of the $(n)$-solutions $\pm Z^j$ as provided by Corollary~\ref{cor:Z-cap}. The latter shall be called \textbf{$\mathcal{Z}$-caps}. Here there is a technical point concerning orientations: if $m_j>0$ then to the boundary component $M_n^j$ we must glue a copy of $-Z^j$ (and vice-versa). It is important in the proof that, since $m_1>0$, all the copies of $\pm Z^1$ occur as $-Z^1$ rather than $Z^1$. Let the result be denoted $W_n$ as shown schematically in Figure~\ref{fig:wn}. Note that $\partial W_n=M^1_n$.
\begin{figure}[htbp]
\setlength{\unitlength}{1pt}
\begin{picture}(108,106)
\put(0,0){\includegraphics{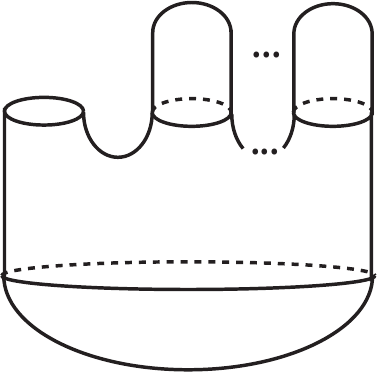}}
\put(50,42){$C$}
\put(50,10){$V$}
\put(89,85){$Z_{}^{j_k}$}
\put(48,85){$Z_{}^{j_1}$}
\put(-30,24){$M_{\tilde{K}}$}
\put(-10,27){\vector(1,0){9}}
\put(-28,71){$M_{n}^{1}$}
\put(-10,74){\vector(1,0){9}}
\end{picture}
\caption{$W_{n}$}\label{fig:wn}
\end{figure}

Now we construct a $4$-manifold $W_{n-1}$ such that $\partial W_{n-1}=M^1_{n-1}$. Since $K^1_n=R^n_{\alpha_n}(K^1_{n-1})$, there is a cobordism $E_{n}$ (as in Figure~\ref{fig:mickey}) whose boundary is the disjoint union of $-M^1_{{n}}$, $M^1_{{n-1}}$ and $M_{R^n}$. Consider $X=E_n\cup W_n$, gluing  $E_n$ to $W_n$ along their the common boundary component $M_{n}^1$ (see Figure~\ref{fig:cobordism1} with $i=n$). The boundary of $X$ is the disjoint union of $M^1_{n-1}$ and $M_{R^n}$. Let $S_n$ denote the exterior of a ribbon disk in $B^4$ for the ribbon knot $R^n$. In fact $R^n$ may have more than one ribbon disk and in this case we must choose carefully which one to employ. This choice will be made inductively in the midst of the proof (Remark~\ref{rem:whichS_i}). Let $W_{n-1}$ be obtained from $X$ by  capping off off the $M_{R^n}$ boundary component of $X$ using $S_n$. Thus $\partial W_{n-1}=M^1_{n-1}$.

Continuing in this way (refer to Figure~\ref{fig:cobordism1}), adjoining $E_i$, for $i=n,\dots,1$, and $S_i$ for $i=n,...,2$ we obtain $4$-manifolds $W_i$ such that $\partial W_{i}=M^1_{i}$ for $1\leq i\leq n$
and $\partial W_{0}=M^1_{0}\cup M_{R_1}$ since we will not cap off $M_{R^1}$ at the very last step. Recall that $M^1_{0}$ is zero surgery on $K_0^1$.

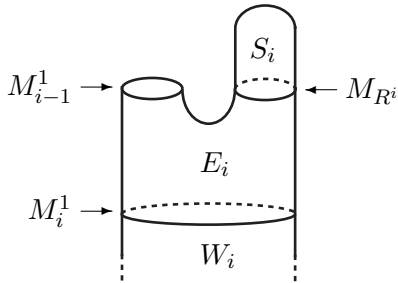
\begin{figure}[htbp]
\setlength{\unitlength}{1pt}
\begin{picture}(108,106)
\put(0,0){\includegraphics{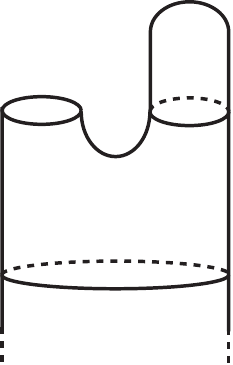}}
\put(30,42){$E_i$}
\put(30,7){$W_{i}$}
\put(49,85){$S_i$}
\put(82,73){\vector(-1,0){12}}
\put(85,70){$M_{R^i}$}
\put(-35,24){$M^1_{i}$}
\put(-15,27){\vector(1,0){12}}
\put(-43,71){$M_{i-1}^{1}$}
\put(-15,74){\vector(1,0){12}}
\end{picture}
\caption{$W_{i-1}=W_i\cup E_i \cup S_i$ for $i\neq 1$}\label{fig:cobordism1}
\end{figure}

The crucial result is that, with careful choices of the ribbon disk exteriors $S_i$, we can achieve certain subtle non-triviality results.

\begin{prop}\label{prop:familyofmickeyszero} With regard to the $4$-manifold $W_0$ constructed above (but see Remark~\ref{rem:whichS_i} of the proof for proper choice of the slice disk exteriors $S_i$) and letting $\pi=\pi_1(W_0)$
\begin{itemize}
\item [(1)] Under the inclusion $j:M_{0}^1\subset\partial W_0\to W_0$,

$$
j_*(\pi_1(M^1_{0}))\cong \mathbb{Z}\subset \pi^{(n)}/\pi^{(n+1)}_\mathcal{P};
$$
\item [(2)] Under the inclusion $j:M_{R^1}\subset\partial W_0\to W_0$,
$$
j_*(\pi_1(M_{R^1}))\cong \mathbb{Z}\subset \pi^{(n-1)}/\pi_\mathcal{P}^{(n)};
$$
\end{itemize}
\end{prop}

Before proving Proposition~\ref{prop:familyofmickeyszero}, we use it to finish the proof of Theorem~\ref{thm:robust}. Let $\phi$ denote $\pi\to\pi/\pi^{(n+1)}_\mathcal{P}\equiv \G$. We will compute
$$
\sigma^{(2)}(W_0,\phi)-\sigma(W_0).
$$
Recall that $W_0$ is a union of $V$, $C$, $E_n$ through $E_1$, $S_n$ through $S_2$ and the various $\mathcal{Z}$-caps. It follows from property $4$ of Proposition~\ref{prop:rhoprops} and Theorem~\ref{thm:generalsignaturesobstruct} respectively that
$$
\sigma^{(2)}(S_i,\phi)-\sigma(S_i)=0=\sigma^{(2)}(V,\phi)-\sigma(V).
$$
By ~\cite[Lemma 2.4]{CHL3},
$$
\sigma^{(2)}(E_i,\phi)-\sigma(E_i)=0,
$$
and
$$
\sigma^{(2)}(C,\phi)-\sigma(C)=0
$$
(the latter was also shown in ~\cite[Lemma 4.2]{COT2}). For each $\mathcal{Z}$-cap, $\pm Z^j$,
$$
\sigma^{(2)}(\pm Z^j,\phi)-\sigma(\pm Z^j)=\pm \rho(Z^j,\phi),
$$
which, by Corollary~\ref{cor:Z-cap}, is either zero or equal to $\pm\rho_0(K_0^j)$. Therefore
\begin{equation}\label{eq:1}
\rho(\partial W_0,\phi)=\sigma^{(2)}(W_0,\phi)-\sigma(W_0)=-\sum_{j}C^j\rho_0(K_0^j)
\end{equation}
where $C^1\geq 0$ (since $m_1-1\geq 0$), from which we have
\begin{equation}\label{eq:2}
\rho(M_0^1,\phi)+ \rho(M_{R^1},\phi)=\rho(\partial W_0,\phi)=-\sum_{j}C^j\rho_0(K^j_0).
\end{equation}
Then, by property $(1)$ of Proposition~\ref{prop:familyofmickeyszero}, the restriction of $\phi$ to $\pi_1(M_0^1)$ factors non-trivially through $\Z$. Hence, by properties $2$ and $3$ of Proposition~\ref{prop:rhoprops}, $\rho(M_0^1,\phi)=\rho_0(K^1_0)$. Consequently
\begin{equation}\label{eq:second}
(1+C^1)\rho_0(K_0^1)+\sum_{j>1}C^j\rho_0(K^j_0)=-\rho(M_{R^1},\phi_{R^1}).
\end{equation}
We claim that
\begin{equation}\label{eq:third}
\rho(M_{R^1},\phi_{R^1})\in \mathcal{FOS}(R^1).
\end{equation}
Granting this for the moment, we would then have
\begin{equation}\label{eq:fourth}
-(1+C^1)\rho_0(K_0^1)+\sum_{j>1}C^j\rho_0(K^j_0)\in \mathcal{FOS}(R^1).
\end{equation}
Since $1+C^1>0$  we would have expressed a non-zero  linear combination of $\{\rho_0(K^j_0)\}$ as an element of $\mathcal{FOS}(R^1)$ contradicting our choice of $\{K^j_0\}$.

To justify ~\eqref{eq:third}, note that by property $(2)$ of Proposition~\ref{prop:familyofmickeyszero},
\begin{equation}\label{eq:M_R}
j_*(\pi_1(M_{R^1}))\subset \pi^{(n-1)}.
\end{equation}
Let $G=\pi_1(M_{R^1})$ and let $\phi_R$ denote the restriction of $\phi$ to $G$. Then ~\eqref{eq:M_R} implies that $\phi_R$ factors through $G/G^{(2)}$.  We claim that kernel($\phi_R)\subset G^{(1)}$. For suppose that $x\in \text{kernel}(\phi_R)$ and $x=\mu^my$ where $\mu$ is a meridian of $R$ and $y\in G^{(1)}$. Then certainly $x$ is in the kernel of the composition
$$
\psi:G\overset{\phi_R}\longrightarrow \pi^{(n-1)}/\pi^{(n+1)}_\mathcal{P}\to\pi^{(n-1)}/\pi^{(n)}_\mathcal{P}.
$$
Moreover, by ~\eqref{eq:M_R}, $\phi_R(G^{(1)})\subset \pi^{(n)}_\mathcal{P}$ so $G^{(1)}$ is in the kernel of $\psi$. Therefore $\mu^m\in \ker\psi$ and the image of $\psi$ has order at most $m$. If $m\neq 0$ this contradicts property ($2$) of Proposition~\ref{prop:familyofmickeyszero}. Thus $m=0$ and kernel($\phi_R)\subset G^{(1)}$. Therefore $\phi_R$ is determined by the kernel, called $\tilde{P}$, of
$$
\bar\phi_R:G^{(1)}/G^{(2)}\cong \mathcal{A}_0(R^1)\to \text{image}(\phi_R).
$$
Moreover, by property $1$ of Proposition~\ref{prop:rhoprops},
$$
\rho(M_{R^1},\phi_{R^1})=\rho(M_{R^1},\bar\phi_R).
$$
Since $\tilde{P}$ is normal in $G/G^{(2)}$, it is preserved under  conjugation by a meridional element, implying that $\tilde{P}$ represents a \emph{submodule} $P\subset\mathcal{A}_0(R^1)$.  Since $R^1$ is a robust operator, the Alexander polynomial of $R^1$ is the product, $\delta(t)\delta(t^{-1})$, of two irreducible factors. Thus $\mathcal{A}_0(R^1)$ admits precisely $4$ submodules: $P_1=\mathcal{A}_0(R^1)$, $P_0=0, P_+=\langle\delta(t) \rangle$ and $P_-=\langle\delta(t^{-1} \rangle$. The first case is when $\phi_R$ factors through $\Z$. In this case $\rho(M_{R^1},\phi_R)=\rho_0(R^1)$, by property $3$ of Proposition~\ref{prop:rhoprops}, and hence vanishes since since $R^1$ is an algebraically slice knot. But, since $R^1$ is a ribbon knot, $0\in \mathcal{FOS}(R^1)$. The other cases are isotropic submodules so in all cases we have $\rho(M_{R^1},\phi_R)\in \mathcal{FOS}(R^1)$, verifying ~\eqref{eq:third}.

Therefore we have established ~\eqref{eq:fourth} and this contradiction finishes the proof of Theorem~\ref{thm:robust}, modulo the proof of Proposition~\ref{prop:familyofmickeyszero}.

\end{proof}

\end{subsection}

The proof of Proposition~\ref{prop:familyofmickeyszero} is accomplished inductively using the following. Clearly, the case $i=0$ of Proposition~\ref{prop:familyofmickeys} gives Proposition~\ref{prop:familyofmickeyszero}.

\begin{prop}\label{prop:familyofmickeys} With regard to the $4$-manifold $W_{i}$ constructed above (fixing $i$ and $n$) and letting $\pi=\pi_1(W_{i})$,

\begin{itemize}
\item [(1)] Under the inclusion $j:M_{i}^1\subset\partial W_{i}\to W_{i}$,
$$
j_*(\pi_1(M^1_{i}))\cong \mathbb{Z}\subset \pi^{(n-i)}/\pi^{(n-i+1)}_\mathcal{P};
$$
\item [(2)] in the case $i=0$, under the inclusion $j:M_{R^1}\subset\partial W_0\to W_0$,
$$
j_*(\pi_1(M_{R^1}))\cong \mathbb{Z}\subset \pi^{(n-1)}/\pi_\mathcal{P}^{(n)};
$$
\item [(3)] $W_i$ is an $(n,\mathcal{P})$-bordism for $\partial W_i$.
\end{itemize}
\end{prop}

An $(n,\mathcal{P})$-bordism will be defined below.

\begin{proof}[Proof of Proposition~\ref{prop:familyofmickeys}] We proceed by reverse induction on $i$. Consider the case $i=n$ (refer to Figure~\ref{fig:wn}). By ~(\ref{eq:firstderived}),  $\pi^{(1)}_\mathcal{P}=\pi^{(1)}_r$. Thus property ($1$) of Proposition~\ref{prop:familyofmickeys} is merely the statement that the inclusion $H_1(M_n^1;\mathbb{Q})\to H_1(W_n;\mathbb{Q})$ is injective . But this is easy to verify. Since $\tilde{K}\in \mathcal{F}^\mathcal{P}_{n.5}$ via $V$, by Definition~\ref{def:Gnsolvable} (part $1$), the inclusion-induced map
$$
j_*:~H_1(M_{\tilde K};\mathbb{Z})\to H_1(V;\mathbb{Z})
$$
is an isomorphism. It follows from duality that
$$
j_*:~H_2(M_{\tilde K};\mathbb{Z})\to H_2(V;\mathbb{Z})
$$
is the zero map. Similarly, each $H_1(M_n^j)\to H_1(\pm Z^j)$ is an isomorphism and by duality $H_2(M_n^j)\to H_2(Z^j)$ is the zero map. The integral homology of $C$ was analyzed in ~\cite[p. 113-114]{COT2}. From the latter we know that $H_1(C;\mathbb{Z})\cong \mathbb{Z}$, generated by any one of the meridians of any of the knots, and that $H_2(C;\mathbb{Z})$ is $\oplus_j H_2(M^j_n;\mathbb{Z})^{|m_j|}$. In particular $H_2(C)$ arises from its ``top'' boundary. Also the generator of $i_*(H_2(M_{\tilde K}))$ is merely the sum of the generators of the $H_2(M_n^j;\mathbb{Z})$ summands. Combining this information with several Mayer-Vietoris sequences, we deduce that $H_1(M_n^1;\mathbb{Z})\cong H_1(W_n;\mathbb{Z})$ as desired. The details of these elementary deductions were given in the proof of ~\cite[Proposition 8.2]{CHL3} so we will not repeat them here. This finishes the case $i=n$ of part $(1)$ of Proposition~\ref{prop:familyofmickeys}, which is the base of our induction. By the same token we also conclude that
$$
H_2(W_n;\mathbb{Z})/j_*(H_2(\partial W_n;\mathbb{Z}))\cong H_2(W_n;\mathbb{Z})\cong H_2(V;\mathbb{Z})\oplus_{\mathbb{\mathcal{Z}}-caps}H_2(Z_n^j;\Z).
$$
To establish part $(3)$ (for $i=n$) we need more properties of $W_n$. Since $\tilde{K}\in \mathcal{F}^\mathcal{P}_{n.5}\subset \mathcal{F}^\mathcal{P}_{n}$ via $V$, $H_2(V)$ admits a basis given by $\{L^V_i,D^V_i\}$, a collection of surfaces satisfying Definition~\ref{def:Gnsolvable} for $n$. That is $V$ is an $(n,\mathcal{P})$-solution. Similarly each $\mathcal{Z}$-cap $Z^j$ is an ($n$)-solution, so each $H_2(Z^j)$ admits such a basis $\{L^j_i,D^j_i\}$ (using $\mathcal{F}_n\subset\mathcal{F}^\mathcal{P}_{n})$. The union of these bases is a basis for $H_2(W_n)$. Since the polarized derived series localized at $\mathcal{P}$ is weakly functorial, and by the naturality of the intersection form with twisted coefficients, the union of these surfaces exhibits $W_n$ as an $(n,\mathcal{P})$-solution for $M^1_n$. As we shall see below, an $(n,\mathcal{P})$-solution is an $(n,\mathcal{P})$-bordism, so we are done.

This finishes the case $i=n$ of Proposition~\ref{prop:familyofmickeys}, which is the base of our induction. To proceed to analyze the other $W_i$ we need a generalization of an $(n,\mathcal{P})$-solution. We also need to establish strong properties for such $4$-manifolds.
So we take a not-so-brief break from the proof of Proposition~\ref{prop:familyofmickeys}.

\begin{subsection}{(n,$\mathcal{P}$)-bordisms}\label{subsec:npbordisms}\

For the purposes of this subsection, $\mathcal{P}$ will denote an \textbf{arbitrary commutator series} (that need not be functorial). We retain the $\mathcal{P}$ notation in order to emphasize the current application.

The following simultaneously generalizes the notion of an $(n,\mathcal{P})$-solution (Definition~\ref{def:Gnsolvable}) and the notion of an $n$-bordism ~\cite[Section 5]{CHL3}.

\begin{defn}\label{def:npbordism} A compact spin smooth $4$-manifold $W$ is an \textbf{$(n,\mathcal{P})$-bordism}  for $\partial W$ if
\begin{itemize}
\item [2.] $H_2(W;\Z)/H_2(\partial W;\Z)$ has a basis consisting of connected compact oriented surfaces, $\{L_i,D_i|1\leq i\leq r$,  embedded in $W$ with trivial normal bundles, that are pairwise disjoint except that, for each $i$, $L_i$ intersects $D_i$ once transversely  with positive sign.
\item [3.] for each i, $\pi_1(L_i)\subset \pi_1(W)^{(n)}_\mathcal{P}$
and $\pi_1(D_i)\subset \pi_1(W)^{(n)}_\mathcal{P}$.
\end{itemize}
We call it an \textbf{$(n.5,\mathcal{P})$-bordism}  for $\partial W$ if in addition,
\begin{itemize}
\item [4.] for each i, $\pi_1(L_i)\subset \pi_1(W)^{(n+1)}_\mathcal{P}$
\end{itemize}
\end{defn}

An $(n,\mathcal{P})$-solution is \emph{a fortiori} an $(n,\mathcal{P})$-bordism. But an $(n,\mathcal{P})$-bordism need not have connected boundary and the inclusion map from the boundary does not necessarily induce an isomorphism on $H_1$.

We now state the crucial homological properties of $(n,\mathcal{P})$-bordisms that were previously established for $n$-bordisms in ~\cite[Section 5]{CHL3}. When the proofs are identical to those of ~\cite[Section 5]{CHL3}, they are not repeated here. We should point out that the rank hypothesis below on $H_1(M_i;\mathbb{Z}\Lambda)$ is \emph{always satisfied} if $\beta_1(M_i)=1$ (by ~\cite[Proposition 2.11]{COT}), which will always be the case in this paper.

\begin{lem}\label{lem:exact} Suppose $\mathcal{P}$ is an arbitrary commutator series, $W$ is a $(k,\mathcal{P})$-bordism and
$\phi:\pi_1(W)\ra\Lambda$ is a non-trivial coefficient system where
$\Lambda$ is a PTFA group with $\phi(\pi_1(W)^{(k)}_\mathcal{P})=1$. Let $\mathcal{R}$ be an Ore localization of $\mathbb{Z}\Lambda$ so $\mathbb{Z}\Lambda\subset\mathcal{R}\subset \mathcal{K}\Lambda$. Suppose, for each component $M_i$ of $\partial W$ for which $\phi$ restricted to $\pi_1(M_i)$ is nontrivial, that $\text{rank}_{\mathbb{Z}\Lambda}H_1(M_i;\mathbb{Z}\Lambda)=\beta_1(M_i)-1$. Then
\begin{itemize}
\item [1.] The $\mathbb{Q}$-rank of $(H_2(W)/j_*(H_2(\partial W))$ is equal to the $\mathcal{K}\Lambda$-rank of $H_2(W;\mathcal{R})/I$ where
$$
I=\text{image}(j_*(H_2(\partial W;\mathcal{R})\to H_2(W;\mathcal{R}))).
$$
and
\item [2.]
$$
TH_2(W,\partial W;\mathcal{R})\xrightarrow{\partial}TH_1(\partial W;\mathcal{R})\xrightarrow{j_{\ast}} TH_1(W;\mathcal{R})
$$
is exact, where $T\mathcal{M}$ denotes the $\mathcal{R}$-torsion submodule of the $\mathcal{R}$-module $\mathcal{M}$.
\end{itemize}
\end{lem}

\begin{proof} The proof is (verbatim) identical to that of the corresponding result for $n$-bordisms \cite[Lemma~5.10]{CHL3}, with the rational derived series being replaced by the arbitrary functorial commutator series. The key point is that if $L$ is, say, a surface that is part of a Lagrangian from Definition~\ref{def:Gnsolvable} for which $\pi_1(L)\subset \pi_1(W^{(k)}_\mathcal{P})$ then $\phi(\pi_1(L))=0$ so $L$ can be part of a Lagrangian with $\Q\Lambda$ or $\mathcal{R}$ coefficients.
\end{proof}

We also have  a partial generalization of our Theorem~\ref{thm:generalsignaturesobstruct} and a generalization of ~\cite[Theorem 5.9]{CHL3}. Once again the rank hypothesis is always satisfied if $\beta_1(M_i)=1$, which will always be the case in the present paper. Once again the proof is identical to the one for ~\cite[Theorem 5.9]{CHL3}.

\begin{thm}\label{thm:sigvanishesbordism} Suppose $\mathcal{P}$ is an arbitrary commutator series,  $W$ is an $(n+1,\mathcal{P})$-bordism and
$\phi:\pi_1(W)\ra\G$ is a non-trivial coefficient system where
$\G$ is a PTFA group and $\phi(\pi_1(W^{(n+1)}_\mathcal{P}))=1$. Suppose for each component $M_i$ of $\partial W$ for which $\phi$ restricted to $\pi_1(M_i)$ is nontrivial, that $\text{rank}_{\mathbb{Z}\Lambda}H_1(M_i;\mathbb{Z}\G)=\beta_1(M_i)-1$. Then
$$
\rho(\partial W,\phi)= 0.
$$
\end{thm}

Higher-order Alexander modules and higher-order linking forms for classical knot exteriors and for closed $3$-manifolds with $\beta_1(M)=1$ were introduced in ~\cite[Theorem 2.13]{COT} and further developed in ~\cite{C} and ~\cite{Lei1}. These were defined on the so called higher-order Alexander modules $TH_1(M;\mathcal{R})$. These modules are not necessarily of homological dimension one.

\begin{thm}\label{thm:blanchfieldexist}~\cite[Theorem 2.13]{COT} Suppose $M$ is a closed, connected, oriented $3$-manifold with $\beta_1(M)=1$ and $\phi:\pi_1(M)\to \Lambda$ is a PTFA coefficient system. Suppose $\mathcal{R}=(\mathbb{Z}\Lambda)S^{-1}$ is an Ore localization where $S$ is closed under the natural involution on $\mathbb{Z}\Lambda$. Then there is a (possibly singular) linking form:
$$
\mathcal{B}l^M_{\mathcal{R}}: TH_1(M;\mathcal{R})\to (TH_1(M;\mathcal{R}))^{\#}\equiv \overline{\Hom _{\mathcal{R}}(TH_1(M;\mathcal{R}), \mathcal{K}\Lambda/\mathcal{R})}.
$$
\end{thm}

One needs $S$ closed under the involution in order that $\Q\La\hookrightarrow\mathcal{R}$ be a map of rings with involution. When $\beta_1(M)=1$, as long as the coefficient system is non-trivial, $H_1(M;\mathcal{R})$ is a torsion module and hence $TH_1(M;\mathcal{R})=H_1(M;\mathcal{R})$ ~\cite[Proposition 2.11]{COT}.

From the above follows another generalization of key results of ~\cite[Theorem 4.4]{COT}~\cite[Theorem 6.3]{CHL3} concerning solvability and null-bordism whose generalization to null-$\mathcal{P}$-bordism will be a crucial new ingredient in our proofs. Again, the rank hypothesis is automatically satisfied if $\beta_1(M_i)=1$. The proof is identical to that of ~\cite[Theorem 6.3]{CHL3}.

\begin{thm}\label{thm:selfannihil} Suppose $\mathcal{P}$ is an arbitrary commutator series, $W$ is a $(k,\mathcal{P})$-bordism and
$\psi:\pi_1(W)\ra\Lambda$ is a non-trivial coefficient system where
$\Lambda$ is a PTFA group and $\psi(\pi_1(W)^{(k)}_\mathcal{P})=1$. Let $\mathcal{R}=(\mathbb{Q}\Lambda)S^{-1}$ be an Ore localization where $S$ is closed under involution. Suppose that, for each component $M_i$ of $\partial W$ for which $\psi$ restricted to $\pi_1(M_i)$ is nontrivial, that $\text{rank}_{\mathbb{Z}\Lambda}H_1(M_i;\mathbb{Z}\Lambda)$ is $\beta_1(M_i)-1$. If $P$ is the kernel of the inclusion-induced map
$$
TH_1(\partial W;\mathcal{R})\xrightarrow{j_{\ast}} TH_1(W;\mathcal{R}),
$$
then $P\subset P^\perp$ with respect to the Blanchfield form on $TH_1(\partial W;\mathcal{R})$.
\end{thm}

In many important situations (including the ones with which we will be here concerned), we will be considering $M_K$, the zero surgery on a knot $K$,  endowed with a coefficient system $\phi:\pi_1(M_K)\to \Lambda$ that factors non-trivially through $\mathbb{Z}$, the abelianization. In this special case the higher-order Alexander module of $M_K$ and the higher-order Blanchfield form $\mathcal{B}l^K_{\mathcal{R}}$ are merely the classical Blanchfield form on the classical Alexander module, ``tensored up''. What is meant by this is the following. Since $\phi$ is both nontrivial and factors through the abelianization, the induced map $\text{image}(\phi)\equiv\mathbb{Z}\hookrightarrow \Lambda$ is an embedding so it induces embeddings
$$
\phi:\mathbb{Q}[t,t^{-1}]\hookrightarrow \mathbb{Q}\Lambda\hookrightarrow (\mathbb{Q}\Lambda)S^{-1},~\text{and}~~ \phi:\mathbb{Q}(t)\hookrightarrow \mathcal{K}\Lambda.
$$
Since $\Z\subset \La$, $\Q \La$ is a free and hence flat left $\Q[t,t^{-1}]$-module. Any Ore localization, $(\Q \La)S^{-1}$, is a flat left $\Q \La$-module ~\cite[Proposition II.3.5]{Ste}. Consequently, the homology modules of $M_K$ are simply the classical homology modules of $M_K$, ``tensored up'':
$$
H_*(M_K;\mathcal{R})\cong H_1(M_K;\mathbb{Q}\La)\otimes_{\mathbb{Q}\La} \mathcal{R} \cong H_*(M_K;\mathbb{Q}[t,t^{-1}])\otimes_{\mathbb{Q}[t,t^{-1}]}\mathcal{R}.
$$
Since $\Q [t,t^{-1}]$ is a PID, each of the classical modules $H_*(M_K;\mathbb{Q}[t,t^{-1}])$ is finitely generated and has homological dimension $1$. Moreover each such module has a finite presentation. Since tensor product with a flat module is an exact functor, the above remarks (on flatness) imply that each of the modules $H_*(M_K;\mathcal{R})$ has a finite presentation and hence is finitely generated and has homological dimension $1$. In particular the higher-order Alexander module, $H_1(M_K;\mathcal{R})$, decomposes as
\begin{equation}\label{eq:decompmodule}
H_1(M_K;\mathcal{R})\cong \mathcal{A}(K)\otimes_{\mathbb{Q}[t,t^{-1}]}(\mathbb{Q}\Lambda)S^{-1},
\end{equation}
where $(\mathbb{Q}\Lambda)S^{-1}$ is a $\mathbb{Q}[t,t^{-1}]$-module via the map $t\to \phi(\mu_K)$ ~\cite[Theorem 8.2]{C}.

Moreover the higher-order Blanchfield form on $H_1(M_K;\mathcal{R})$ (see Theorem~\ref{thm:blanchfieldexist}) has a corresponding decomposition:
\begin{equation}\label{eq:tensorup}
\mathcal{B}l_{\mathcal{R}}^K(x\otimes 1,y\otimes 1)=\ov\phi(\mathcal{B}l^K_0(x,y))
\end{equation}
for any $x,y\in \mathcal{A}(K)$, where $\mathcal{B}l_0^K$ is the localized classical Blanchfield form on the Alexander module of $K$ ~\cite[Proposition 3.6]{Lei3}~\cite[Theorem 4.7]{Lei1} (see also ~\cite[Section 5.2.2]{Cha2}), where
$$
\ov{\phi}:\mathbb{Q}(t)/\mathbb{Q}[t,t^{-1}]\longrightarrow \mathcal{K}\Lambda/\mathbb{Q}\Lambda S^{-1},
$$
is induced by $\phi$ (beware: $\ov{\phi}$ is not, in general, injective).

We will need the following, which relies crucially on the coefficient system factoring through $\Z$. Both parts of its conclusion are false without this assumption.

\begin{lem}\label{lem:non-singular} Suppose $K$ is a knot and $\phi:\pi_1(M_K)\to \Lambda$ is a PTFA coefficient system that factors nontrivially through $\Z$. Suppose $\mathcal{R}=(\mathbb{Z}\Lambda)S^{-1}$ is an Ore localization  where $S$ is closed under the natural involution on $\mathbb{Z}\Lambda$. Then the linking form $\mathcal{B}\ell^K_\mathcal{R}$ of Theorem~\ref{thm:blanchfieldexist} is a non-singular linking form on a finitely generated  module of homological dimension one (indeed the module has a square presentation matrix).
\end{lem}
\begin{proof} The existence is guaranteed by Theorem~\ref{thm:blanchfieldexist}. We have shown in the preceding paragraphs that the module is finitely generated and has homological dimension $1$. The definition of the linking form is the composition of 3 maps: Poincar\'{e} duality with $\R$ coefficients, the inverse of a Bockstein homomorphism, and the Kronecker evalation map
$$
H^1(M_K;\K/\R)\overset{\kappa}{\to}\overline{\Hom _\R(H_1(M_K;\R),\K/\R)}.
$$
It was shown in ~\cite[Theorem 2.13]{COT} that the first two are isomorphisms. Thus we need only show that $\kappa$ is an isomorphism. There is a universal coefficient spectral sequence ~\cite[Theorem 2.3]{L4}
$$
E^{p,q}_2\cong Ext_\R^q(H_p(M_K;\R),\K/\R)\Rightarrow H^{*}(M_K;\K/\R),
$$
with differential $d^r$ of degree $(1-r,r)$. Since we have shown that the modules $H_p(M_K;\R)$ have homological dimension one, the Ext terms vanish for $q>1$. It follows that the spectral sequence collapses and the usual universal coefficient sequence holds:
$$
0\to Ext_\R^1(H_0(M_K;\R),\K/\R)\to H^{1}(M_K;\K/\R)\overset{\kappa}\to \Hom _\R(H_1(M_K;\R),\K/\R)\to 0.
$$
Associated to the short exact sequence of $\R$-modules
$$
0\to \R\to \K \to \K/\R \to 0
$$
is a long exact sequence of $Ext_\R^*(H_1(M_K;\R),-)$, yielding
$$
\to Ext_\R^1(H_0(M_K;\R),\K)\to Ext_\R^1(H_0(M_K;\R),\K/\R)\to Ext_\R^2(H_0(M_K;\R),\R)\to
$$
where the last term is zero as observed previously. But $\K$ is a torsion-free, divisible module over the Ore domain $\R$ and hence is an injective $\R$-module ~\cite[Propositions 3.8, 7.8]{Ste}. Thus $Ext_\R^1(H_0(M_K;\R),\K)=0$ and so $\kappa$ is an isomorphism.

$H_1(M_K;\R)$ has a square presentation matrix simply because $H_1(M_K;\Q[t,t^{-1}])$ does.
\end{proof}

\end{subsection}

\begin{subsection}{Return to inductive step in proof of Proposition~\ref{prop:familyofmickeys}}\label{subsect:inductivefamilyofmick}

Recall that we have already established Proposition~\ref{prop:familyofmickeys} in the base case, $i=n$, of a (downward) induction on $i$. Now assume that $W_i$, for some $i$, $1\leq i\leq n$, has been shown to satisfy properties ($1$) and ($3$) of Proposition~\ref{prop:familyofmickeys}. We will derive some further important properties of $W_i$ that will enable us to analyze $W_{i-1}$.

By construction $\partial W_i=M_{i}^1$ and by property $(3)$,  $W_i$ is an $(n,\mathcal{P})$-bordism and hence an $(n-i+1,\mathcal{P})$-bordism since $n-i+1\leq n$. Let $\pi=\pi_1(W_i)$, $\Lambda=\pi/\pi^{(n-i+1)}_\mathcal{P}$ and let $\psi:\pi\to \Lambda$ be the canonical surjection. Apply Theorem~\ref{thm:selfannihil} to $(W_i, \psi)$ with $k=n-i+1$. Property ($1$) for $W_i$ ensures that $\psi$ restricted to $\pi_1(M_{i}^1)$ is non-trivial and factors through the abelianization. Let $S=S_{p_k}(\pi)\subset \Q[\pi_\mathcal{P}^{(k-1)}/\pi_\mathcal{P}^{(k)}]$ as in Definition~\ref{def:defderivedlocalp}, unless $k=1$ in which case let $S=S^*_{p_1}$  as in Definition~\ref{def:Sp}; and let $\mathcal{R}=(\Q\La) S^{-1}$. Hence by Theorem~\ref{thm:selfannihil} and ~\eqref{eq:decompmodule}, the kernel $P'$ of the composition
$$
\mathcal{A}(K^1_i) \otimes_{\mathbb{Q}[t,t^{-1}]}\mathcal{R} \overset{\cong}{\to} H_1(M_i^1;\mathcal{R})\overset{j_*}\to H_1(W_i;\mathcal{R}).
$$
satisfies $P'\subset (P')^\perp$ with respect to the Blanchfield form $\mathcal{B}\ell^{K^1_i}_\mathcal{R}$. Recall that $K^1_{i}=R^i_{\alpha_i}(K_{i-1}^1)$  is obtained from $R^i_{\alpha_i}$ by an infection along a circle $\alpha_i$. By hypothesis this circle is a  generator of $\mathcal{A}(R^1_{\alpha_i})$. We will show that
\begin{equation}\label{eq:alphanotinP}
\alpha_i\otimes 1 \notin P'.
\end{equation}
To see this recall that
$$
\mathcal{A}(K^1_{i})\cong\mathcal{A}(R^i_{\alpha_i})\cong \frac{\Q[t,t^{-1}]}{\langle p_k(t) \rangle},
$$
since $k=n-i+1$ (see hypotheses of Theorem~\ref{thm:robust}).  This cyclic module is non-trivial since $p_k$ is not a unit, Thus
$$
\mathcal{A}(K^1_i) \otimes_{\mathbb{Q}[t,t^{-1}]}\Q \La\cong \frac{\Q\La}{p_k(\mu)\Q\La},
$$
is a non-trivial cyclic module generated by the image of $\alpha_i\otimes 1$, where here $\mu=\mu_{K^1_i}$, or more properly $\phi(\mu_{K^1_i})$, the class in $\La$ represented by the meridian of $K^1_{i}$ in $M^1_i$. By property ($1$) of the inductive assumption, $\mu$ is an element of infinite order in $\La$. Then
$$
H_1(M_i^1;\mathcal{R})\cong\mathcal{A}(K^1_i) \otimes_{\mathbb{Q}[t,t^{-1}]}\mathcal{R}= \frac{\Q\La}{p_k(\mu)\Q\La}S_{p_k}^{-1}\cong \frac{\mathcal{R}}{p_k(\mu)\mathcal{R}},
$$
a cyclic $\mathcal{R}$-module generated by the image of $\alpha_i\otimes 1$. If the generator $\alpha_i\otimes 1$ were to lie in $P'$ then, since $P'\subset (P')^\perp$, it would follow that $\mathcal{B}\ell^{K^1_i}_\mathcal{R}$ were identically zero. But by Lemma~\ref{lem:non-singular}, this linking form is nonsingular. This could only happen if $H_1(M_i^1;\mathcal{R})$ were the zero module. However, by the first part of Theorem~\ref{thm:noStorsion}, if $k\neq 1$ then
$$
\frac{\Q\La}{p_k(\mu)\Q\La}\hookrightarrow \frac{\Q\La}{p_k(\mu)\Q\La}S_{p_k}^{-1}
$$
is a monomorphism. If $k=1$ then $\La\cong \Z$ and, by Proposition~\ref{prop:noS*torsion},
$$
\frac{\Q\La}{p_1(\mu)\Q\La}\hookrightarrow \frac{\Q\La}{p_1(\mu)\Q\La}(S_{p_1}^*)^{-1}
$$
is a monomorphism (here we need that $p_1(t)\doteq p(t^{-1})$, which holds since $p_1$ is, by hypothesis, the order of the generator of the classical Alexander module of a knot). In particular, in either case, if $H_1(M_i^1;\mathcal{R})$ were the zero module, then it would follow that
$$
\frac{\Q\La}{p_k(\mu)\Q\La}=0.
$$
This would force $p_k(\mu)$ to be a unit in $\Q\La$. If there were $x\in \Q\La$ such that $p_k(\mu)x=1$, then, using the fact that $\Q\La$ is a free $\Q[\Z]=\Q[\mu,\mu^{-1}]$-module (here we use that $\mu$ is of infinite order in $\La$), we could decompose over the cosets of $\Z$ to get $p_k(\mu)x_e=1$ for some $x_e\in \Q[\mu,\mu^{-1}]$. This is not possible since $p_k$ is the Alexander polynomial of a robust knot so in particular $p_k(t)$ is not a unit. Hence
\begin{equation}\label{eq:nontrivialH1}
H_1(M_i^1;\mathcal{R})\neq 0.
\end{equation}
This contradiction establishes ~\eqref{eq:alphanotinP}.

Now we translate the homological data of ~\eqref{eq:alphanotinP} into a statement in $\pi_1$, namely we claim:
\begin{equation}\label{eq:alphanontrivial}
\alpha_i\in \pi^{(k)} ~\text{but} ~\alpha_i\notin\pi_\mathcal{P}^{(k+1)}.
\end{equation}
To establish this claim, first note that, by property ($1$) of Proposition~\ref{prop:familyofmickeys} for $W_i$, we have
$$
\pi_1(M^1_{i})^{(1)}\subset \pi^{(n-i+1)}=\pi^{(k)}.
$$
Since $\alpha_i \in \pi_1(M^1_{i})^{(1)}$,  $\alpha_i\in \pi^{(k)}$, which establishes the first part of claim ~\eqref{eq:alphanontrivial}. Recall that $H_1(W_i;\mathbb{Q}\Lambda)$ is identifiable as the ordinary rational homology of the covering space of $W$ whose fundamental group is the kernel of $\phi:\pi\to \Lambda$. Since this kernel is precisely $\pi^{(k)}_\mathcal{P}$ where $k=n-i+1$, we have that
$$
H_1(W_i;\mathbb{Q}\Lambda)\cong (\pi^{(k)}_\mathcal{P}/[\pi^{(k)}_\mathcal{P},\pi^{(k)}_\mathcal{P}])\otimes_\mathbb{Z} \mathbb{Q},
$$
and so
$$
H_1(W_i;\mathcal{R})\cong (\pi^{(k)}_\mathcal{P}/[\pi^{(k)}_\mathcal{},\pi^{(k)}_\mathcal{P}])\otimes_{\mathbb{Z}\La}\R
$$

Now consider the commutative diagram below.  The vertical map $j$ is injective by the definition of the commutator series given in \eqref{eq:defserieslocSn}.

$$
\begin{diagram}\label{diagram-pirelatehomology}\dgARROWLENGTH=1.2em
\node{\pi_1(M^1_{i})^{(1)}}\arrow[2]{e,t}{j_*}\arrow{s,r}{\pi}\node[2]{\pi^{(k)}}\arrow{e,t}{\phi}\arrow{s,r}{} \node{\pi^{(k)}_\mathcal{P}/\pi^{(k+1)}_\mathcal{P}}\arrow{s,r}{j} \\
  \node{\mathcal{A}(K^1_{i})}\arrow{e,t}{id\otimes 1}\node{ H_1(M^1_{i};\mathcal{R})}\arrow{e,t}{j_*}\node{H_1(W_i;\R)}\arrow{e,t}
{\cong}\node{(\pi^{(k)}_\mathcal{P}/[\pi^{(k)}_\mathcal{P},\pi^{(k)}_\mathcal{P}])\otimes (\Q\Lambda)S^{-1}_{p_k}}
\end{diagram}
$$
Since, by ~\eqref{eq:alphanotinP}, $\alpha_i$ is not in the kernel of the composition in the bottom row, $\alpha_i$ is not in the kernel of the composition in the top row. Hence $\alpha_i\notin \pi^{(k+1)}_\mathcal{P}$. This establishes the second part of claim ~\eqref{eq:alphanontrivial}.

Now let $P$ be the kernel of the bottom composition of the diagram above. This composition is a map of $\Q[t,t^{-1}]$-modules so certainly $P$ is a submodule. It is a proper submodule by ~\eqref{eq:alphanotinP}. Since $\mathcal{A}(R^i_{\alpha_i})$ is, by choice, a cyclic module whose order is a product of two primes $\delta(t)\delta(t^{-1})$, it has (at most) three proper submodules, $P_0=0$, $P_+=\langle\delta \rangle$ and $P_-=\langle\delta(t^{-1}) \rangle$. As previously observed, each of these submodules is isotropic for the Blanchfield form.

This establishes:
\begin{itemize}
\item [\textbf{Fact $1$}:]
The kernel, $\tilde{P}=\pi^{-1}(P)$, of the composition in the top row of the diagram above is of the form $\pi^{-1}(P)$ for some \emph{proper} submodule $P\subset \mathcal{A}(K^1_{i})$ that satisfies $P\subset P^\perp$ with respect to the classical Blanchfield form.  Note that
$$
j_*(\tilde{P})\subset \pi_\mathcal{P}^{(k+1)}.
$$
\end{itemize}

\noindent Note that the inclusions $M^1_i\hookrightarrow E_i$, and $M_{R^i}\hookrightarrow E_i$ induce an isomorphism $\mathcal{A}(K^1_{i})\cong \mathcal{A}(R^i_{\alpha_i})$. With respect to this identification we can view $P\subset \mathcal{A}(R^i_{\alpha_i})$, and we claim further that
\begin{itemize}
\item [\textbf{Fact $2$}:] If $i\geq 2$, then the $P\subset \mathcal{A}(R^i_{\alpha_i})$ that occurs in Fact $1$ corresponds to a ribbon disk for $R^i_{\alpha_i}$, and hence may be assumed to be $P_+=\langle \delta (t)\rangle$ for specificity (see Definition~\ref{def:robust}).
\end{itemize}

\begin{rem}\label{rem:whichS_i} In deciding which ribbon disk exterior $S_i$ to attach to $M_{R^i}$ to form $W_{i-1}$ from $W_i$, we must use one guaranteed by Fact $2$.
\end{rem}

 We now establish Fact $2$.  By property ($3$) of Proposition~\ref{prop:familyofmickeys} for $W_i$, $W_i$ is a $(k+1,\mathcal{P})$-bordism for $M_{i}^1$ since $k+1=n-i+2\leq n$. Now consider the coefficient system
$$
\psi:\pi\to \G=\pi/\pi_\mathcal{P}^{(k+1)}.
$$
and its restriction to $\pi_1(M_{i}^1)$ that we call $\phi$. Noting that $\psi(\pi^{(k+1)}_\mathcal{P})=1$, apply Theorem~\ref{thm:sigvanishesbordism} to $(W_i,\psi)$ to get that
$$
\rho(M_{i}^1,\phi)=0.
$$
Moreover by property $(1)$ for $W_i$, $j_*(\pi_1(M_{i}^1)^{(2)})\subset \pi_1(W_i)^{(k+1)}$. Therefore $\phi$ factors nontrivially through $\pi_1(M^1_{i})/\pi_1(M^1_{i})^{(2)}$. Moreover the map $\phi$ is determined by the composition in the top row of the above diagram, whose kernel is precisely $\tilde{P}$. So in fact by property $1$ of Proposition~\ref{prop:rhoprops},
$$
0=\rho(M_{i}^1,\phi)=\rho\left(M_{i}^1,\pi_1(M^1_{i})\to \frac{\pi_1(M^1_{i})}{\pi_1(M^1_{i})^{(2)}\tilde{P}}\right),
$$
Recall that $M_i^1$ is the zero surgery on the knot $K^1_i$. The previous equation implies that the first-order signature of $K^1_{i}$ corresponding to $P$, $\rho(K^1_{i},\phi_P)$, is zero (see Definition~\ref{defn:firstordersignatures}). But $K^1_{i}$ is obtained from the ribbon knot $R^i_{\alpha_i}$ by an infection along $\alpha_i$ using the knot $K^1_{i-1}$. Therefore (since $\phi(\alpha_i)\neq 1$ by ~\eqref{eq:alphanotinP}), using elementary additivity results for $\rho$ invariants ~\cite[Lemma 2.3]{CHL3} (more details on this computation are given in ~\cite[Example 4.3]{CHL4}),
$$
0=\rho(K^1_{i},\phi_P)=\rho(R^i_{\alpha_i},\phi_P))+\epsilon\rho_0(K^1_{i-1}).
$$
where $\epsilon\in \{0,1\}$. Since $i-1\geq 1$, $K^1_{i-1}=R^{i-1}_{\alpha_{i-1}}(K^1_{i-2})$, which lies in $\mathcal{F}_{1}$ by Proposition~\ref{prop:operatorsact}. Therefore $K^1_{i-1}$ is algebraically slice and so $\rho_0(K^1_{i-1})=0$. Hence the first-order signature
$$
\rho(R^i_{\alpha_i},\phi_P))=0.
$$
By definition of a robust operator this implies that $P$ corresponds to a ribbon disk for $R^i$. This finishes the verification of Fact $2$.

\vspace{.5in}

Finally we can recall the construction of $W_{i-1}$ and set about to establish the properties of Proposition~\ref{prop:familyofmickeys} for $W_{i-1}$. Refer to Figure~\ref{fig:cobordism1}. Recall that if $i=1$ then $W_{i-1}=W_{i}\cup E$ whereas if $i>1$ then $W_{i-1}=W_{i}\cup E_i\cup S_i$, where $S_i$ is the exterior of a ribbon disk for $R^i_{\alpha_i}$. Here we specify that we shall choose $S_i$ to be $B^4-\Delta_i$ where $\Delta_i$ satisfies Fact $2$.

\vspace{.2in}

\textbf{Property (3) of Proposition~\ref{prop:familyofmickeys}}: \textbf{$W_{i-1}$ is an  $(n,\mathcal{P})$-bordism}.

\vspace{.2in}

We have already verified this for $W_i$. Since $W_{i-1}$ is created from $W_i$ by adding pieces with $H_2=0$ (in the case of $S_i$) or with $H_2(E_i)/H_2(\partial E_i)=0$, a short Mayer Vietoris argument shows that
$$
H_2(W_{i-1};\mathbb{Z})/j_*(H_2(\partial W_{i-1};\mathbb{Z}))\cong H_2(W_{i};\mathbb{Z})/j_*(H_2(\partial W_{i};\mathbb{Z})).
$$
It follows that the same surfaces can be used to show that $W_{i-1}$ is an $(n,\mathcal{P})$-bordism as were used to show $W_i$ is an $(n,\mathcal{P})$-bordism. If, for example, $L$ is such a surface and $\pi_1(L)\subset \pi_1(W_i)^{(n)}_\mathcal{P}$ then certainly
$\pi_1(L)\subset \pi_1(W_{i-1})^{(n)}_\mathcal{P}$ by the weak functoriality of the series.

This completes the verification of the property $(3)$ of Proposition~\ref{prop:familyofmickeys} for $W_{i-1}$.

\vspace{.2in}

\textbf{Property ($1$) of Proposition~\ref{prop:familyofmickeys} holds for $W_{i-1}$}:

\vspace{.2in}

Consider $M^1_{i-1}\subset\partial W_{i-1}$. Recall that $\pi_1(M^1_{i-1})$ is normally generated by the meridian, $\mu_{i-1}$, and, by Lemma~\ref{lem:mickeyfacts}, this is isotopic in $E_{i-1}$ to a push-off of $\alpha_i$ in $M^1_{i}=\partial W_i$. By ~\eqref{eq:alphanontrivial}
$$
\alpha_i\in \pi_1(W_i)^{(k)}\subset \pi_1(W_{i-1})^{(k)}
$$
Thus
$$
\pi_1(M^1_{i-1})\subset \pi_1(W_{i-1})^{(k)}
$$
where $k=n-i+1$. Thus establishes the first part of property $(1)$ for $W_{i-1}$. To prove the second part it would suffice to show that
$j_*(\alpha_i)$ is non-zero in $\pi_1(W_{i-1})^{(k)}/\pi_1(W_{i-1})^{(k+1)}_\mathcal{P}$. Equation ~\eqref{eq:alphanontrivial} provides precisely this, except with $\pi_1(W_i)$ instead of $\pi_1(W_{i-1})$. Therefore it suffices to show that inclusion induces an isomorphism
\begin{equation}\label{eq:isomorph}
\pi_1(W_{i})/\pi_1(W_{i})^{(k+1)}_\mathcal{P}\cong \pi_1(W_{i-1})/\pi_1(W_{i-1})^{(k+1)}_\mathcal{P}.
\end{equation}
By Proposition~\ref{prop:idempotency}, it now suffices to show that
\begin{equation}\label{eq:isomorph2}
\ker\left(\pi_1(W_i))\to \pi_1(W_{i-1})\right)\subset \pi_1(W_{i})^{(k+1)}_\mathcal{P}.
\end{equation}
The map $\pi_1(W_i))\to \pi_1(W_i\cup E_i)$ is a surjection whose kernel is the normal closure of the longitude $\ell$ of the copy of $S^3-K^1_{i-1}\subset M^1_{i}$ (by property $(1)$ of Lemma~\ref{lem:mickeyfacts}). The group $\pi_1(S^3-K^1_{i-1})$ is normally generated by the meridian of this copy of $K^1_{i-1}$. By the definition of infection, this meridian is identified to a push-off of the curve $\alpha_i$ and we have seen that $\alpha_i\in \pi_1(W_i)^{(k)}$. Thus $\ell\in \pi_1(W_{i})^{(k+1)}_\mathcal{P}$ as required. If $i=1$, $W_{i-1}=W_i\cup E_i$ so this establishes ~\ref{eq:isomorph}.

Now suppose $i>1$. Then the kernel of $\pi_1(W_i\cup E_i))\to \pi_1(W_i\cup E_i\cup S_i)=\pi_1(W_{i-1})$ is the normal closure of the kernel of $\pi_1(M_{R^i})\to \pi_1(S_i)$. The latter is of course contained in the commutator subgroup of $\pi_1(M_{R^i})$. Any element of $\pi_1(M_{R^i})$ is homotopic in $E_i$ to an element of $\pi_1(M^1_i)$. Thus any element of $\pi_1(M_{R^i})^{(2)}$ is equal to an element of $\pi_1(M^1_i)^{(2)}$. But by property ($1$) of Proposition~\ref{prop:familyofmickeys} for $W_i$ (or see our big diagram above)
$$
\pi_1(M^1_i)^{(2)}\subset \pi_1(W_i)^{(k+1)}\subset \pi_1(W_i)^{(k+1)}_\mathcal{P}.
$$
Therefore we may ignore elements in $\pi_1(M_{R^i})^{(2)}$ and so it suffices to consider a generator of the kernel of
$$
\mathcal{A}(R^i)\cong\frac{\pi_1(M_{R^i})^{(1)}}{\pi_1(M_{R^i})^{(2)}}\to \frac{\pi_1(S_i)^{(1)}}{\pi_1(S_i)^{(2)}}\cong \mathcal{A}(S_i),
$$
which, by Fact $2$ and Remark~\ref{rem:whichS_i}, is the cyclic module denoted $P$. But by Fact $1$, under the identification $\mathcal{A}(R^i)\cong \mathcal{A}(M^1_i)$ we see that a representative of $P$ lies in $\tilde{P}$ and
$$
\tilde{P}\subset \pi_1(W_i)^{(k+1)}_\mathcal{P}
$$
as required.  This completes the verification of ~\ref{eq:isomorph2} and hence that of property ($1$) for $W_{i-1}$.

\vspace{.2in}

\textbf{Property ($2$) of Proposition~\ref{prop:familyofmickeys} holds for $W_{0}$}:

\vspace{.2in}

The group $j_*(\pi_1(M_{R^1_{\alpha_1}}))$ is normally generated by a meridian of $R^1_{\alpha_i}$ which is isotopic in $E_1$ to the meridian of $K^1_1$ in $\pi_1(M_1^1)$. By property $1$ of Proposition~\ref{prop:familyofmickeys} for $i=1$,
$$
j_*(\pi_1(M_1^1))\cong \Z \subset \pi_1(W_1)^{(n-1)}/\pi_1(W_1)^{(n)}_\mathcal{P}.
$$
By ~\eqref{eq:isomorph} with $i=1$ (so $k=n$)
\begin{equation}\label{eq:isomorph3}
\pi_1(W_{1})/\pi_1(W_{1})^{(n+1)}_\mathcal{P}\cong \pi_1(W_{0})/\pi_1(W_{0})^{(n+1)}_\mathcal{P},
\end{equation}
so certainly
$$
\pi_1(W_1)^{(n-1)}/\pi_1(W_1)^{(n)}_\mathcal{P}\cong \pi_1(W_0)^{(n-1)}/\pi_1(W_0)^{(n)}_\mathcal{P}.
$$
Therefore
$$
j_*(\pi_1(M_{R^1_{\alpha_1}}))\cong j_*(\pi_1(M_1^1))\cong \Z \subset\pi_1(W_0)^{(n-1)}/\pi_1(W_0)^{(n)}_\mathcal{P}.
$$
This completes the verification of property ($1$) for $W_{0}$.

This concludes the inductive proof of Proposition~\ref{prop:familyofmickeys}. The proof of Theorem~\ref{thm:robust} is now complete.

\end{subsection}

\end{proof}

In fact, note that the proof of Theorem~\ref{thm:robust} transitioned quite quickly (and necessarily) into the category of $(n,\mathcal{P})$-bordisms. Thus we see that we can prove a stronger version of Theorem~\ref{thm:robust} and Corollary~\ref{cor:infgen}. The point is that in our construction of the manifold $W_0$ of Proposition~\ref{prop:familyofmickeys}, we made very little use of the fact that the submanifold $V$ was an $(n.5,\mathcal{P})$-solution as opposed to an $(n.5,\mathcal{P})$-null-bordism.
Certainly what we needed primarily was that
$$
\sigma^{(2)}(V,\phi)-\sigma(V)=0,
$$
for any $\phi:\pi_1(V)\to\pi_1(V)/\pi_1(V)^{(n+1)}_\mathcal{P}\to \G$ ($\G$ PTFA). But Theorem~\ref{thm:sigvanishesbordism} can (almost) be used to establish this, in place of Theorem~\ref{thm:generalsignaturesobstruct}. We also needed some non-triviality for the map $\pi_1(\partial V)\to \pi_1(V)$ to establish property $(1)$ of Proposition~\ref{prop:familyofmickeys} in the base case of the induction where $i=n$ (the first paragraph of the proof). Since, in a null-bordism, $\partial V\to V$ does not necessarily induce a monomorphism on $H_1$, we must add some additional weak non-triviality condition. This results in a shift of the ``exponents'' in all the arguments in the proof.

\begin{thm}\label{thm:robustgeneralized} Let $\mathcal{P}$ be the unrestricted derived series localized at $\mathcal{P}$. No nontrivial linear combination of the knots of Theorem~\ref{thm:robust} and Corollary~\ref{cor:infgen} is $(n+1+r,\mathcal{P})$-null-bordant via $V$ where
$$
j_*(\pi_1(\partial V))\cong \Z \subset \pi_1(V)^{(r)}/ \pi_1(V)^{(r+1)}_\mathcal{P};
$$
and no such is $(n+.5+r,\mathcal{P})$-null-bordant via $V$ if, in addition,
$$
\sigma^{(2)}(V,\phi)-\sigma(V)=0
$$
for any PTFA coefficient system factoring through $\pi_1(V)/\pi_1(V)^{(n+1+r)}_\mathcal{P}$.
\end{thm}

Note that an actual ($n.5,\mathcal{P}$)-solution satisfies the hypotheses with $k=0$. The more general situation requires us to prove Proposition~\ref{prop:familyofmickeys} with all occurrences of $n$ replaced by $n+r$. Otherwise the proof is identical. The additional signature condition is necessary only because the precise generalization of Theorem~\ref{thm:generalsignaturesobstruct} to $(n.5,\mathcal{P})$-null-bordisms has not appeared in the literature; rather Theorem~\ref{thm:sigvanishesbordism} requires an $(n+1,\mathcal{P})$-null-bordism.

\section{Evidence for the Injectivity of Robust Doubling Operators}\label{sec:injectivity}

For the following theorem, suppose $\mathbb{P}_{n-1}$ is the index set for a collection, $\{\mathcal{Q}_i~|~i\in \mathbb{P}_{n-1}\}$, of $n-1$-tuples $\mathcal{Q}_i=(q_{i_{n-1}}(t),\dots,q_{i_1}(t))$ of non-zero, non-unit polynomials such that, for any $i\neq j \in \mathbb{P}_{n-1}$, at least one coordinate of $\mathcal{Q}_i$ is strongly coprime to the corresponding coordinate of $\mathcal{Q}_j$. This is a slightly more stringent condition than Definition~\ref{def:stronglycoprime}. This is ncessary since, in the following result, the Alexander polynomial of the operator $R_\alpha$ will play the role of the ``first'' polynomial in $\mathcal{P}$.

\begin{thm}\label{thm:injectivity} Suppose $R_\alpha$ is a robust operator. Then, for any $n\geq 1$, the composition
$$
\mathcal{C}\overset{R_\alpha}{\longrightarrow}\mathcal{C}\to \frac{\mathcal{C}}{\mathcal{F}_{n.5}}
$$
is injective on any subgroup generated by
$$
\mathcal{J}_{n-1}=\{\mathcal{R}_i(K^{i,j})~|i\in \mathbb{P}_{n-1}, ~1\leq j< \infty\}
$$
where the $\mathcal{R}_i(K^{i,j})$ are as in Theorem~\ref{thm:fractal}, so  $R_\alpha$ is injective on subgroups
$$
\bigoplus_{\substack{\mathbb{P}_{n-1}}}\Z^\infty\subset \mathcal{F}_{n-1}\subset\mathcal{C}
$$
as in ~\eqref{eq:bigsubgroup} and Corollary~\ref{cor:bigsubgroup}. Thus, by letting $n$ vary, it follows that $R_\alpha:\mathcal{C}\to \mathcal{C}$ is injective on the subgroup
\begin{equation}\label{eq:big2}
\bigoplus_{\substack{n}}\bigoplus_{\substack{\mathbb{P}_n}}\Z^\infty\subset \mathcal{C}
\end{equation}
 as in~\eqref{eq:bigsubgroup}. Moreover, if $R_\alpha$ and $R'_\beta$ are robust operators for which the classical Alexander polynomials of $R$ and $R'$ are coprime, then $R_\alpha$ and $R'_\beta$ have distinct images (intersect only in $\{0\}$), when restricted to the subgroups ~\eqref{eq:big2}.
\end{thm}

\begin{proof} We abbreviate $R_\alpha$ by $R$. Suppose
$$
R(J)= R(J') ~\text{in}~ \frac{\mathcal{F}_n}{\mathcal{F}_{n.5}}
$$
where $J=\#m_{ij}\mathcal{R}_i(K^{i,j}), J'=\#n_{ij}\mathcal{R}_i(K^{i,j})$ are distinct,  $\mathcal{R}_i(K^{i,j})\in \mathcal{J}_{n-1}$ where the sum runs over all $1\leq j< \infty$, and all $i\in \mathbb{P}_{n-1}$, but where all but finitely many of the integers $m_{ij}$ and $n_{ij}$ are zero. If $n=1$ then the index $i$ takes on only one value which we call $i=1$. We shall derive a contradiction. The strategy of the proof is identical to that of Theorem~\ref{thm:robust}. Here is a sketch. Since
$$
R(J)~\# -(R(J')) \in \mathcal{F}_{n.5},
$$
it follows that
$$
R(J)~\# -(R(J')) \in \mathcal{F}_{n.5}^\mathcal{P},
$$
for \textbf{any} $n$-tuple $\mathcal{P}$. We must choose $\mathcal{P}$ wisely so that this is false. Since $J$ and $J'$ are distinct, by relabeling both $i$ and $j$, we may assume, without loss of generality, that $m_{11}>n_{11}\geq 0$. In particular we may assume that the knot $\mathcal{R}_1(K^{1,1})$ occurs non-trivially as a summand of $J$. Recall that $\mathcal{R}_i=R^{i,{n-1}}_{\alpha^i_{n-1}}\circ\dots\circ R^{i,1}_{\alpha_1^i}$, a composition of $n-1$ robust doubling operators, so $\mathcal{R}_1=R^{1,{n-1}}\dots\circ\dots\circ R^{1,1}$. Similar to the proof of Theorem~\ref{thm:robust}, we will construct a $4$-manifold $W_0$ whose boundary is $m_{11}$ copies of the zero surgery on $R^{1,1}$ and $m_{11}$ copies of the zero surgery on $K^{1,1}$.  Let $(p_2(t),...,p_n(t))$ denote the sequence of orders of the Alexander modules of $(R^{1,{n-1}},..., R^{1,1})$. Let $p_1(t)$ be the Alexander polynomial of $R$. Let $\mathcal{P}=(p_1(t),p_2(t),...,p_{n}(t))$. Let $V$ be an $(n.5,\mathcal{P})$-solution for $R(J)~\# -R(J')$. Let $C$ be the standard cobordism from $\partial V$ to the disjoint union of the zero surgeries on $R(J)$ and $-R(J')$. Attach to $V\cup C$, along $M_{-R(J')}$, a manifold $E'$ from Lemma~\ref{lem:mickeyfacts} where $\partial E'= M_{R(J')}\coprod M_{-R} \coprod M_{-J'}$. To this attach along $M_R \subset\partial E'$, $S'$, the complement of a slice disk for $R$. Then attach the standard cobordism $C'$ from $M_{-J'}$ to the disjoint union of copies of the zero surgeries on $\pm\mathcal{R}_i(K^{i,j})$  according to the decomposition $-J'=\#-n_{ij}\mathcal{R}_i(K^{i,j})$. For each such new boundary component attach the $(n-1)$ solution $\pm Z(i,j)$ from Corollary~\ref{cor:Z-cap}. The resulting manifold, denoted $W_n$ and shown in Figure~\ref{fig:W_n},  has boundary $M_{R(J)}$.
\begin{figure}[htbp]
\setlength{\unitlength}{1pt}
\begin{picture}(200,250)
\put(-60,0){\includegraphics{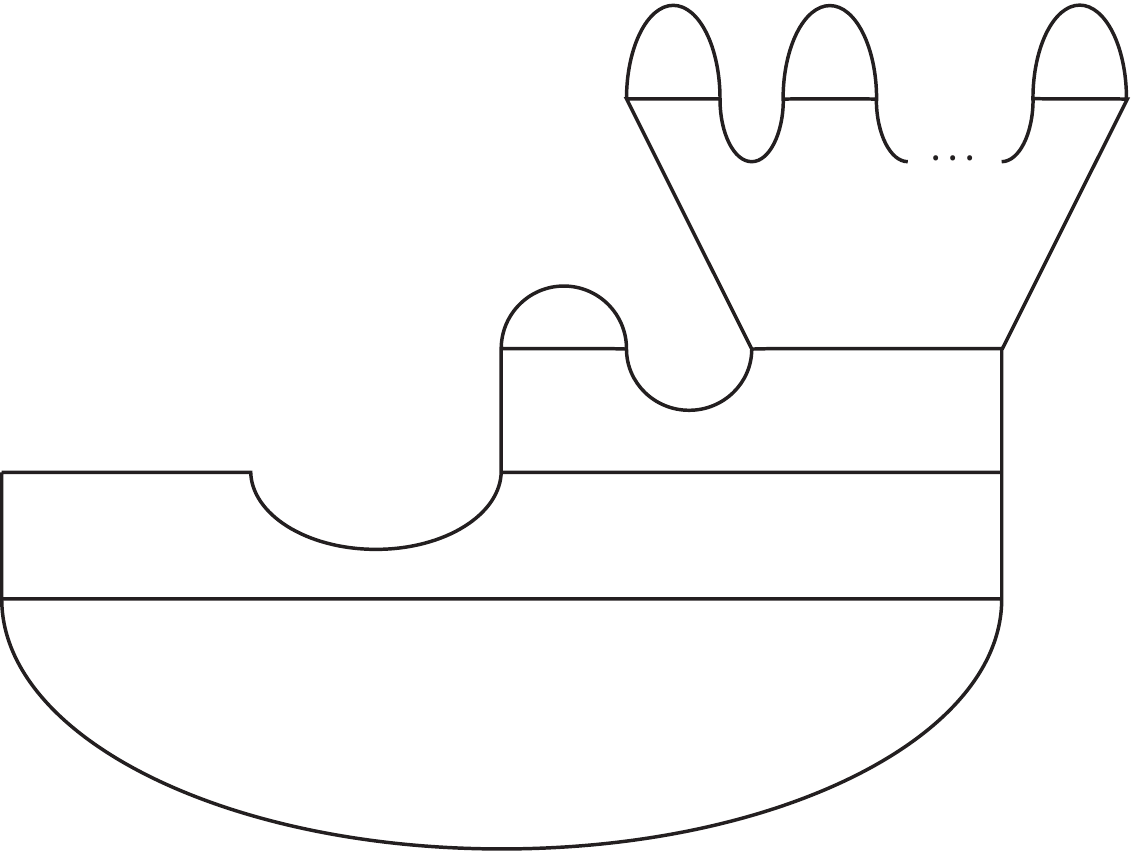}}
\put(85,82){$C$}
\put(83,29){$V$}
\put(148,115){$E'$}
\put(98,150){$S'$}
\put(-30,99){$M_{R(J)}$}
\put(178,99){$-M_{R(J')}$}
\put(178,135){$-M_{J'}$}
\put(91,135){$-M_{R}$}
\put(185,165){$C'$}
\put(188,225){$\leftarrow Z(i,j)\rightarrow$}
\end{picture}
\caption{$W_{n}$}\label{fig:W_n}
\end{figure}

Since $\pi_1(Z(i,j))\subset \pi_1(W_n)^{(1)}$, these $(n-1)$-solutions are effective ($n$)-solutions. In particular $W_n$ is an $(n,\mathcal{P})$-null-bordism for $M_{R(J)}$. Consider the coefficient system
$$
\psi: \pi_1(W_n)\to \pi_1(W_n)/\pi_1(W_n)^{(1)}_\mathcal{P}\cong \Z.
$$
By Theorem~\ref{thm:selfannihil}, with $k=1$ and $\mathcal{R}=\Q[t,t^{-1}](S^*_{p_1})^{-1}$, the kernel of the inclusion
$$
\mathcal{A}(R)(S^*_{p_1})^{-1}\cong\mathcal{A}(R(J))(S^*_{p_1})^{-1}\to \mathcal{A}(W_n)(S^*_{p_1})^{-1}
$$
is self-annihilating (in fact known to be a Lagrangian in this case). Therefore the curve $\alpha$, being a generator of $\mathcal{A}(R(J))$ and being $p_1(t)$ torsion, cannot lie in $\pi_1(W_n)^{(2)}_\mathcal{P}$. Therefore we have verified the analogue of the case $i=n$ (the base case) of Proposition~\ref{prop:familyofmickeys}.

The rest of the proof is very similar to that of Proposition~\ref{prop:familyofmickeys} and the proof of Theorem~\ref{thm:robust}, but not identical because in this case $J$, which is analogous to the knot $K_{n-1}$ in that proof, is a connected sum of knots in the images of iterated operators rather than being a single such knot. In this case to form $W_{n-1}$  first glue an $E$ to $W_n$ so that the boundary is $M_R\coprod M_J$. Consider the coefficient system
$$
\phi:\pi_1(W_n\cup E)\to \pi_1(W_n\cup E)/\pi_1(W_n\cup E)^{(2)}.
$$
Since $n\geq 1$, one easily sees that $\rho(\partial(W_n\cup E), \phi)=0$. But $\rho_0(J)=0$ if $n-1\neq 0$. Thus one shows that the first-order signature of $R$ corresponding to $P$ is zero, where $P$ is the kernel of the inclusion $\mathcal{A}(R)\cong\mathcal{A}(R(J))\to \mathcal{A}(W_n)$. Since $R$ is robust, $P$ corresponds to a  ribbon disk. Add to $E$ along $M_R$, the corresponding ribbon disk exterior $S$. Also add the standard cobordism $C''$, from $M_J$ to the disjoint union of the zero surgeries of copies of $\pm\mathcal{R}_i(K^{i,j})$  according to the decomposition $J=\#m_{ij}\mathcal{R}_i(K^{i,j})$. Then to each zero surgery on $\mathcal{R}_i(K^{i,j}$ for $(i,j)\neq (1,1)$, adjoin a $\mathcal{Z}-cap$ $Z(i,j)$ from Corollary~\ref{cor:Z-cap}. The resulting manifold, denoted $W_{n-1}$ has boundary equal to $m_{11}$ copies of the zero surgery on $\mathcal{R}_1(K^{1,1})$.
\begin{figure}[htbp]
\setlength{\unitlength}{1pt}
\begin{picture}(200,250)
\put(-83,0){\includegraphics{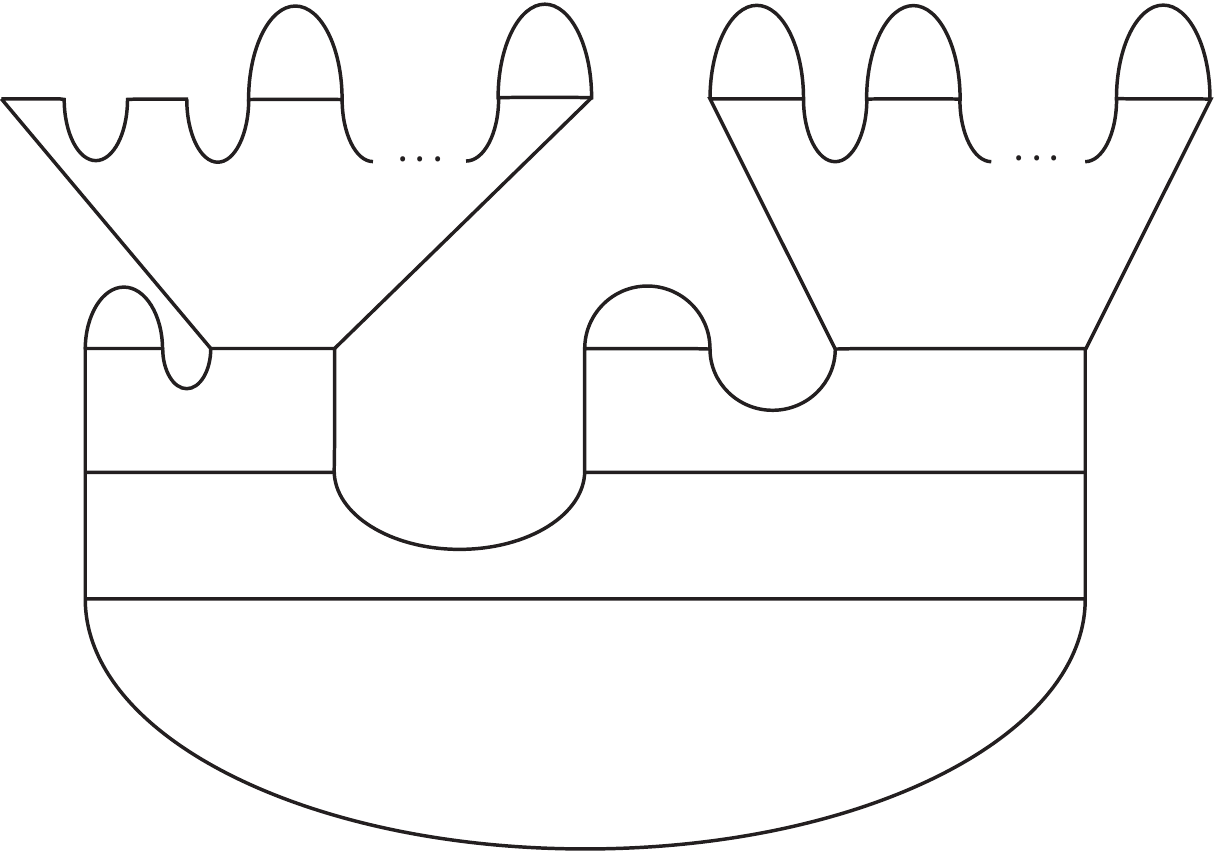}}
\put(85,82){$C$}
\put(83,29){$V$}
\put(148,115){$E'$}
\put(98,150){$S'$}
\put(-30,99){$M_{R(J)}$}
\put(-32,122){$E$}
\put(178,99){$-M_{R(J')}$}
\put(178,135){$-M_{J'}$}
\put(91,135){$-M_{R}$}
\put(-55,136){$M_{R}$}
\put(-82,223){$M_{R^{1,1}}$}
\put(-46,223){$M_{R^{1,1}}$}
\put(-15,135){$M_{J}$}
\put(-52,150){$S$}
\put(185,165){$C'$}
\put(-15,165){$C''$}
\put(185,225){$\leftarrow Z(i,j)\rightarrow$}
\put(11,225){$\leftarrow Z(i,j)\rightarrow$}
\end{picture}
\caption{$W_{n-1}$}\label{fig:W_{n-1}}
\end{figure}

This is an $(n,\mathcal{P})$-null bordism for its boundary. Now we continue to construct $4$-manifolds $W_{n-2},\dots, W_0$ just as in Proposition~\ref{prop:familyofmickeys}, except that here we have $m_{11}$ boundary components rather than $1$. We arrive finally at $W_0$ whose boundary is $m_{11}$ copies of the zero surgery on $R^{1,1}$ and $m_{11}$ copies of the zero surgery on $K^{1,1}$. Then we consider the coefficient system
$$
\phi:\pi_1(W_0)\to \pi_1(W_0)/(\pi_1(W_0))^{n+1}_\mathcal{P}.
$$
One shows that on the one hand, $\rho(\partial W_0,\phi)$ is equal to the sum of $m_{11}\rho_0(K^{1,1})$ and $m_{11}$ (possibly different) first-order signatures of $R^{1,1}$. Then we assert that
$$
\sigma^{(2)}(W_0,\phi)-\sigma(W_0)
$$
is equal to the sum of $c_j\rho_0(K^{1,j})$ where $c_1$ is less in absolute value than $m_{1,1}$. These two assertions contradict the choice of the $K^{i,j}$ (see the hypotheses of Theorem~\ref{thm:fractal}).

Clearly only the $\mathcal{Z}$-caps contribute to the signature of $W_0$. Those that correspond to values of $i\neq 1$ can be ignored since they are effective $(n+1,\mathcal{P})$-solutions, since each $i\neq 1$ corresponds to an $n-1$-tuple of orders $\mathcal{Q}_i$ that is strongly coprime to $(p_2,...,p_n)$. The $\mathcal{Z}$-caps with $i=1$ have signatures equal to zero or $\rho_0(K^{1,j})$ by Corollary~\ref{cor:Z-cap}. Therefore $|c_1|$ is less than or equal to $n_{11}$ which was less than $m_{11}$.

This concludes the proof that $R_\alpha$ is injective on the claimed subgroup, which is the first claim of the theorem.

For the second claim of the theorem, merely note that $\mathcal{J}_{n-1}$ is precisely the collection of ~\eqref{eq:bigsubgroup2} (for the case $n-1$).

For the third statement, to prove that $R_\alpha$ is injective on the subgroup ~\eqref{eq:big2}, suppose not and let $N$ be the minimum value of $n$ that occurs in the putative counterexample. Then look modulo $\mathcal{F}_{N.5}$. Since then all terms where $n>N$ then can be ignored, this contradicts the first part of the theorem.

For the final statement of the theorem, suppose $R_\alpha(J)=R'_\beta(J'')$. Thus
$$
R_\alpha(J)\#~-R'_\beta(J'')\in \mathcal{F}_{n.5}.
$$
We may assume that $J,J''$ lie in the subgroup generated by $\mathcal{J}_{n-1}$ for some $n\geq 1$.
As in the proof of the first part of the theorem, we may assume that some knot $\mathcal{R}_1(K^{1,1})$ occurs non-trivially as a summand of $J$. Recall that $\mathcal{R}_i=R^{i,n-1}_{\alpha_1^{n-1}}\circ\dots\circ R^{i,1}_{\alpha_1^i}$, a composition of $n-1$ robust doubling operators, so $\mathcal{R}_1=R^{1,{n-1}}\dots\circ\dots\circ R^{1,1}$. Let $p_1(t)$ be the Alexander polynomial of $R_\alpha$ (i.e. of $R$), and let $(p_2(t),...,p_n(t))$ denote the sequence of orders of the Alexander modules of $(R^{1,{n-1}},...,R^{1,1})$. Let $\mathcal{P}=(p_1(t),p_2(t),...,p_{n}(t))$. Let $(q_2,\dots,q_n)$ be the sequence of Alexander polynomials of the operators involved in some constituent knot of $J''$. Since the Alexander polynomial, $q_1$, of $R'$ is coprime to $p_1(t)$, the sequence $(q_1,\dots,q_n)$ is strongly coprime to $\mathcal{P}$. Thus
$$
-R'_\beta(J'')\in \mathcal{F}_{n+1}^\mathcal{P}
$$
by (a slight generalization of) Theorem~\ref{thm:qliesinpnplusone}, so we conclude
$$
R_\alpha(J)\in \mathcal{F}_{n.5}^\mathcal{P}.
$$
This implies that $J=0$ by the first part of the theorem. It follows also then that $J"=0$.
\end{proof}

\section{Another application: Cochran-Orr-Teichner knots are distinct from Cochran-Harvey-Leidy knots}\label{sec:COTknots}

In this section we show that the knots we have been discussing in this paper are not sufficient to generate $\mathcal{C}$. In fact, we show that almost none of the types of knots considered in the early papers of Cochran-Orr-Teichner, Cochran-Teichner, Kim and Friedl ~\cite{COT}\cite{COT2}\cite{CT}\cite{CK}\cite{Ki2} are even concordant to any of the knots we have considered in this paper (the exceptions being the types of knots generating $\mathcal{F}_1/\mathcal{F}_{1.5}$ which are common to both and were first considered by Casson-Gordon, Gilmer, Litherland and Livingston). On a related topic, we remark that S. Kim and T. Kim have announced a proof that the COT knot given in from ~\cite[Section 6]{COT}) is not concordant to any genus one knot. The class of CHL knots is not restricted to genus one knots.

A \textbf{COT knot} (at level $n$) is one that is obtained from a fixed slice knot $R$ (with the degree of its Alexander polynomial at least $4$) by infecting along a collection of circles $\{\alpha_1,...,\alpha_k\}$ lying in $\pi_1(S^3-R)^{(n)}$ using Arf invariant zero knots $\{K_1,...,K_k\}$. Thus
\begin{equation}\label{eq:COT}
J=R_{\{\alpha_1,...,\alpha_k\}}(K_1,...,K_k).
\end{equation}
An application of Proposition~\ref{prop:operatorsact} recovers the (previously known) fact that each such  knot $J$ is necessarily  an element of $\mathcal{F}_{n}$ ~\cite[Prop.3.1]{COT2}\cite[Prop.5.11]{CK}. In order to prove non-triviality up to concordance one also assumes that each $\rho_0(K_i)$, the average classical signature, is greater than a fixed positive constant (the Cheeger-Gromov constant of $M_R$) and that the $\{\alpha_i\}$ are carefully chosen to satisfy ~\cite[Theorem 5.13]{CK} (see also ~\cite[Theorem 4.3]{CT}). Under these conditions, it was shown in ~\cite[Theorem 5.14]{CK} (improving on ~\cite[Theorem 4.2]{CT}) that $J\notin \mathcal{F}_{n.5}$. These were the examples used to prove that $\mathcal{F}_{n}/\mathcal{F}_{n.5}$ has positive rank for $n>2$.

We claim that COT knots are even more robust than claimed by previous authors, for in fact:

\begin{prop}\label{prop:COTrobust} If $J\in \mathcal{F}_{n}$ is a COT knot (at level $n$) then
$$
J\notin \mathcal{F}_{n.5}^{cot}.
$$
where the latter is  with respect to the COT series (at level $n+1$) as given in Example~\ref{ex:commseriesCOT}.
\end{prop}
The proof is postponed.
By contrast a CHL knot $J$ (at level $n$) is defined to be one obtained as the result of applying $n$ iterated doubling operators to a knot $K\in \mathcal{F}_{0}$,
$$
J=R^n_{\alpha_n}\circ...\circ R^1_{\alpha_1}(K).
$$
Here, as usual, the $R^i$ are slice knots and the $\alpha_i$ are unknotted circles in $S^3-R_i$ that have zero linking number with $R_i$. Such knots $J$ are known to be ($n$)-solvable (for example apply Proposition~\ref{prop:operatorsact}). As we have seen, certain conditions must be imposed on $\{R_i, \alpha_i,K\}$ to ensure that $J\notin \mathcal{F}_{n.5}$ (for example see Theorem~\ref{thm:robust}) , but these are not relevant to our present discussion, for we assert that in any case:

\begin{prop}\label{prop:CHLnotrobust} If $J$ is concordant to a CHL knot at level $n$ for some $n> 1$ then
$$
J\in \mathcal{F}_{n.5}^{cot}.
$$
\end{prop}
\begin{cor}\label{cor:COTneqCHL} For $n>1$ no COT knot is concordant to any CHL knot.
\end{cor}
\begin{cor}\label{cor:theresmore} For $n\geq 2$ the cokernel of any embedding from Equation~\eqref{eq:bigsubgroup2}
$$
\bigoplus_{\substack{\mathbb{P}_n}}\Z^\infty\subset \frac{\mathcal{F}_n}{\mathcal{F}_{n.5}}
$$
is infinite.
\end{cor}

\begin{proof}[Proof of Proposition~\ref{prop:COTrobust}] Suppose $J$ is a COT knot and suppose $J\in \mathcal{F}_{n.5}^{cot}$.  Then $M_J$ bounds a $4$-manifold $W$ as in Definition~\ref{def:Gnsolvable}. We shall arrive at a contradiction. Letting $G=\pi_1(W)$ and

$$
\phi:\pi_1(M_J)\overset{j_*}{\to}G\to G/\gnp_{cot},
$$
we can apply Theorem~\ref{thm:generalsignaturesobstruct} to conclude that
$$
\rho(M_J,\phi)=0.
$$
But it follows from the additivity results of ~\cite[Prop.3.2]{COT2}\cite[Lemma 5.12]{CK} that
$$
\rho(M_J,\phi)=\rho(M_R,\phi_R)+\sum_{i=1}^k\epsilon_i\rho_0(K_i)
$$
where $\epsilon_i=0$ or $1$ according as $\phi(\alpha_i)=e$ or not. If any $\epsilon_i=1$, this is clearly a contradiction since, by choice, every $\rho_0(K_i)>|\rho(M_R,\phi_R)|$. Since $W$ is an ($n$)-solution, by ~\cite[Theorem 5.13]{CK}, by choice, for at least one $i$, $j_*(\alpha_i)\notin \gnp_r$. But in fact, the proof proves more. It shows the stronger fact that
$$
j_*(\alpha_i)\notin G^{(n+1)}_{cot},
$$
which implies $\phi(\alpha_i)\neq e$ and so finishes our proof by contradiction. That the proof shows the stronger fact is seen as follows. One first observes that
$$
\frac{\gn_r}{[\gn_r ,\gn_r ]}\cong H_1(W;\Z[G/\gn_r]),
$$
so
$$
\frac{\gn_r}{[\gn_r ,\gn_r ]}S_{n,cot}^{-1}\cong H_1(W;\Z[G/\gn_r]S_{n,cot}^{-1}).
$$
Now observe that, by definition, $G_{cot}^{(n+1)}$ is the kernel of the map in Equation~(\ref{eq:defCOTseries}), so $G^{(n)}_r/G_{cot}^{(n+1)}$ embeds in the codomain of the map in Equation~(\ref{eq:defCOTseries}). Combining these two facts we have that
$$
G^{(n)}_r/G_{cot}^{(n+1)}\hookrightarrow H_1(W;\Z[G/\gn_r]S_{n,cot}^{-1}).
$$
The ring $\Z[G/\gn_r]S_{n,cot}^{-1}$ is the twisted Laurent polynomial ring and (noncommutative) PID $\mathbb{K}_n[t,t^{-1}]$ that was used by Cochran-Orr-Teichner  and others. Therefore we see that to establish that $j_*(\alpha_i)\notin G^{(n+1)}_{cot}$ it is sufficient (and necessary) to show that $j_*(\alpha_i)\neq 0$ in $H_1(W;\mathbb{K}_n[t,t^{-1}])$. But in fact this is precisely what was shown by Cochran-Kim in the proof of their Theorem 5.13 (see ~\cite[Theorems 3.8, 6.4 and page 1440]{CK}). Thus we have a contradiction.
\end{proof}

\begin{proof}[Proof of Proposition~\ref{prop:CHLnotrobust}]  Suppose $J\in \mathcal{F}_{n}$ is concordant to the CHL knot $K=R^n_{\alpha_n}\circ...\circ R^1_{\alpha_1}(K_0)$ for $K_0\in  \mathcal{F}_{0}$ for some $n>1$. Then $M_K$ bounds the special $4$-manifold $Z$ as constructed in the first paragraph of the proof of Theorem~\ref{thm:qliesinpnplusone}. Let $K_1=R^1_{\alpha_1}(K_0)$,..., $K_{i}=R^i_{\alpha_1}(K_{i-1})$ and $K_n=K$.

Since $J$ is concordant to $K$,  $M_J$ is homology cobordant to $M_K$ via a $4$-manifold $C$.  Let $W=Z\cup C$ so that $\partial W=M_J$.  We claim that $J\in \mathcal{F}_{n+1}^{cot}$ via $W$, and hence $J\in \mathcal{F}_{n.5}^{cot}$. The proof of this fact is very similar to the proof of Proposition~\ref{prop:operatorsact} and to the proof of Theorem~\ref{thm:qliesinpnplusone}.

First, as in those  proofs, a Mayer-Vietoris sequence implies that $H_2(W)\cong H_2(V)$ where $K_{0}\in \mathcal{F}_{0}$ via $V$ and $\pi_1(V)\cong\Z$. Thus $H_2(V)$ has a basis of connected compact oriented surfaces, $\{L_j,D_j|1\leq j\leq r_i$, satisfying the conditions of Definition~\ref{def:Gnsolvable}.
We claim that
\begin{equation}\label{eq:natural2}
\pi_1(V)\subset \pi_1(W)^{(n+1)}_{cot}.
\end{equation}
Assuming this for the moment it then would follow from Proposition~\ref{prop:commseriesprops} that
$$
\pi_1(L_j)\subset \pi_1(V)\subset \pi_1(W)^{(n+1)}_{cot},
$$
and similarly for $\pi_1(D_j)$. This would then complete the verification that $J\in\mathcal{F}^{(n+1)}_{cot}$ via $W$.

In the rest of the proof we establish claim ~(\ref{eq:natural2}). Since $\mu_0$ generates $\pi_1(V)$,  we need only show that $\mu_0\in \pi_1(W)^{(n+1)}_{cot}$. Let  $G=\pi_1(W)$. First we show
$$
\mu_0\in G^{(n)}\equiv G^{(n)}_{cot}.
$$
This was already established in Lemma~\ref{lem:mui}. Let $\mu_1$ denote the meridian of $K_1$ in $\pi_1(M_{K_1})$. Then Lemma~\ref{lem:mui} also shows that $\mu_1\in G^{(n-1)}$.
Now we seek to show that $\mu_0\in G^{(n+1)}_{cot}$. From Equation~\ref{eq:defCOTseries}, we see that we need to establish that $\mu_0$ represents $S_{n,cot}$-torsion in the module
$$
\frac{\gn_r}{[\gn_r ,\gn_r ]}
$$
where  $S_{n,cot}=\Z[G^{(1)}/\gn_r]-\{0\}$.  Since $\mu_0$ is identified to $\alpha_1\subset M_{K_1}$, it suffices to show that $\alpha_1$ represents $S_{n,cot}$-torsion. Let $\Delta(t)$ be the Alexander polynomial of $K_1$ (which is the same as the Alexander polynomial of $R^1$). Then $\Delta(t)$ annihilates $\alpha_1$ in the Alexander module of $K_1$. This can be interpreted in term of the fundamental group of $M_{K_1}$ as follows  ~\cite[Section 7D]{R}. If $\Delta(t)=\sum m_it^i$ then we have
$$
\prod \mu_1^{-i}\alpha_1^{m_i}\mu_1^{i}\in \pi_1(M_{K_1})^{(2)}\subset [\gn_r,\gn_r]
$$
since we have shown above that $\pi_1(M_{K_1})=\langle\mu_1\rangle\subset G^{(n-1)}\subset G^{(n-1)}_r$. Therefore $\Delta(\mu_1)$ annihilates $\alpha_1$ in the module
$$
\frac{\gn_r}{[\gn_r ,\gn_r ]}.
$$
But since $\mu_1\in G^{(n-1)}$ and $n\geq 2$, $\mu_1\in G^{(1)}$. Thus $\Delta(\mu_1)\in \Z[G^{(1)}/\gn_r]$. Moreover $\Delta(\mu_1)\neq 0$ since $\Delta(1)=\pm 1$ ($\Delta$ is the Alexander polynomial of a knot in $S^3$). Hence $\Delta(\mu_1)\in S_{n,cot}$.

This concludes the verification of ~(\ref{eq:natural2}).
\end{proof}

\bibliographystyle{plain}
\bibliography{mybib7mathscinet,mybib5mathscinet}
\end{document}